\documentclass{amsart}

\usepackage{graphicx,epsfig}

\usepackage{amsmath,amssymb,latexsym, amsfonts, amscd, amsthm, xy}

\usepackage{url}

\usepackage{hyperref}

\usepackage{mathrsfs}

\usepackage{tikz}

\usepackage[latin1]{inputenc}

\usepackage{tikz, tikz-cd}

\usepackage{enumitem, hyperref}
\makeatletter
\def\namedlabel#1#2{\begingroup
    #2%
    \def\@currentlabel{#2}%
    \phantomsection\label{#1}\endgroup
}
\makeatother

\theoremstyle{plain}
\newtheorem{theorem}{Theorem}[section]

\newtheorem{proposition}[theorem]{Proposition}
\newtheorem{corollary}[theorem]{Corollary}

\newtheorem{lemma}[theorem]{Lemma}

\theoremstyle{definition}
\newtheorem{definition}[theorem]{Definition}
\newtheorem{remark}[theorem]{Remark}
\newtheorem{example}[theorem]{Example}

\def\bC{\mathbb{C}}
\def\bF{\mathbb{F}}

\def\bK{\mathbb{K}}

\def\bQ{\mathbb{Q}}
\def\bZ{\mathbb{Z}}

\def\A{\mathbf{A}}

\def\F{\mathbf{F}}

\def\H{\mathbf{H}}

\def\P{\mathbf{P}}

\def\cA{\mathcal{A}}
\def\cB{\mathcal{B}}
\def\cC{\mathcal{C}}
\def\cD{\mathcal{D}}

\def\cF{\mathcal{F}}

\def\cG{\mathcal{G}}

\def\cO{\mathcal{O}}

\def\cP{\mathcal{P}}

\def\cQ{\mathcal{Q}}

\def\cR{\mathcal{R}}

\def\cU{\mathcal{U}}

\def\fA{\mathfrak{A}}

\def\fB{\mathfrak{B}}

\def\fG{\mathfrak{G}}

\def\fF{\mathfrak{F}}
\def\fG{\mathfrak{G}}

\def\fH{\mathfrak{H}}

\def\fK{\mathfrak{K}}
\def\fL{\mathfrak{L}}
\def\fM{\mathfrak{M}}

\def\fP{\mathfrak{P}}
\def\fQ{\mathfrak{Q}}

\def\fS{\mathfrak{S}}

\def\fe{\mathfrak{e}}

\def\fp{\mathfrak{p}}
\def\fq{\mathfrak{q}}

\def\deg{\mathbf{deg}}

\def\Gal{\mathrm{Gal}}

\def\ulim{\mathbf{ulim}}

\def\alg{\mathbf{alg}}

\def\sep{\mathbf{sep}}

\def\card{\mathrm{card}}

\def\ultra{\mathrm{ultra}}

\def\icard{\mathrm{icard}}

\begin{document}

\title[Ultra-Galois theory and the Kronecker--Weber theorem]{Ultra-Galois theory and an analogue of the Kronecker--Weber theorem for rational function fields over ultra-finite fields}

\author{Dong Quan Ngoc Nguyen}

\address{Department of Mathematics \\
University of Maryland at College Park \\
College Park, MD 20742 USA}


\email{\href{mailto:dongquan.ngoc.nguyen@gmail.com}{\tt dongquan.ngoc.nguyen@gmail.com} \\
\href{mailto:dongquan@umd.edu}{\tt dongquan@umd.edu}}

\maketitle

\tableofcontents

\begin{abstract}
  
  In the first part of this paper, we develop a general framework that permits a comparison between explicit class field theories for a family of rational function fields $\bF_s(t)$ over arbitrary constant fields $\bF_s$ and explicit class field theory for the rational function field $\fK(t)$ over the nonprincipal ultraproduct $\fK$ of the constant fields $\bF_s$. Under an additional assumption that the constant fields $\bF_s$ are perfect procyclic fields, we prove a correspondence between ramifications of primes $P$ in $\fK(t)$ and ramifications of primes $P_s$ in $\bF_s(t)$, where the $P_s$ are primes in $\bF_s(t)$ whose nonprincipal ultraproduct coincides with $P$. 
  
  In the second part of the paper, we are mainly concerned with rational function fields over a large class of fields, called $n$-th level ultra-finite fields that are a generalization of finite fields. At the $0$-th level, ultra-finite fields are simply finite fields, and for an arbitrary positive integer $n$, an $n$-th level ultra-finite field is inductively defined as a nonprincipal ultraproduct of $(n - 1)$-th level ultra-finite fields. We develop an analogue of cyclotomic function fields for rational function fields over $n$-th level ultra-finite fields that generalize the works of Carlitz and Hayes for rational function fields over finite fields such that these cyclotomic function fields are in complete analogy with the classical cyclotomic fields $\bQ(\zeta)$ of the rationals $\bQ$. The main result in the second part of the paper is an analogue of the Kronecker--Weber theorem for rational function fields over $n$-th level ultra-finite fields that explicitly describes, from a model-theoretic viewpoint, the maximal abelian extension of the rational function field over a given $n$-th level ultra-finite field for all $n \ge 1$.

\end{abstract}

\section{Introduction}

The main goal of class field theory, a fundamental branch of algebraic number theory, is to describe and classify all the abelian extensions of a given field $F$. Among the classical examples are those where $F$ is a number field--a finite extension of the rationals $\bQ$, a function field of one variable over a finite field $\bF_q$ such as the rational function field $\bF_q(t)$, or a local field such as the field of $p$-adic numbers $\bQ_p$. 

\textit{Explicit class field theory for a given field $F$} asks for the construction of the \textit{maximal abelian extension} of $F$, i.e., it asks for a description of all the abelian extensions of $F$, using only intrinsic information of $F$. The celebrated 
Kronecker--Weber theorem (see Washington \cite{washington}) describes a complete picture of explicit class field theory for the field $\bQ$ of rational numbers, i.e., every finite abelian extension of $\bQ$ is contained in a cyclotomic field $\bQ(\zeta)$ for some root of unity $\zeta$. The theorem was first stated by Kronecker \cite{kronecker} in 1853, but his argument was not complete for abelian extensions of degree $2^n$ for some positive integer $n$. Weber \cite{Weber-KW} provided a proof of the theorem in 1886, but his proof also contained some gaps and errors that were pointed out and corrected by Neumann \cite{neumann-kronecker-weber}. The first complete proof of the Kronecker--Weber theorem was given by Hilbert \cite{Hilbert} in 1896. 
Hilbert's 12th Problem, one of the $23$ mathematical Hilbert problems, asks for the extension of the Kronecker--Weber theorem on abelian extensions of the rationals $\bQ$ to a more general algebraic number field.

Carlitz \cite{Carlitz-1935, Carlitz-1938} constructed a large class of abelian extensions over rational function fields $\bF_q(t)$ over finite fields $\bF_q$ that are an analogue of cyclotomic fields $\bQ(\zeta)$ in the function field context, using  what is now known as the \textit{Carlitz module}, the prototypical \textit{Drinfeld module} (see Goss \cite{goss}, Rosen \cite{rosen}, or Thakur \cite{thakur} for an exposition of Drinfeld modules; we will recall some basic notions concerning the Carlitz module and cyclotomic function fields, an analogue of cyclotomic fields $\bQ(\zeta)$, for rational function fields $\bF_q(t)$). Building on the work of Carlitz, and viewing the abelian extensions of $\bF_q(t)$ that are constructed by Carlitz from the 1930s, as an analogue of the classical cyclotomic fields, Hayes \cite{hayes} provided an explicit class field theory for rational function fields $\bF_q(t)$ over finite fields $\bF_q$, and proved an analogue of the classical Kronecker--Weber theorem in the function field context. Drinfeld \cite{drinfeld1, drinfeld2} and Hayes \cite{hayes-global-FF} independently describe explicit class field theory for an arbitrary global field over finite fields. The works of Drinfeld and Hayes completely solve Hilbert's 12th Problem for global function fields over finite fields.

The main aim of this paper is to prove an analogue of the Kronecker--Weber theorem for rational function fields over a large class of fields, called \textit{$n$-th level ultra-finite fields} that are a generalization of finite fields. At the $0$-th level, ultra-finite fields are simply finite fields. A $1$-st level ultra-finite field, or for brevity, an ultra-finite field is a nonprincipal ultraproduct of finite fields with respect to a given nonprincipal ultrafilter (see Subsection \ref{subsec-ultraproducts}). For a general positive integer $n$, an $n$-th level ultra-finite field is inductively defined as a nonprincipal ultraproduct of $(n -1)$-th level ultra-finite fields. Since an ultraproduct of the same finite field $\bF_q$ is $\bF_q$ itself (see Bell--Slomson \cite[Lemma 3.7]{bell-slomson}), every finite field is an $n$-th level ultra-finite field for all positive integers $n$. Therefore rational function fields over $n$-th level ultra-finite fields can be viewed as a generalization of rational function fields over finite fields.

From a model-theoretic point of view, an ultra-finite field is an infinite model of the theory of all finite fields that was started and systematically studied by Ax \cite{ax-1968}. In view of this relation, $n$-th level ultra-finite fields share many analogous properties with finite fields, and it is natural to expect an explicit class field theory for rational function fields over $n$-th level ultra-finite fields in the same spirit as in the works of Drinfeld and Hayes for rational function fields over finite fields. We state a simplified version of one of the main results in this paper that can be viewed as an analogue of the Kronecker--Weber theorem for rational function fields over an ultra-finite field. 

\begin{theorem}
\label{thm-Theorem-A-in-the-Introduction}
$(\text{See Theorems \ref{thm-greatest-thm-an-analogue-of-KW-for-F=fK(t)} and \ref{thm-the-greatest-thm-KW-thm-for-function-fields-over-nth-level-ultra-finite-fields}})$

Let $S$ be an infinite set, and let $\cD$ be a nonprincipal ultrafilter on $S$. Let $\bF_s$ be a finite field of $q_s$ elements, and let $\fF_s = \bF_s(t)$ be the function field of one variable over $\bF_s$ for $s \in S$. Let $\fK$ be the ultraproduct of the finite fields $\bF_s$ with respect to $\cD$, and let $\F = \fK(t)$ be the function field of one variable over $\fK$.

Let $\cA_s$ be the maximal abelian extension of $\fF_s$ that is explicitly constructed as in the work of Hayes \cite{hayes} for $s \in S$, and let $\cA$ be the ultraproduct of the $\cA_s$ with respect to $\cD$. Then the maximal abelian extension of $\F = \fK(t)$ is precisely the set of all elements in $\cA$ that are algebraic over $\F$.

\end{theorem}

The above theorem even holds for rational function fields over a general $n$-th level ultra-finite field for all $n$. Theorem \ref{thm-Theorem-A-in-the-Introduction} follows from a \textit{general framework} that we develop to establish a relation between explicit class field theories for a family of rational function fields of one variable over general fields $\bF_s$ and explicit class field theory for the rational function field of one variable over the ultraproduct $\fK$ of the fields $\bF_s$. We describe the framework in more detail.

At the constant field level, $\bF_s$ is an arbitrary field of arbitrary characteristic for each $s \in S$, where $S$ always denotes an infinite set that is equipped with a nonprincipal ultrafilter $\cD$. Such an ultrafilter always exists, following the work of Tarski \cite{Tarski}. We form the ultraproduct $\fK$ of the fields $\bF_s$ with respect to $\cD$ (see Subsection \ref{subsec-ultraproducts} for a notion of ultraproducts). The ultraroduct $\fK$ is equipped with a field structure that inherits from the field structures of the fields $\bF_s$. Such an ultraproduct is called an \textit{ultra-field} \footnote{This terminology was first used in Schoutens \cite{schoutens}. Inspired by such usage, we always add the prefix ``ultra" to an algebraic structure to indicate an ultraproduct of sets equipped with the same algebraic structure such as groups, rings, fields, modules, etc. For example, an ultraproduct of groups will be called an \textit{ultra-group}.}  in this paper. 

The characteristic of $\fK$ depends on those of the fields $\bF_s$. More precisely, if there exists a prime $p > 0$ such that the set $\{s \in S\; | \; \text{$\bF_s$ is of characteristic $p$}\}$ belongs to the ultrafilter $\cD$, then $\fK$ is of the same characteristic $p$; otherwise $\fK$ is of characteristic $0$. 

For $s \in S$, let $\fF_s = \bF_s(t)$ denote the function field of one variable over $\bF_s$, and let $\F = \fK(t)$ be the function field of one variable over $\fK$. At the constant field level, \L{}o\'s' theorem, one of the most fundamental transfer theorem in model theory, implies that the family $(\bF_s)_{s\in S}$ satisfies a first-order property $(P)$ in the language of fields if and only if the ultraproduct $\fK$ satisfies the same first-order property. In extending the same argument to the family $(\fF_s)_{s\in S}$ of function fields, the function field $\F$ may not satisfy some first-order property $(P)$ that the family $(\fF_s)_{s\in S}$ satisfies since the ultraproduct of the function fields $\fF_s$, called the \textit{ultra-hull of $\F$} and denoted by $\cU(\F)$, properly contains $\F$. In fact there are infinitely many transcendental elements over $\F$ in the complement $\cU(\F)\setminus \F$, and thus if $(P)$ is a first-order property concerning algebraic extensions of $\F$ or algebraic elements over $\F$ (see Chatzidakis \cite[Section 2.1]{Chatzidakis} for examples of such properties), then $\F$ may fail to satisfy $(P)$. The aim of the framework that we develop, is to remedy this problem and thus permits a comparison between Galois extensions of function fields $\fF_s$ and Galois extension of $\F$. 

Let $\fL_s$ be an algebraic extension of $\fF_s$ for $s \in S$. In order to find a corresponding algebraic extension of $\F$ that shares analogous properties with the family $(\fL_s)_{s\in S}$, we fist group the fields $\fL_s$ together, and form the ultraproduct $\fL$ of the fields $\fL_s$. The ultra-field $\fL$ is an extension of the ultra-hull $\cU(\F)$ of $\F$ that contains information intrinsic to the family of extensions $\fL_s/\fF_s$. In order to avoid transcendental elements over $\F$ that arise when taking the ultraproduct of $\fL_s$, it is natural to consider only the set of elements of $\fL$ that are algebraic over $\F$, i.e., the set $\fL_{\alg} = \fL \cap \F^{\alg}$, where $\F^{\alg}$ denotes an algebraic closure of $\F$ that is suitably chosen such that the ultraproduct of given algebraic closures $\fF_s^{\alg}$ of $\fF_s$ contains $\F^{\alg}$. The set $\fL_{\alg}$ is clearly a subfield of $\F^{\alg}$, and an algebraic extension of $\F$ that we call the \textit{algebraic part of $\fL$ over $\F$}. A key result proved in Section \ref{sec-algebraic-part-of-an-ultra-finite-extension} (see Lemma \ref{rem-irreducibility-is-the-same-in-F-and-U(F)}) that shows that for a polynomial $P \in \F[x]$, the irreducibility of $P$ over $\F$ is the same as the irreducibility of $P$ over the ultra-hull $\cU(\F)$ of $\F$, permits proving that the algebraic part of the ultraproduct of Galois extensions of $\fF_s$ is also Galois over $\F$, i.e., $\fL_{\alg}$ is Galois over $\F$ if $\fL_s$ is Galois over $\fF_s$ for all $s \in S$. The procedure of taking algebraic parts provides a construction of certain Galois extensions of $\F$. It is natural to ask whether \textit{this procedure exhausts all Galois extensions of $\F$ in a given separable closure of $\F$}. We answer this question in the affirmative by developing a \textit{theory of shadows of separable extensions of $\F$} in Subsection \ref{subsec-shadows-and-ultra-shadows} that for a given Galois extension $\H$ of $\F$ with Galois group $G$, establishes the existence of a Galois extension $\fH_s$, called the \textit{$s$-th shadow of $\H$}, of $\fF_s$ with the same Galois group $G$ for each $s \in S$ such that the algebraic part of the ultraproduct of the Galois extensions $\fH_s$ coincides with $\H$. Thus there is a correspondence between a family of Galois extensions $(\fL_s)_{s\in S}$ of the function fields $(\fF_s)_{s\in S}$ and Galois extensions of $\F$ which permits a comparison between Galois theory over the family $(\fF_s)_{s\in S}$ and Galois theory over $\F$. Using this correspondence, we prove that there is a relation between explicit class field theories for the family $(\fF_s)_{s\in S}$ of rational function fields over constant fields $\bF_s$ and explicit class field theory for the rational function field $\F$ over the ultraproduct $\fK$ of the constant fields $\bF_s$. Theorem \ref{thm-Theorem-A-in-the-Introduction} follows immediately by letting $\bF_s$ be the finite field of $q_s$ elements for $s\in S$, and using the work of Hayes \cite{hayes} that provides explicit class field theories for rational function fields over finite fields $\bF_s$. 

We summarize the discussion above in Table \ref{table1}.

\begin{table}[h!]
\begin{center}
\label{table1}
\caption{The correspondence between Galois extensions of rational function fields $(\fF_s)_{s\in S}$ and Galois extensions of rational function field $\F$}
\begin{tabular}{l|c}
\textbf{Rational function fields $\fF_s$} & \textbf{Rational function field $\F$} \\
\hline
\hline
\text{Galois extensions $\fL_s$ with ultraproduct $\fL$} &\text{The algebraic part $\fL_{\alg} = \fL \cap \F^{\alg}$} \\
& \text{is Galois over $\F$} \\
\hline 
\text{$s$-th shadows $\fH_s$ of $\H$ are Galois over $\fF_s$ } &\text{Galois extension $\H$ with} \\
\text{for $s \in S$ with the same Galois group $G$} & \text{Galois group $G$} \\
\text{such that the algebraic part of } &  \\
\text{the ultraproduct of the shadows $\fH_s$ } &  \\
\text{coincides with $\H$} &  \\

\hline
\end{tabular}
\end{center}
\end{table}

The above correspondence is the contents of Theorem \ref{thm-main-thm1-Algebraic-part-of-L-is-a-finite-separable-extension}, \ref{thm-main-thm2-Galois-property-of-algebraic-parts}, and \ref{thm-main-thm3-Galois-group-structures-of-shadows} in Section \ref{sec-ramification-of-primes-Hilbert-theory}.

Under an additional assumption that the constant fields $\bF_s$ are \textit{perfect procyclic fields} (see Ax \cite{ax-1968} or Fried--Jarden \cite{FJ} for a notion of these fields), we can establish, for a prime $P$ of the rational function field $\F$, a relation between ramifications of primes $(P_s)_{s\in S}$ of the function fields $(\fF_s)_{s\in S}$ in Galois extensions $(\fL_s)_{s\in S}$ and ramification of the prime $P$ in the algebraic part of the ultraproduct $\fL$ of the extensions $\fL_s$, where the $P_s$ are primes in $\fF_s$ such that the ultraproduct of $P_s$ coincides with $P$ (see Lemma \ref{lem-elementary-lemma0} for the existence of such primes $P_s$). The ramification relation between primes in $(\fF_s)_{s\in S}$ and primes in $\F$ is established in Theorems \ref{thm-big-theorem-I-the-finite-case} and \ref{thm-big-theorem-II-infinite-case}.

In Section \ref{sec-Hilbert-12th-problem}, using the theories of algebraic parts and shadows that we develop in Section \ref{sec-algebraic-part-of-an-ultra-finite-extension} and the ramification correspondence that we establish in Section \ref{sec-ramification-of-primes-Hilbert-theory}, we construct a large class of abelian extensions of a given rational function field $\F^{(n)}$ over an arbitrary $n$-th level ultra-finite field that can be viewed as an analogue of classical cyclotomic fields $\bQ(\zeta)$. We call these abelian extensions \textit{cyclotomic function fields for $\F^{(n)}$}. Our construction generalizes that of cyclotomic function fields for rational function fields $\bF_q(t)$ over finite fields $\bF_q$ in the works of Carlitz \cite{Carlitz-1938} and Hayes \cite{hayes}, and provides an explicit description of elements in the maximal abelian extension of the rational function field $\F^{(n)}$. In the same section, we prove a stronger statement (see Theorems \ref{thm-greatest-thm-an-analogue-of-KW-for-F=fK(t)} and \ref{thm-the-greatest-thm-KW-thm-for-function-fields-over-nth-level-ultra-finite-fields}) than Theorem \ref{thm-Theorem-A-in-the-Introduction} that explicitly describes the maximal abelian extension of a rational function field over an $n$-th level ultra-finite field for all $n$. 

In Subsections \ref{subsec-inverse-Galois-problem} and \ref{subsec-ramificaion-in-the-inverse-Galois-problem}, we use the theory of shadows developed in Subsection \ref{subsec-shadows-and-ultra-shadows} and the prime ramification relation proved in Subsections \ref{subsec-prime-ramification-in-L-alg-finite-case} and \ref{subsec-prime-ramification-in-L-alg-infinite-case} to study the inverse Galois problem over rational function fields $\bF_q(t)$ over finite fields $\bF_q$ and ramification in the inverse Galois problem over rational function fields over algebraically closed fields of positive characteristics. The results concerning the inverse Galois problem are not new, but we provide a new look at these results, using the framework that we develop in this paper.

In Subsection \ref{subsubsection-relation-the-maximal-abelian-extension-of-function-fields-overACFs-of-p-and-0-chars}, based on the theory of algebraic parts that we develop in Subsection \ref{subsec-algebraic-parts}, we explain from a model-theoretic viewpoint why the construction of the maximal abelian extension of a rational function field over an algebraically closed field of characteristic $0$ is carried out in complete analogy with the construction of the maximal pro-prime-to-$p$ extension of a rational function field over an algebraically closed field of characteristic $p > 0$.

\section{Some basic notions and notation}
\label{sec-basic-notions-and-notation}

In this section, we recall some basic notions in model theory that we need in this paper, and fix some notation. Our main references for model theory are Bell--Slomson \cite{bell-slomson}, Chang--Keisler \cite{CK}, Rothmaler \cite{rothmaler}, and Schoutens \cite{schoutens}. The main reference for hyperintegers is Goldblatt \cite{Goldblatt}.

\subsection{Ultrafilters}

Let $S$ be an infinite set that we fix throughout this paper. A collection $\cD$ consisting of infinite subsets of $S$ is called a \textbf{nonprincipal ultrafilter on $S$} if the following conditions are satisfied.
\begin{itemize}

\item [(i)] $A_1 \cap A_2 \cap \cdots \cap A_n$ belong to $\cD$ for any sets $A_1, \ldots, A_n \in \cD$; and

\item [(ii)] for any subset $A \subset S$, either $A$ or its complement $S\setminus A$ belongs to $\cD$.

\end{itemize}

By definition, the empty set $\emptyset$ does not belong to $\cD$, and it thus follows from (ii) above that $S$ belongs to $\cD$. The existence of nonprincipal ultrafilters on an infinite set is a consequence of Zorn's lemma (see Tarski \cite{Tarski}, or Rothmaler \cite[Lemma 4.3.1]{rothmaler} for a modern account). 

For the rest of this paper, the symbol $\cD$ always denotes a nonprincipal ultrafilter on $S$. We will need the following basic result in some places of this paper.

\begin{lemma}
\label{lem-at-least-one-set-in-th-union-of-sets-is-in-D}

Let $A$ be a set in the nonprincipal ultrafilter $\cD$. If $A = \cup_{i = 1}^h B_i$ for some subsets $B_i \subset S$, then there exists an integer $1 \le \ell \le h$ such that $B_{\ell} \in \cD$.

\end{lemma}

\begin{proof}

Assume the contrary, i.e., $B_i \not\in \cD$ for all $1 \le i \le h$. Since $\cD$ is an ultrafilter, $S\setminus B_i \in \cD$ for all $1 \le i \le h$, and thus
\begin{align*}
\bigcap_{i = 1}^hS\setminus B_i = S\setminus \bigcup_{i = 1}^h B_i = S \setminus A \in \cD,
\end{align*}
which is a contradiction since $A \in \cD$.

\end{proof}

\subsection{Ultraproducts}
\label{subsec-ultraproducts}

In this subsection, we recall a notion of ultraproducts. Let $A_s$ be a set for each $s \in S$. We form the Cartesian product $\widehat{\cA}$ of the sets $A_s$, i.e., $\widehat{\cA}$ consists of all tuples $(a_s)_{s\in S}$ for $a_s \in A_s$. For elements $\widehat{a} = (a_s)_{s\in S}$, $\widehat{b} = (b_s)_{s \in S}$ in $\widehat{\cA}$, we say that $\widehat{a}$ is equivalent to $\widehat{b}$ if the set $\{s \in S \; | \; a_s = b_s\}$ belongs to $\cD$. We write $\widehat{a} \sim_{\cD} \widehat{b}$ whenever $\widehat{a}$ is equivalent to $\widehat{b}$. It is clear that ``$\sim_{\cD}$" defines an equivalence relation on $\widehat{\cA}$.

\begin{definition}
\label{def-ultraproducts}
(Ultraproduct)

The \textbf{nonprincipal ultraproduct of the set $A_s$ with respect to $\cD$} is the set of all equivalence classes $\widehat{\cA}/\sim_{\cD}$. In notation, we always write $\prod_{s\in S}A_s/\cD$ for the nonprincipal ultraproduct of $A_s$ with respect to $\cD$. For brevity, and since $\cD$ is always assumed to be nonprincipal throughout the paper, we simply call $\prod_{s\in S}A_s/\cD$ the \textbf{ultraproduct of the sets $A_s$}, without referring to the nonprincipal ultrafilter $\cD$.

\end{definition}

We denote each element $a \in \prod_{s\in S}A_s/\cD$, being the equivalence class of the tuple $(a_s)_{s\in S}\in \widehat{\cA}$, by $\ulim_{s\in S}a_s$ to indicate the \textbf{$s$-th component $a_s$ of $a$} for each $s \in S$. We also say $a$ is the \textbf{ultra-limit of the $a_s$}. 

In the case where $A_s = A$ for all $s \in S$ for some set $A$,  the ultraproduct of the $A_s$ is called the \textbf{ultrapower of $A$}, and denoted by $A^{\#}$. There is a canonical map from $A$ to $A^{\#}$ that sends each element $a \in A$ to the equivalence class $\ulim_{s \in S}a \in A^{\#}$ in which all $s$-th components are equal to $a$. We identify $A$ with its image in $A^{\#}$ under this map, which permits viewing $A$ as a subset of $A^{\#}$.

Throughout this paper, for a property $(P)$, if the set $\{s \in S\; : \; \text{(P) holds in $A_s$}\}$ belongs to the ultrafilter $\cD$, we say that \textbf{$P$ holds for $\cD$-almost all $s \in S$}. 

If all the sets $A_s$ are equipped with an algebraic structure such as groups, ring, or fields, the ultraproduct $\cA = \prod_{s\in S}A_s/\cD$ can be equipped with the same structure by defining addition and multiplication on $\cA$ using componentwise addition and multiplication from the $s$-th components $A_s$ of $\cA$. More precisely, we can define $ab = \ulim_{s\in S}a_sb_s$ and $a + b = \ulim_{s\in S}(a_s + b_s)$ for elements $a = \ulim_{s\in S}$, $b = \ulim_{s\in S}b_s$ in $\cA$. 

We recall one of the most fundamental transfer principle for ultraproducts that we frequently use in this paper.

\begin{theorem}
$(\text{\L{}o\'s' theorem, see \cite{bell-slomson} or \cite{rothmaler}})$
\label{thm-Los}

Let $\cR$ be a ring, let $A_s$ be an $\cR$-algebra for each $s \in S$, and let $\cA$ be their ultraproduct. Let $\psi(\zeta_1, \ldots, \zeta_m)$ be a formula with parameters from $\cR$ whose free variables are among $\zeta_1, \ldots, \zeta_m$. For each $s \in S$, let $\widehat{a}_s$ be an $m$-th tuple in $A_s$, and let $\widehat{a}$ be their ultraproduct which is an $m$-th tuple in $\cA$. Then $\psi(\widehat{a}_s)$ holds in $A_s$ for $\cD$-almost all $s \in S$ if and only if $\psi(\widehat{a})$ holds in $\cA$.

\end{theorem}

\subsection{Hyperintegers and hypernaturals}
\label{subsec-hyperintegers}

In this subsection, we recall a notion of hyperintegers, and its action on ultraproducts of rings and of fields.

We denote by $\bZ$ the set of integers. We write $\bZ_{> 0}$, $\bZ_{\ge 0}$ for the sets of positive and nonnegative integers, respectively.

Throughout this paper, we let $\bZ^{\#}$ be the ultrapower $\prod_{s \in S}\bZ/\cD$ whose elements are called \textbf{hyperintegers}. As discussed in the previous subsection, we can view $\bZ$ as a subset of $\bZ^{\#}$ by embedding it diagonally into $\bZ^{\#}$. 

We denote by $\bZ^{\#, > 0}$ the ultrapower $\prod_{s\in S}\bZ_{>0}/\cD$ whose elements are hyperintegers $n = \ulim_{s\in S}n_s$ such that $n_s > 0$ for $\cD$-almost all $s \in S$. Similarly, we let $\bZ^{\#, \ge 0}$ be the ultrapower $\prod_{s\in S}\bZ_{\ge 0}/\cD$.

\subsubsection{The action of hyperintegers on certain ultraproducts}
\label{subsubsec-action-of-hyperinteger-on-ultraproducts}

Let $\cA = \prod_{s\in S}A_s/\cD$ be an ultraproduct of rings $A_s$. For each hyperinteger $n = \ulim_{s\in S}n_s \in \bZ^{\#, \ge 0}$ and each nonzero element $a = \ulim_{s\in S}a_s \in \cA$ with $a_s \in A_s$, we can define
\begin{align}
\label{def-action-of-positive-hyperintegers-on-ultraproducts}
a^n := \ulim_{s\in S}a_s^{n_s} \in \cA.
\end{align}

Whenever $a = 0$ and $n = \ulim_{s\in S}n_s \in \bZ^{\#, >0}$, we let $a^n = 0$. 

It is clear that the above definition is well-defined, i.e., independent of the choices of $a_s$ and $n_s$ for $a$ and $n$, respectively. By \L{}o\'s' theorem, one can verify that  $(ab)^n = a^nb^n$, $a^na^m = a^{n + m}$, and $(a^n)^m = a^{nm}$ for all nonzero elements $a, b \in \cA$ and $n, m \in \bZ^{\#, \ge 0}$. 

If $\cA = \prod_{s\in S}A_s$ is an ultraproduct of fields $A_s$, we can extend the action (\ref{def-action-of-positive-hyperintegers-on-ultraproducts}) to the whole ring of hyperintegers $\bZ^{\#}$ by letting
\begin{align*}
a^n := \ulim_{s\in S}a_s^{n_s}
\end{align*}

for all nonzero elements $a \in \cA$ and all $n \in \bZ^{\#}$. 

\subsubsection{Arithmetic of hyperintegers} 

Let $n = \ulim_{s \in S}n_s$, $m = \ulim_{s \in S}m_s$ be hyperintegers for some integers $n_s, m_s \in \bZ$. We say that \textbf{$n$ divides $m$} (or \textbf{$n$ is a divisor of $m$}) if $m = n\ell$ for some hyperinteger $\ell = \ulim_{s \in S}\ell_s$ with $\ell_s \in \bZ$, i.e., $m_s = n_s\ell_s$ for $\cD$-almost all $s \in S$.

In particular, for any \textit{integer} $n$, $n$ divides a hyperinteger $m = \ulim_{s \in S}m_s$ if and only if $n$ divides $m_s$ for $\cD$-almost all $s \in S$.

\subsection{Some notation and basic notions}

\begin{itemize}

\item [--] For a given unique factorization domain $R$, we always denote by $\gcd(a_1, \ldots, a_n)$ the \textbf{greatest common divisor of $a_1, \ldots, a_n$} for elements $a_1, \ldots, a_n \in R$.
    
\item [--] If $L/K$ is a Galois extension, $\Gal(L/K)$ denotes the Galois group of $L$ over $K$.

\item [--] If $L$ is a finite algebraic extension of $K$, the notation $[L:K]$ denotes the \textbf{degree of $L$ over $K$} which is the same as the dimension of $L$ over $K$ as a vector space.

\item [--] If $P(t)$ is a polynomial in variable $x$ with coefficients in a ring $R$, we always denote by $\deg(P(t))$ the \textbf{degree of $P(t)$ in $t$}. The degree of the zero polynomial, by convention, equals $-\infty$.

\item [--] Let $P(t) = a_0 + a_1t + \cdots + a_mt^m$ be a polynomial in $R[t]$, where $R$ is a unique factorization domain and the $a_i$ belong to $R$. Then $P$ is a \textbf{primitive} polynomial if $\gcd(a_0, a_1, \cdots, a_m) = 1$.

\item [--] For a given polynomial ring $F[t]$ with $F$ being a field, a \textbf{prime of degree $d$ in $F[t]$} is an irreducible polynomial of degree $d$ in $F[t]$. A \textbf{monic prime in $F[t]$} is a prime in $F[t]$ whose leading coefficient is $1$. Our main reference for the theory of algebraic functions of one variable is Chevalley \cite{chevalley}.

\item [--] For a set $A$, we always denote by $\card(A)$ the cardinality of $A$.

\item [--] For a group $G$, we always denote by $1_G$ the identity element of $G$. 

\end{itemize}

\section{Ultra-Galois theory and ultra-finite fields}
\label{sec-higher-dimensional-ultra-finite-fields}

In this section, we recall a notion of ultra-fields in Nguyen \cite{nguyen-APAL-2024}, and prove some basic results concerning ultra-field extensions that we will need in some places of this paper. Building on the theory of ultra-fields, we introduce ultra-Galois theory that is suitable for understanding extensions of ultraproducts of Galois extensions of ultra-fields. In general, an ultra-Galois extension of a given ultra-field $F$ is of \textit{uncountable dimension} over $F$ as a vector space. The main result in this section shows that every ultra-Galois extension of a given ultra-field $F$ is Galois over $F$ in the classical sense whenever the extension is of finite dimension over $F$ as a vector space.

\subsection{Ultra-fields and ultra-Galois theory}
\label{subsec-ultra-fields}

In this subsection, we introduce ultra-Galois theory that is analogous to the classical Galois theory, and is suitable for understanding extensions of ultraproducts of fields. In Nguyen \cite{nguyen-APAL-2024}, the author introduces a theory of ultra-fields that is analogous to the classical field theory. We begin by recalling some notions from \cite{nguyen-APAL-2024} that we need for developing ultra-Galois theory. 

\subsubsection{Ultra-fields and ultra-algebraic extensions}
\label{subsubsec-ultra-fields}

An ultraproduct of fields $F = \prod_{s\in S}F_s/\cD$, where $F_s$ is a field for $\cD$-almost all $s \in S$, is called an \textbf{ultra-field}. 

By \L{}o\'s' theorem, every ultra-field is a field (see Nguyen \cite[Proposition 2.1]{nguyen-APAL-2024}).

\begin{definition}
\label{def-ultra-field-extensions}
(ultra-field extensions)

An \textbf{ultra-field extension of an ultra-field $F = \prod_{s\in S}F_s/\cD$}  is an ultra-field $L = \prod_{s\in S}L_s/\cD$ such that $L_s$ is a field extension of $F_s$ for $\cD$-almost all $s \in S$. $F$ is also called an \textbf{ultra-subfield of $L$}.

\end{definition}

\begin{definition}
\label{def-ultra-hulls}

Let $F = \prod_{s\in S}F_s/\cD$ be an ultra-field. The \textbf{ultra-hull of the polynomial ring $F[t]$} is the ultraproduct $\prod_{s\in S}F_s[t]/\cD$ of polynomials rings $F_s[t]$. We always denote by $\cU(F[t])$ the ultra-hull of $F[t]$.

In a similar way, we define the \textbf{ultra-hull of the rational function field $F(t)$} is the ultraproduct $\prod_{s\in S}F_s(t)/\cD$ of rational function fields $F_s(t)$. We denote by $\cU(F(t))$ the ultra-hull of $F(t)$.

It is not difficult to verify that $\cU(F[t])$ is an integral domain, and $\cU(F(t))$ is its quotient field.

\end{definition}

An \textbf{ultra-polynomial $f$ over an ultra-field $F = \prod_{s\in S}F_s/\cD$} is an element in the ultra-hull $\cU(F[t])$, i.e., $f = \ulim_{s\in S}f_s$, where $f_s$ is a polynomial in $F_s[t]$ for $\cD$-almost all $s \in S$.

For an element $a = \ulim_{s\in S}a_s$ in an ultra-field extension $L = \prod_{s\in S}L_s/\cD$ of an ultra-field $F = \prod_{s\in S}F_s/\cD$ and an ultra-polynomial $f = \ulim_{s\in S}f_s \in \cU(F[t])$ for some polynomials $f_s \in F_s[t]$, the \textbf{value of $f$ at $a$} is the element $\ulim_{s\in S}f_s(a_s) \in L$. We denote by $f(a)$ the value of $f$ at $a$. 

If $f(a) = 0$, $a$ is called a \textbf{root or zero of $f$}. Thus $a$ is a root of $f$ if and only if $a_s$ is a root of $f_s$ for $\cD$-almost all $s \in S$.

\begin{definition}
\label{def-ultra-algebraic-extensions}
(ultra-algebraic extensions)

Let $F = \prod_{s\in S}F_s/\cD$ be an ultra-field. An element $a$ in an ultra-field extension of $F$ is called \textbf{ultra-algebraic over $F$} if $f(a) = 0$ for some ultra-polynomial $f$ over $F$. 

An ultra-field extension $L = \prod_{s\in S}L_s/\cD$ is \textbf{ultra-algebraic over $F$} if every element of $L$ is ultra-algebraic over $F$.

If the degree $[L_s : F_s]$ as a vector space over $F_s$ equals a positive integer $n_s$ for $\cD$-almost all $s \in S$, the \textbf{ultra-degree of $L$ over $F$} is defined to be the positive hyperinteger $\ulim_{s\in S}n_s \in \bZ^{\#, > 0}$.

\end{definition}

\begin{proposition}
\label{prop-explicit-description--for--algebraic-extension-of-degree-d-of-ultra-fields}

Let $F = \prod_{s\in S}F_s/\cD$ be an ultra-field, where $F_s$ is a field for $\cD$-almost all $s \in S$. Let $F_s(d_s)$ denote an algebraic extension of degree $d_s$ over $F_s$ for $\cD$-almost all $s \in S$, where $d_s$ is a positive integer. If for some positive integer $d$,
\begin{align*}
\{s \in S\; | \; d_s = d\} \in \cD,
\end{align*}
then the ultra-field $L = \prod_{s\in S}F_s(d_s)/\cD$ is an algebraic extension of degree $d$ over $F$. In particular, the ultra-degree of $L$ over $F$ coincides with the degree of $L$ over $F$, both of which equal $d$.

\end{proposition}

\begin{proof}

Since
\begin{align*}
\{s \in S\; | \; d_s = d\} \in \cD,
\end{align*}
we deduce that
\begin{align*}
L = \prod_{s\in S}F_s(d_s)/\cD = \prod_{s\in S}F_s(d)/\cD.
\end{align*}

Let $V_s =\{\alpha_{1, s}, \alpha_{2, s}, \ldots, \alpha_{d, s}\}$ be a subset of $F_s$ consisting of elements of $F_s(d)$ that are linearly independent over $F_s$ for $\cD$-almost all $s \in S$. Set
\begin{align*}
V = \prod_{s\in S}V_s/\cD \subset L =\prod_{s\in S}F_s(d)/\cD.
\end{align*}

Since $\card(V_s) = d$ for $\cD$-almost all $s \in S$, we deduce from \cite[Lemma 3.7]{bell-slomson} that $\card(V) = d$. Let $\alpha_1, \ldots, \alpha_d$ be all distinct elements in $V$ such that
\begin{align*}
\alpha_i = \ulim_{s\in S}\alpha_{s(i), s},
\end{align*}
where $\alpha_{s(i), s} \in V_s$ for some integer $1 \le s(i) \le d$ for $\cD$-almost all $s \in S$. 

We prove that $\alpha_1, \ldots, \alpha_d$ are linearly independent over $F$. Indeed, take arbitrary elements $a_1, \ldots, a_d \in F$ such that
\begin{align*}
a_1\alpha_1 + \cdots a_d\alpha_d = 0.
\end{align*}

For each $1 \le i \le d$, write $a_i = \ulim_{s\in S}a_{i, s}$, where $a_{i, s} \in F_s$ for $\cD$-almost all $s \in S$. Thus
\begin{align*}
\ulim_{s\in S}(a_{1, s}\alpha_{1, s} + \cdots a_{d, s}\alpha_{d, s}) = a_1\alpha_1 + \cdots a_d\alpha_d = 0,
\end{align*}
which implies that
\begin{align*}
a_{1, s}\alpha_{1, s} + \cdots a_{d, s}\alpha_{d, s} = 0
\end{align*}
for $\cD$-almost all $s \in S$. Since $\alpha_{1, s}, \ldots, \alpha_{d, s}$ are linearly independent over $F_s$, we deduce that $a_{1, s} = \cdots = a_{d, s} = 0$ for $\cD$-almost all $s \in S$. Thus 
\begin{align*}
a_1 = a_2 = \cdots = a_d = 0.
\end{align*}
Therefore $\alpha_1, \ldots, \alpha_d$ are linearly independent over $F$. 

Assume that there exist $d + 1$ elements $\beta_1, \ldots, \beta_{d + 1}$ of $L$ that are linearly independent over $F$. For each $1 \le i \le d + 1$, write $\beta_i = \ulim_{s\in S}\beta_{i, s}$ for some elements $\beta_{i, s}\in F_s$ for $\cD$-almost all $s \in S$. We claim that $\beta_{1, s}, \ldots, \beta_{d + 1, s}$ are linearly independent over $F_s$ for $\cD$-almost all $s \in S$. Indeed, for $\cD$-almost all $s \in S$, take an arbitrary elements $b_{1, s}, \ldots, b_{d + 1, s}$ in $F_s$ such that
\begin{align*}
b_{1, s}\beta_{1, s} + \cdots + b_{d + 1, s}\beta_{d + 1, s} = 0.
\end{align*}
Thus
\begin{align*}
b_1\beta_1 + \cdots + b_{d + 1}\beta_{d + 1} = 0,
\end{align*}
where $b_i = \ulim_{s\in S}b_{i, s}$ for every $1 \le i \le d + 1$. Since $\beta_1, \ldots, \beta_{d + 1}$ are linearly independent over $F$, $b_1 = \cdots = b_{d + 1} = 0$, and thus for every $1 \le i \le d + 1$,
\begin{align*}
0 = b_i = \ulim_{s\in S}b_{i, s}.
\end{align*}
Therefore
\begin{align*}
A_i = \{s \in S \; | \; b_{i, s} = 0\} \in \cD
\end{align*}
for all $1 \le i \le d + 1$. Since $\cD$ is an ultrafilter, 
\begin{align*}
A_1 \cap A_2 \cap \cdots \cap A_{d + 1} \in \cD,
\end{align*}
which implies that $b_{1, s} = b_{2, s} = \cdots = b_{d + 1, s} = 0$ for $\cD$-almost all $s \in S$. Thus $\beta_{1, s}, \ldots, \beta_{d + 1, s}$ are linearly independent over $F_s$ for $\cD$-almost all $s \in S$, a contradiction since $F_s(d_s)$ is of degree $d_s = d$ over $F_s$ for $\cD$-almost all $s \in S$.

Therefore by what we have showed, the dimension of $L$ over $F$ as a vector space over $F$ is precisely $d$, and thus $L$ is an algebraic extension of degree $d$ over $F$.

\end{proof}

\begin{remark}
\label{rem-ultra-algebraic-extensions-generalize-algebraic-extensions}

The above proposition shows that the notion of ultra-algebraic extensions generalizes that of algebraic extensions in field theory.

\end{remark}

\subsubsection{Ultra-Galois extensions}
\label{subsubsec-ultra-Galois-extensions}

We introduce a notion of ultra-Galois extensions.

\begin{definition}
\label{def-ultra-galois-extensions}
(Ultra-Galois extensions)

Let $F = \prod_{s \in S}F_s/\cD$ be an ultra-field. An ultra-field extension $L = \prod_{s\in S}L_s/\cD$  of $F$ is called an \textbf{ultra-Galois extension of $F$} if $L_s$ is a Galois extension of $F_s$ for $\cD$-almost all $s \in S$. 

If $\Gal(L_s/F_s)$ denotes the Galois group of $L_s$ over $F_s$ for $\cD$-almost all $s \in S$, we call the ultraproduct $\prod_{s\in S}\Gal(L_s/F_s)/\cD$ the \textbf{ultra-Galois group of $L$ over $F$}, and denote it by $\Gal_{\ultra}(L/F)$.

\end{definition}

\begin{remark}
\label{rem-elements-in-ultra-Galois-groups}

Each element $\sigma_s \in \Gal(L_s/F_s)$ is a morphism from $L_s$ to itself that can be viewed as a subset $\cA_s$ of the Cartesian product $L_s \times L_s$, consisting of ordered pairs $(a_s, \sigma_s(a_s))$ for $a_s$ ranging over $L_s$. Thus the ultraproduct $\Gal_{\ultra}(L/F)$ can be interpreted as the ultraproduct $\prod_{s\in S}\cA_s/\cD$. We see that
\begin{align}
\label{e-remark-about-elements-in-ultra-Galois-group}
\prod_{s\in S}\cA_s/\cD = \{(\ulim_{s\in S}a_s, \ulim_{s\in S}\sigma_s(a_s))\; | \; a_s \in L_s\}.
\end{align}

We define a map $\sigma : L \to L$ by sending each element $a = \ulim_{s\in S}a_s\in L = \prod_{s\in S}L_s$ to $\sigma(a) = \ulim_{s\in S}\sigma_s(a_s) \in \prod_{s\in S}L_s/\cD = L$. Such a map $\sigma$ can be set-theoretically interpreted as a subset of the Cartesian product $L \times L$, consisting of ordered pairs $(a, \sigma(a)) = (\ulim_{s\in S}a_s, \ulim_{s\in S}\sigma_s(a_s))$ for $a = \ulim_{s\in S}a_s$ ranging over $L$. By (\ref{e-remark-about-elements-in-ultra-Galois-group}), the set of such maps $\sigma$ coincides with $\prod_{s\in S}\cA_s/\cD = \prod_{s\in S}\Gal(L_s/F_s)$. Throughout this paper, we write $\sigma = \ulim_{s\in S}\sigma_s$ to denote such a map $\sigma$. Thus
\begin{align*}
\Gal_{\ultra}(L/F) = \prod_{s\in S}\Gal(L_s/F_s)/\cD = \{\ulim_{s\in S}\sigma_s\; | \; \sigma_s \in \Gal(L_s/F_s)\}.
\end{align*}

\end{remark}

\begin{lemma}
\label{lem-ultra-Galois-group-is-a-subgroup-of-Galois-group}

Let $L = \prod_{s\in S}L_s/\cD$ be an ultra-Galois extension of an ultra-field $F = \prod_{s\in S}F_s/\cD$. Then every element $\sigma \in \Gal_{\ultra}(L/F) = \prod_{s\in S}\Gal(L_s/F_s)/\cD$ is a field automorphism of $L$ that fixes every element of $F$, i.e., $\sigma(a) = a$ for all $a \in F$.

\end{lemma}

\begin{proof}

By Remark \ref{rem-elements-in-ultra-Galois-groups}, every element $\sigma \in \Gal_{\ultra}(L/F)$ is of the form $\ulim_{s\in S}\sigma_s : L \to L$ that sends each element $a = \ulim_{s\in S}a_s \in L$ to $\sigma(a) = \ulim_{s\in S}\sigma_s(a_s) \in L$.

Since $\sigma_s$ is a field automorphism of $L_s$ for $\cD$-almost all $s \in S$, $\sigma_s(a_sb_s) = \sigma_s(a_s)\sigma_s(b_s)$ and $\sigma_s(a_s + b_s) = \sigma_s(a_s) + \sigma_s(b_s)$ for all $a_s, b_s \in L_s$, we deduce that
\begin{align*}
\sigma(ab) = \ulim_{s\in S}\sigma_s(a_sb_s) = \ulim_{s\in S}\sigma_s(a_s)\ulim_{s\in S}\sigma_s(b_s) = \sigma(a)\sigma(b)
\end{align*}
and
\begin{align*}
\sigma(a + b) = \ulim_{s\in S}\sigma_s(a_s+ b_s) = \ulim_{s\in S}\sigma_s(a_s) + \ulim_{s\in S}\sigma_s(b_s) = \sigma(a) + \sigma(b),
\end{align*}
and thus $\sigma$ is a field homomorphism from $L$ to itself.

It is trivial that $\sigma$ is an injective since $\sigma$ is a field homomorphism of $L$ to itself.

Take an arbitrary element $b = \ulim_{s\in S}b_s \in L$, where $a_s \in L_s$ for $\cD$-almost all $s \in S$. Since $\sigma_s$ is an automorphism of $L_s$, there exists an element $a_s \in L_s$ such that $\sigma_s(a_s) = b_s$ for $\cD$-almost all $s \in S$. Thus
\begin{align*}
\sigma(a) = \sigma(\ulim_{s\in S}a_s) = \ulim_{s\in S}\sigma_s(a_s) = \ulim_{s\in S}b_s = b,
\end{align*}
where $a = \ulim_{s\in S}a_s \in L$. Thus $\sigma$ is surjective, and thus $\sigma$ is an automorphism of $L$.

Take an arbitrary element $a = \ulim_{s\in S}a_s \in F = \prod_{s\in S}F_s/\cD$. Since $\sigma_s$ fixes elements of $F_s$, we deduce that
\begin{align*}
\sigma(a) = \ulim_{s\in S}\sigma_s(a_s) = \ulim_{s\in S}a_s = a,
\end{align*}
and thus $\sigma$ fixes elements of $F$.

\end{proof}

The following result shows that the notion of ultra-Galois extensions generalizes that of Galois extensions in the classical Galois theory.

\begin{proposition}
\label{prop-ultra-Galois-extensions-are-Galois}

Let $L = \prod_{s\in S}L_s/\cD$ be an ultra-Galois extension of an ultra-field $F = \prod_{s\in S}F_s/\cD$ such that $L_s$ is a finite Galois extension of degree $d_s$ over $F_s$ for $\cD$-almost all $s \in S$. If there exists a positive integer $d$ such that
\begin{align*}
\{s \in S \; | \; d_s = d\} \in \cD,
\end{align*}
then $L$ is a Galois extension of degree $d$ over $F$ and the Galois group of $L$ over $F$ coincides with the ultra-Galois group of $L$ over $F$, i.e.,
\begin{align*}
\Gal(L/F) = \Gal_{\ultra}(L/F) = \prod_{s\in S}\Gal(L_s/F_s)/\cD.
\end{align*}

\end{proposition}

\begin{proof}

By Proposition \ref{prop-explicit-description--for--algebraic-extension-of-degree-d-of-ultra-fields}, $L$ is an algebraic extension of degree $d$ over $F$. We contend that $L$ is a Galois extension of $F$.

Take an arbitrary irreducible polynomial $P(x) \in F[x]$ such that $P(x)$ has a root $\alpha \in L$. Write
\begin{align*}
P(x) = a_0 + a_1x + \cdots + a_nx^n
\end{align*}
for some elements $a_i \in F$ with $a_n \ne 0$ and some positive integer $n$. For each $0 \le i \le n$, write $a_i = \ulim_{s\in S}a_{i, s}$, where $a_{i, s} \in F_s$ for $\cD$-almost all $s \in S$, and set
\begin{align*}
P_s(x) = a_{0, s} + a_{1, s} + \cdots + a_{n, s}x^n \in F_s[x].
\end{align*}

It is straightforward to verify that $P(x) = \ulim_{s\in S}P_s(x)$. Since $P(x)$ is irreducible over $F = \prod_{s\in S}F_s/\cD$ and irreducibility is a first-order property (see \cite{FJ}), \L{}o\'s' theorem implies that $P_s(x)$ is irreducible over $F_s$ for $\cD$-almost all $s \in S$. 

Since $\alpha \in L = \prod_{s\in S}L_s/\cD$, $\alpha = \ulim_{s\in S}\alpha_s$, where $\alpha_s \in L_s$ for $\cD$-almost all $s \in S$. We see that
\begin{align*}
0 = P(\alpha) = \ulim_{s\in S}P_s(\alpha_s),
\end{align*}
which implies that $P_s(\alpha_s) = 0$ for $\cD$-almost all $s \in S$. Since $L_s$ is a normal extension of $F_s$ for $\cD$-almost all $s \in S$, $P_s(x)$ factors completely in $L_s[x]$ into linear factors, i.e.,
\begin{align*}
P_s(x) = a_{n, s}(x - \alpha_{1, s})\cdots (x - \alpha_{n, s}),
\end{align*}
where $\alpha_{i, s}$ are some elements in $L_s$ for $\cD$-almost all $s \in S$. Thus
\begin{align*}
P(x) &= \ulim_{s\in S}P_s(x) \\
&= \ulim_{s\in S}a_{n, s}(x - \ulim_{s\in S}\alpha_{1, s})\cdots (x - \ulim_{s\in S}\alpha_{n, s}) \\
&= a_n(x - \alpha_1)\cdots (x - \alpha_n),
\end{align*}
where $\alpha_i = \ulim_{s\in S}\alpha_{i, s} \in \prod_{s\in S}L_s/\cD = L$. Thus $P(x)$ factors completely in $L[x]$ into linear factors, and therefore $L$ is normal over $F$.

We now prove that $L$ is separable over $F$. By the primitive element theorem and since $L_s$ is separable over $F_s$, there exists an element $\beta_{s} \in L_s$ such that $L_s = F_s(\beta_{s})$ for $\cD$-almost all $s \in S$. Since $L_s$ is of degree over $d_s = d$ over $F_s$ for $\cD$-almost all $s \in S$, $\beta_{s}$ is of degree $d$ over $F_s$, and thus the minimal polynomial $f_s(x)$ of $\beta_{s}$ is of the form
\begin{align*}
f_s(x) = b_{0, s} + b_{1, s}x + \cdots + b_{d - 1, s}x^{d - 1} + x^d \in F_s[x],
\end{align*}
where $b_{i, s} \in F_s$ for $\cD$-almost all $s \in S$. For each $0 \le i \le d - 1$, setting $b_i = \ulim_{s\in S}b_{i, s} \in \prod_{s\in S}F_s/\cD = F$ and $f(x) = b_0 + b_1x + \cdots + b_{d - 1}x^{d - 1} + x^d \in F[x]$, we deduce that $f(x) = \ulim_{s\in S}f_s(x)$. By \L{}o\'s' theorem and since $f_s(x)$ is irreducible over $F_s$ for $\cD$-almost all $s \in S$, $f(x)$ is irreducible over $F$. Since $f_s(\beta_s) =0$, we deduce that $f(\beta) = 0$, where $\beta = \ulim_{s\in S}\beta_{s} \in L$, and thus $f(x)$ is the minimal polynomial of $\beta$ over $F$. Since $f$ is of degree $d$, the extension $F(\beta)$ is algebraic of degree $d$ over $F$, and thus $L = F(\beta)$ since $L$ contains $F(\beta)$.

Note that all the roots of $f$ belong to $L$ since $L$ is normal over $F$ and the root $\beta$ of $f$ belongs to $L$. 

In order to prove that $L$ is separable over $F$, it suffices to prove that $\beta$ is separable over $F$, which is equivalent to proving that $f$ is separable. Assume the contrary, i.e., $f(x)$ is not separable, and thus there exists a root $\lambda$ of $f$ in $L$ such that
\begin{align*}
f(x) = (x- \lambda)^eg(x),
\end{align*}
where $e \ge 2$ and $g(x) \in L[x]$. One can write $\lambda = \ulim_{s\in S}\lambda_s$ for some elements $\lambda_s \in L_s$ for $\cD$-almost all $s \in S$. Similarly, one can write $g(x) = \ulim_{s\in S}g_s(x)$ for some polynomials $g_s(x) \in L_s[x]$ for $\cD$-almost all $s \in S$. Thus
\begin{align*}
\ulim_{s\in S}f_s(x) = f(x) = \ulim_{s\in S}(x - \lambda_s)^eg_s(x),
\end{align*}
and thus
\begin{align*}
f_s(x) = (x - \lambda_s)^eg_s(x)
\end{align*}
for $\cD$-almost all $s \in S$. Therefore $f_s(x)$ is not separable, which is a contradiction since $f_s(x)$ is the minimal polynomial of $\beta_s$ and $L_s = F_s(\beta_s)$ is a separable extension of $F_s$. Therefore $L$ is separable over $F$, and thus $L$ is a finite Galois extension of degree $d$ over $F$. 

We now prove that the Galois group of $L$ over $K$ coincides with the ultra-Galois group of $L$ over $K$, i.e., $\Gal(L/K) = \Gal_{\ultra}(L/K)$. Since $\card(\Gal(L_s/K_s)) = n$ for $\cD$-almost all $s \in S$ and $\Gal_{\ultra}(L/K) = \prod_{s\in S}\Gal(L_s/K_s)/\cD$, we deduce from \cite[Lemma 3.7]{bell-slomson} that the cardinality of $\Gal_{\ultra}(L/K)$ equals $n$. 

By Lemma \ref{lem-ultra-Galois-group-is-a-subgroup-of-Galois-group}, $\Gal_{\ultra}(L/K)$ is a subset of $\Gal(L/K)$. Since $L$ is of degree $n$ over $F$, $\Gal(L/K)$ is of cardinality $n$, and thus $\Gal_{\ultra}(L/K) = \Gal(L/K)$ since both sets have the same finite cardinality.

\end{proof}

\subsection{$n$-th level hyperintegers}

In this subsection, we introduce a notion of $n$-th level hyperintegers that generalizes that of hyperintegers in Subsection \ref{subsec-hyperintegers}. Each $n$-th level hyperinteger represents a certain \textit{internal cardinality of an $n$-th level ultra-finite field} that we will introduce in a subsequence subsection.

\begin{definition}
\label{def-n-th-level-hyperintegers}
($n$-th level hyperintegers)

Let $S_1, S_2, \ldots,$ be a sequence of infinite sets such that each set $S_i$ is equipped with a nonprincipal ultrafilter $\cD_i$. For each integer $n \ge 0$, we define $n$-th level hyperintegers inductively as follows.

\begin{itemize}

\item [(0)] \textbf{$0$-th level hyperintegers} are the usual integers $0, \pm 1, \pm 2, \ldots$. So the set of $0$-th level hyperintegers is just the set of integers $\bZ$.

\item [(1)] \textbf{$1$st level hyperintegers} are elements in the ultrapower $\bZ^{\#} = \prod_{s\in S_1}\bZ/\cD_1$ of $\bZ$ with respect to $\cD_1$ that is defined in Subsection \ref{subsec-hyperintegers}. So the set of $1$st level hyperintegers is that of hyperintegers $\bZ^{\#}$.

    \item [(2)] Assume that for a positive integer $n \ge 1$, the set of $n$-th level hyperintegers has already been defined, and is denoted by $\bZ^{\#}_{(n)}$. Then \textbf{the set of $(n + 1)$-th level hyperintegers, denoted by $\bZ^{\#}_{(n + 1)}$, is defined to be the ultrapower $\prod_{s\in S_{n + 1}}\bZ^{\#}_{(n)}/\cD_{n + 1}$ of $\bZ^{\#}_{(n)}$.}

\end{itemize}

\end{definition}

\begin{remark}

\begin{itemize}

\item []

\item [(i)] In the definition above, the $S_i$ are not necessarily distinct, and accordingly, the $\cD_i$ are not necessarily distinct.

\item [(ii)] The definition of $\bZ^{\#}_{(n)}$ depends on a sequence of infinite sets $S_1, \ldots, S_n$ and a sequence of nonprincipal ultrafilters $\cD_1, \ldots, \cD_n$ on $S_1, \ldots, S_n$, respectively. Such sequences of sets and nonprincipal ultrafilters  will always be clear from the context, and thus in the notation of $\bZ^{\#}_{(n)}$, we only indicate the level $n$ without referring to neither the sets $S_1, \ldots, S_n$ nor the ultrafilters $\cD_1, \ldots, \cD_n$. 

\end{itemize}
\end{remark}

\subsection{$n$-th level ultra-finite fields}
\label{subsec-nth-level-ultra-finite-fields}

In this subsection, we introduce a class of fields that is analogous to finite fields, and they are constructed via ultraproducts.

\begin{definition}
\label{def-higher-dimensional-ultra-finite-fields}
($n$-th level ultra-finite fields)

Let $S_1, S_2, \ldots,$ be a sequence of infinite sets such that each set $S_i$ is equipped with a nonprincipal ultrafilter $\cD_i$. For each integer $n \ge 0$, we define $n$-th level ultra-finite fields inductively as follows.
\begin{itemize}

\item [(0)] \textbf{$0$-th level ultra-finite fields} are the usual finite fields $\bF_q$, where $q$ is a power of a prime number $p > 0$.

\item [(1)] \textbf{$1$st level ultra-finite fields} are nonprincipal ultraproducts of (not necessarily distinct) finite fields with respect to the nonprincipal ultrafilter $\cD$. More precisely, for $\cD$-almost all $s \in S_1$, if $F_{1, s}$ denotes a $0$-th level ultra-finite field, i.e., $F_{1, s}$ is a finite field, then the ultraproduct $\cF_1 = \prod_{s\in S_1}F_{0, s}/\cD_1$ of the $F_{0, s}$ with respect to $\cD_1$ is a $1$st level ultra-finite field.
    
\item [(2)] Assume that for a positive integer $n \ge 1$, $n$-th level ultra-finite fields have already been defined. Then \textbf{$(n + 1)$-th level ultra-finite fields} are defined to be nonprincipal ultraproducts of $n$-th level ultra-finite fields with respect to $\cD_{n + 1}$. More precisely, let $F_{n, s}$ denote an $n$-th level ultra-finite field for $\cD$-almost all $s \in S_{n + 1}$. Then the ultraproduct $\cF_{n + 1} = \prod_{s\in S_{n + 1}}F_{n, s}/\cD_{n + 1}$  of the $F_{n, s}$ with respect to $\cD_{n + 1}$ is an $(n + 1)$-th level ultra-finite field.

\end{itemize}

\end{definition}

\begin{remark}

\begin{itemize}

\item []

\item [(i)] For simplicity, $1$st level ultra-finite fields are also called \textbf{ultra-finite fields} without indicating the level.  Such fields are systematically introduced and studied in Ax \cite{ax-1968}, and are special examples of \textbf{pseudo-algebraically closed fields} (PAC). In Nguyen \cite{nguyen-APAL-2024}, ultra-finite fields are developed from an algebraic point of view, using ultra-field theory, in analogy with the classical theory of finite fields that is developed using field theory.

\item [(ii)] In the definition above, the $S_i$ are not necessarily distinct, and accordingly, the $\cD_i$ are not necessarily distinct.

\end{itemize}

\end{remark}

\begin{example}
\label{exam-n-th-level-ultra-finite-fields}

\begin{itemize}

\item []

\item [(i)] Finite fields $\bF_2, \bF_3, \bF_5, \ldots$ are $0$-th level ultra-finite fields.

\item [(ii)] Let $\bF_{q_s}$ be a finite field of $q_s$ elements for $\cD_1$-almost all $s \in S_1$. Then $\bF_{q_s}$ is a $0$-th level ultra-finite field for $\cD$-almost all $s \in S_1$. Then the ultraproduct $\cF = \prod_{s\in S_1}\bF_{q_s}/\cD_1$ is a ($1$st level) ultra-finite field.

\item [(iii)] For $r \in S_1$, $s \in S_2$, let $\bF_{q_{r, s}}$ be a finite field of $q_{r, s}$ elements. For each $s \in S_2$, the ultraproduct $\cF_s = \prod_{r\in S_1}\bF_{q_{r, s}}/\cD_1$ is a $1$st level ultra-finite field. Thus the ultraproduct $\cF = \prod_{s \in S_2}\cF_s/\cD_2$ is a $2$nd level ultra-finite field.

\end{itemize}

\end{example}

In order to be able to \textit{count} how many elements there are in an $n$-th level ultra-finite field, we introduce a notion of $n$-th level internal cardinality that is an analogue of the usual cardinality of a finite set. The  notion of $1$st level internal cardinality, or simply internal cardinality as known in literature is known in model theory (see Goldblatt \cite{Goldblatt}).

\begin{definition}
\label{def-higher-dimensional-internal-sets}
($n$-th level internal sets and internal cardinality)

Let $S_1, S_2, \ldots,$ be a sequence of infinite sets such that each set $S_i$ is equipped with a nonprincipal ultrafilter $\cD_i$. For each integer $n \ge 0$, se define $n$-th level internal sets and internal cardinality inductively as follows.

\begin{itemize}

\item [(0)] \textbf{$0$-th level internal sets} are the usual finite sets. The \textbf{$0$-th level internal cardinality of a $0$-th level internal set} is simply the usual cardinality of the set.

\item [(1)] \textbf{$1$st level internal sets} are nonprincipal ultraproducts of (not necessarily distinct) finite sets. More explicitly, let $A_{0, s}$ be a finite set for $\cD_1$-almost all $s \in S_1$. Then the ultraproduct $\cA_1 = \prod_{s\in S_1}A_{0, s}/\cD_1$ of the $A_{0, s}$ with respect to $\cD_1$ is a $1$st level internal set. The \textbf{$1$st level internal cardinality of $\cA_1$} is defined to be the ultraproduct of the cardinality of the $A_{0,s}$, i.e., $\ulim_{s\in S_1}\card(A_{0,s})/\cD_1$. We denote by $\icard_1(\cA_1)$ the $1$st level internal cardinality of $\cA_1$.

\item [(2)] Assume that for a positive integer $n \ge 1$, $n$-th level internal sets and the $n$-level internal cardinalities of $n$-th level internal sets have already been defined. We denote by $\icard_n(\cA)$ the $n$-th level internal cardinality of an $n$-th level internal set $\cA$. Then \textbf{$(n + 1)$-th level internal sets} are defined to be nonprincipal ultraproducts of $n$-th level internal sets. For an $(n + 1)$-th level internal set $\cA_{n + 1} = \prod_{s\in S_{n + 1}}A_{n, s}/\cD_{n + 1}$ with $A_{n, s}$ being $n$-th level internal set for $\cD_{n + 1}$-almost all $s \in S_{n + 1}$, the \textbf{$(n + 1)$-th internal cardinality of $\cA_{n + 1}$} is defined to be the nonprincipal ultraproduct of $\icard_n(A_{n, s})$ with respect to $\cD_{n + 1}$.

\end{itemize}

\end{definition}

\begin{proposition}
\label{prop-n-th-level-cardinality-is-an-nth-level-hyperinteger}

Let $\cA = \prod_{s\in S}A_s/\cD$ be an $n$-th level internal set for some integer $n \ge 1$, where $A_s$ is an $(n - 1)$-th level internal set for $\cD$-almost all $s \in S$. Then the $n$-th level internal cardinality $\icard_n(\cA)$ of $\cA$ is an $n$-th level hyperinteger in $\bZ_{(n)}^{\#}$.

\end{proposition}

\begin{proof}

We prove that $\icard_n(\cA)$ belongs to $\bZ_{(n)}^{\#}$ by induction on $n \ge 1$.

For $n = 1$, since $A_s$ is a $0$-th level internal set which is a finite set, $\card(A_s)$ is an integer for $\cD$-almost all $s \in S$, and thus $\icard_1(\cA) = \ulim_{s\in S}\card(A_s)$ is a $1$st level hyperinteger in $\bZ^{\#}$.

Suppose that the proposition holds for $n - 1$ for some integer $n \ge 2$.  Thus $\icard_{n - 1}(A_s)$ is an $(n - 1)$-th level hyperinteger in $\bZ^{\#}_{(n - 1)}$ for $\cD$-almost all $s \in S$. Thus $\icard_n(\cA) = \ulim_{s\in S}\icard_{n - 1}(A_s)$ is, by definition, an $n$-th level hyperinteger in $\bZ^{\#}_{(n)}$, as required.

\end{proof}

Using induction on $n$ and the above proposition, we obtain the following.

\begin{corollary}
\label{cor-cardinality-of-n-th-level-ultra-finite-field}

Let $\cF = \prod_{s\in S}\cF_s/\cD$ be an $n$-th level ultra-finite field for some integer $n \ge 0$, where $\cF_s$ is an $(n - 1)$-th level ultra-finite field for $\cD$-almost all $s \in S$. Then $\cF$ is an $n$-th level internal set, and the $n$-th level internal cardinality of $\cF$ is an $n$-th level hyperinteger in $\bZ^{\#}_{(n)}$.

\end{corollary}

Recall (see Serre \cite{serre-local-fields}) that a \textbf{quasi-finite field} is a perfect field $F$ such that $F$ has a unique extension $F_n$ of degree $n$ for each integer $n \ge 1$, and that the union of these extensions is equal to an algebraic closure of $F$.

In the language of fields, the property ``the field has exactly one algebraic extension of degree $n$ for every positive integer $n$" can be axiomatized (see \cite[Section 2]{Chatzidakis}). The property ``the field is perfect" can also be axiomatized, and thus using \L{}o\'s' theorem, we deduce the following well-known result.

\begin{proposition}
\label{prop-quasi-finite-for-ultra-fields}

An ultra-field $F = \prod_{s\in S}F_s/\cD$ with $F_s$ being a field for $\cD$-almost all $s \in S$, is quasi-finite if and only if $F_s$ is quasi-finite for $\cD$-almost all $s\in S$. 

\end{proposition}

\begin{remark}

Using the theory of ultra-fields, the author (see \cite[Corollary 3.28]{nguyen-APAL-2024}) provides another proof to the above proposition. 

\end{remark}

There is an explicit description of the unique algebraic extension of degree $d$ over an ultra-field that is quasi-finite.

\begin{lemma}
\label{lem-explicit-description-of-algebraic-extensions-of-quasi-ultra-fields}

Let $F = \prod_{s\in S}F_s/\cD$ be an ultra-field such that $F_s$ is a quasi-finite field for $\cD$-almost all $s \in S$. For each integer $d \ge 1$, let $F_s(d)$ denote the unique algebraic extension of degree $d$ over $F_s$ for $\cD$-almost all $s \in S$, and let $F(d)$ be the unique algebraic extension of degree $d$ over $F$. Then 
\begin{align*}
F(d) = \prod_{s\in S}F_s(d)/\cD.
\end{align*}

\end{lemma}

\begin{proof}

By Proposition \ref{prop-explicit-description--for--algebraic-extension-of-degree-d-of-ultra-fields}, $\prod_{s\in S}F_s(d)/\cD$ is an algebraic extension of degree $d$ over $F$. Since there is a unique algebraic extension $F(d)$ of degree $d$ over $F$, $F(d) = \prod_{s\in S}F_s(d)/\cD$ as required.

\end{proof}

\begin{corollary}
\label{cor-ultra-finite-fields-are-quasi-finite}

Every $n$-th level ultra-finite field is quasi-finite for every integer $n \ge 0$.

\end{corollary}

\begin{proof}

We prove the corollary by induction on $n$. When $n = 0$, $0$-th level ultra-finite fields are finite fields that are quasi-finite. Thus the corollary holds for $n = 0$.

Assume that every $(n - 1)$-th level ultra-finite field is quasi-finite for some integer $n \ge 1$. Take an arbitrary $n$-th level ultra-finite field $\cF = \prod_{s\in S}\cF_s/\cD$, where $\cF_s$ is an $(n - 1)$-th level ultra-finite field for $\cD$-almost all $s \in S$. By Proposition \ref{prop-quasi-finite-for-ultra-fields} and since $\cF_s$ is quasi-finite, $\cF$ is also quasi-finite as required.

\end{proof}

We end this section by showing that $n$-th level ultra-finite fields are very analogous to finite fields in terms of their descriptions using cardinality. 

We first extend the action of $\bZ^{\#}$ on ultra-finite fields in Subsection \ref{subsec-hyperintegers} to an action of $\bZ^{\#}_{(n)}$ on $n$-th level ultra-finite fields by induction on the level $n \ge 1$.

When $n = 1$, the action of $\bZ^{\#}$ on a $1$st level ultra-finite field is defined in the same way as in Subsection \ref{subsec-hyperintegers}. Assume that the action of $\bZ^{\#}_{(n - 1)}$ on $(n - 1)$-th level ultra-finite fields has already been define for some integer $n \ge 1$. Then for any $n$-th level hyperinteger $n = \ulim_{s\in S}n_s \in \bZ^{\#}_{(n)}$ with $n_s \in \bZ^{\#}_{(n- 1)}$ for $\cD$-almost all $s \in S$, and for any nonzero element $\alpha = \ulim_{s\in S}\alpha_s$ in an $n$-th level ultra-finite field $\cF_n = \prod_{s\in S}F_{n - 1, s}/\cD$, with $F_{n - 1, s}$ being an $(n - 1)$-th ultra-finite field and $\alpha_s \in F_{n - 1, s}$ for $\cD$-almost all $s \in S$, we define
\begin{align*}
\alpha^n := \ulim_{s\in S}\alpha_s^{n_s} \in \prod_{s\in S}F_{n - 1, s}/\cD = \cF_n.
\end{align*}

It is easy to verify that the above equation defines an action of $\bZ^{\#}_{(n)}$ on $n$-th level ultra-finite fields for all integers $n \ge 1$.

Recall that any element $\alpha$ in the finite field $\bF_q$ of $q$ elements satisfies the equation $\alpha^q - \alpha = 0$. We prove an analogue of this result for $n$-th level ultra-finite fields in which the cardinality $q$ of $\bF_q$ is replaced by the $n$-th level internal cardinality of an $n$-th level ultra-finite field.

\begin{proposition}
\label{prop-cardinality-description-of-n-th-level-ultra-finite-fields}

Let $\cF = \prod_{s\in S}\cF_s/\cD$ be an $n$-th level ultra-finite field for some integer $n \ge 1$, where $\cF_s$ is an $(n - 1)$-th level ultra-finite field for $\cD$-almost all $s \in S$. Let $\alpha \in \bZ^{\#}_{(n)}$ be the $n$-th level internal cardinality of $\cF$, i.e., 
\begin{align*}
\alpha = \icard_n(\cF) = \ulim_{s\in S}\icard_{n - 1}(\cF_s),
\end{align*}
where $\icard_{n - 1}(\cF_s)$ denotes the $(n - 1)$-th level internal cardinality of $\cF_s$ for $\cD$-almost all $s \in S$.

Then every element $a \in \cF$ satisfies the equation $a^{\alpha}  - a = 0$.

\end{proposition}

\begin{proof}

We prove the assertion by induction on the level $n \ge 1$.

When $n = 1$, $\cF$ is a ($1$st level) ultra-finite field, and the $1$st level internal cardinality $\beta$ of $\cF$ belongs to $\bZ^{\#}$. By Nguyen \cite[Lemma 3.42]{nguyen-APAL-2024}, it follows immediately that every element $a$ in $\cF$ satisfies the equation $a^{\beta} - a = 0$.

Suppose that the assertion holds for every $n$-th level ultra-finite field $\cG$ for some integer $n \ge 1$, i.e., every element $a$ in an $n$-th level ultra-finite field $\cG$ satisfies $a^{\alpha} - a = 0$, where $\alpha$ is the $n$-th level internal cardinality of $\cG$.

Take an arbitrary $(n + 1)$-level ultra-finite field $\cF = \prod_{s\in S}\cF_s/\cD$, where $\cF_s$ is an $n$-th level ultra-finite field for $\cD$-almost all $s \in S$. Let $\alpha$ be the $(n + 1)$-th level internal cardinality of $\cF$, and $\beta_s$ the $n$-th level internal cardinality of $\cF_s$ for $\cD$-almost all $s \in S$. By Proposition \ref{prop-n-th-level-cardinality-is-an-nth-level-hyperinteger}, $\alpha \in \bZ^{\#}_{(n + 1)}$ and $\beta_s \in \bZ^{\#}_{(n)}$ for $\cD$-almost all $s \in S$. 

Take an arbitrary element $a \in \cF$. One can write $a = \ulim_{s\in S}a_s$, where $a_s \in \cF_s$ for $\cD$-almost all $s \in S$. By induction hypothesis, $a_s^{\beta_s} - a_s = 0$ for $\cD$-almost all $s \in S$, and thus
\begin{align*}
a^{\alpha} = \ulim_{s\in S}a_s^{\beta_n} = \ulim_{s\in S}a_s = a,
\end{align*}
which verifies the assertion for the level $n + 1$. Thus the proposition follows immediately from the induction.

\end{proof}

\begin{remark}

Proposition \ref{prop-cardinality-description-of-n-th-level-ultra-finite-fields} will not be used in any other place in this paper. We present it here to point out analogies between finite fields and $n$-th level ultra-finite fields for all $n \ge 1$. For $n = 1$, a brief account of ultra-finite fields from an algebraic viewpoint was presented in Nguyen \cite[Section 3]{nguyen-APAL-2024}.

\end{remark}

\section{The algebraic part of an ultra-field extension of the ultra-hull $\cU(\F)$ over $\F$}
\label{sec-algebraic-part-of-an-ultra-finite-extension}

In this section, we develop the framework that permits a comparison between explicit class field theories for a family of rational function fields over constant fields $\bF_s$ and explicit class field theory for the rational function field over the ultraproduct $\fK$ of the constant fields $\bF_s$. The framework depends solely on two key notions that we introduce in this section, \textit{the algebraic part of an ultra-field extension} and \textit{the shadows and ultra-shadow of a finite separable extension of a given rational function field}.

We begin by introducing the notion of algebraic parts.

\subsection{Algebraic and separable parts of an ultra-finite extension}
\label{subsec-algebraic-parts}

Let $S$ be an infinite set, and let $\cD$ be a nonprincipal ultrafiler on $S$. For each $s \in S$, let $\bF_s$ be a field. Let $\fK = \prod_{s\in S}\bF_{s}/\cD$ denote the ultraproduct of $\bF_s$ with respect to $\cD$. If there exists a prime $p > 0$ such that $\{s \in S\; | \; \text{$\bF_s$ has characteristic $p$}\} \in \cD$, then $\fK$ has characteristic $p$; otherwise $\fK$ is of characteristic $0$.

 Let $\fA_s = \bF_s[t]$ be the polynomial ring in variable $t$ over $\bF_s$, and let $\fF_s = \bF_s(t)$ be the rational function field of variable $t$ over $\bF_s$ for $\cD$-almost all $s \in S$. 
 
 Let $\A = \fK[t]$ be the polynomial ring in variable $t$ over $\fK$, and let $\cU(\A)$ be the ultra-hull of $\A$, i.e., $\cU(\A)$ is the ultraproduct of $\fA_s = \bF_s[t]$ with respect to $\cD$ of the form $\cU(\A) = \prod_{s \in S}\fA_s/\cD$. Let $\F = \fK(t)$ denote the rational function field of variable $t$ over $\fK$, and let $\cU(\F)$ be the ultra-hull of $\F$, i.e., $\cU(\F)$ is the ultraproduct of $\fF_s = \bF_s(t)$ with respect to $\cD$ of the form $\cU(\F) = \prod_{s \in S}\fF_s/\cD$.

 The following are elementary but useful results that we will need in many places of this paper.

\begin{lemma}
  \label{lem-elementary-lemma}
  
  \begin{itemize}
  
  \item []
  
  \item [(i)] For every element $\alpha = \ulim_{s\in S}\alpha_s \in \cU(\A)$ for some elements $\alpha_s \in \fA_s = \bF_s[t]$, $\alpha$ belongs to $\A$ if and only if the degree of $\alpha_s$ is bounded for $\cD$-almost all $s \in S$, i.e., there exists a positive constant $M > 0$ such that
  \begin{align*}
  \{s \in S\; | \; \deg(\alpha_s) < M\} \in \cD.
  \end{align*} 
  
  \item [(ii)] Every polynomial $\alpha \in \A = \fK[t]$ can be written in the from $\alpha = \ulim_{s\in S}\alpha_s$, where $\alpha_s$ is a polynomial in $\fA_s = \bF_s[t]$ of the same degree as the polynomial $\alpha$ for $\cD$-almost all $s \in S$.
  
  \end{itemize}

\end{lemma}
 
 \begin{proof}

We first prove part (i).

Suppose that $\alpha \in \A = \fK[t]$, and thus $\alpha = a_0 + a_1t + \cdots + a_nt^n$ for some elements $a_i \in \fK$ with $a_n \ne 0$ and some integer $n \ge 0$. For each $0 \le i \le n$, write $a_i = \ulim_{s\in S}a_{i, s}$ for some elements $a_{i, s}\in \bF_s$ for $\cD$-almost all $s \in S$. Set
\begin{align*}
\beta_s = a_{0, s} + a_{1, s}t + \cdots + a_{n, s}t^n \in \fA_s = \bF_s[t].
\end{align*}
Since $a_n = \ulim_{s\in S}a_{n, s} \ne 0$, $a_{n, s} \ne 0$ for $\cD$-almost all $s \in S$, and thus 
\begin{align*}
\{s \in S\; | \; \deg(\beta_s) = n\} \in \cD.
\end{align*}

It is straightforward to verify that $\alpha = \ulim_{s\in S}\beta_s$. Since $\alpha = \ulim_{s\in S}\alpha_s = \ulim_{s\in S}\beta_s$, we deduce that $\alpha_s = \beta_s$ for $\cD$-almost all $s \in S$. Thus 
\begin{align*}
\{s \in S\; | \; \deg(\alpha_s) = n\} \in \cD,
\end{align*}
and therefore the degree of $\alpha_s$ is bounded for $\cD$-almost all $s \in S$.

Conversely, suppose that the degree of $\alpha_s$ is bounded for $\cD$-almost all $s \in S$, and let $N$ be a positive integer such that
\begin{align*}
\{s \in S \; | \; \deg(\alpha_s) \le N\} \in \cD.
\end{align*}
Then one can write
\begin{align*}
\alpha_s = b_{0, s} + a_{1, s}t + \cdots + b_{N, s}t^N \in \fA_s = \bF_s[t],
\end{align*}
 where $b_{i, s}\in \bF_s$. Thus
 \begin{align*}
 \alpha = \ulim_{s\in S}\alpha_s = b_0 + b_1t + \cdots + b_Nt^N \in \A = \fK[t],
 \end{align*}
 where $b_i = \ulim_{s\in S}b_{i, s} \in \fK$.

Part (ii) follows immediately from the same arguments as in part (i). 
 
 \end{proof}

\begin{lemma}
\label{lem-elementary-lemma0}

\begin{itemize}

\item []

\item [(i)] Let $P$ be a prime of degree $d \in \bZ_{>0}$ in $\A = \fK[t]$, i.e., $P$ is an irreducible polynomial of degree $d$ in $\A$. Then $P$ is a prime in the ultra-hull $\cU(\A)$ of $\A$, and $P = \ulim_{s\in S}P_s$, where $P_s$ is a prime of degree $d$ in $\fA_s = \bF_s[t]$ for $\cD$-almost all $s \in S$.

\item [(ii)] Let $P$ be a polynomial in $\A = \fK[t]$. By Lemma \ref{lem-elementary-lemma}, one can write $P = \ulim_{s\in S}P_s$, where $P_s$ is a polynomial in $\fA_s = \bF_s[t]$ of the same degree as $P$ for $\cD$-almost all $s \in S$. If $P_s$ is irreducible in $\fA_s$ for $\cD$-almost all $s \in S$, $P$ is irreducible in $\A$.

\end{itemize}

\end{lemma}

\begin{proof}

We first prove part (i).

By Nguyen \cite[Lemma 4.12, p.27]{nguyen-APAL-2024}, $P$ is a prime in $\cU(\A)$. One can write
\begin{align*}
P(t) = a_0 + a_1t + \cdots + a_dt^d,
\end{align*}
where the $a_i \in \fK$ such that $a_d \ne 0$. Since $\fK = \prod_{s\in S}\bF_s/\cD$ and $a_i \in \fK$, one can write $a_i = \ulim_{s\in S}a_{i, s}$ for some elements $a_{i, s}\in \bF_s$. Let
\begin{align*}
P_s(t) = a_{0, s} + a_{1, s}t + \cdots + a_{d, s}t^d \in \fA_s = \bF_s[t].
\end{align*}
It is clear that $P(t) = \ulim_{s\in S}P_s(t)$. On the other hand, since $a_d = \ulim_{s\in S}a_{d, s} \ne 0$, $a_{d, s} \ne 0$, and thus $P_s$ is of degree $d$ for $\cD$-almost all $s \in S$. 

Since $P$ is a prime in $\cU(\A) = \prod_{s\in S}\fA_s/\cD$ and $P = \ulim_{s\in S}P_s$, \L{}o\'s' theorem implies that $P_s$ is a prime in $\fA_s$ for $\cD$-almost all $s \in S$. Since $\cD$ is an ultrafilter, we deduce that $P_s$ is a prime of degree $d$ for $\cD$-almost all $s \in S$ as required.

We now prove part (ii). Assume the contrary, i.e., $P$ is reducible in $\A$, and thus $P(t) = A(t)B(t)$, where $A, B$ are elements in $\A$ such that $\deg(A), \deg(B) > 1$ such that $\deg(A) + \deg(B) = \deg(P)$. By Lemma \ref{lem-elementary-lemma}, one can write $A = \ulim_{s\in S}A_s$ and $B = \ulim_{s\in S}B_s$, where the $A_s, B_s$ are polynomials in $\fA_s$ of the same degrees as $A, B$, respectively for $\cD$-almost all $s \in S$. Then
\begin{align*}
\ulim_{s\in S}P_s = P = AB = \ulim_{s\in S}A_sB_s,
\end{align*}
which implies that $P_s = A_sB_s$ for $\cD$-almost all $s \in S$, a contradiction since $P_s$ is irreducible for $\cD$-almost all $s \in S$. Thus $P$ is irreducible in $\A$ as required.

\end{proof}

  For $\cD$-almost all $s \in S$, choose an algebraic closure $\fF_s^{\alg}$ and a separable closure $\fF_s^{\sep}$ inside the algebraic closure $\fF_s^{\alg}$. Let $\cU(\F)^{\alg}_{\ultra}$ denote an ultra-algebraic closure of $\cU(\F)$, i.e., 
  \begin{align*}
  	\cU(\F)^{\alg}_{\ultra} = \prod_{s\in S}\fF_s^{\alg}/\cD,
  \end{align*}
  and let $\cU(\F)^{\sep}_{\ultra} = \prod_{s\in S}\fF_s^{\sep}/\cD$ denote the ultra-separable closure of $\cU(\F)$. 
  
  Inside $\cU(\F)^{\alg}_{\ultra}$, we choose an algebraic closure $\F^{\alg}$ of $\F$, and a separable closure $\F^{\sep}$ of $\F$ inside $\F^{\alg}$. Since $\F \subset \cU(\F) \subset \cU(\F)^{\alg}_{\ultra}$ and $\cU(\F)^{\alg}_{\ultra}$ is an algebraically closed field, such an algebraic closure $\F^{\alg}$ can be chosen to be the set of all elements in $\cU(\F)^{\alg}_{\ultra}$ that are algebraic over $\F$ (see Lang \cite{lang-algebra} or Zariski--Samuel \cite{Zariski-Samuel}). We will prove in Corollary \ref{cor-F^sep-is-contained-in-U(F)^sep_ultra} that $\F^{\sep}$ is contained in $\cU(\F)^{\sep}_{\ultra}$.

At the constant field level, we choose an algebraic closure $\bF_s^{\alg}$ for $\cD$-almost all $s \in S$, and let $\bK = \prod_{s\in S}\bF_s^{\alg}/\cD$ be an ultra-algebraic closure of $\fK = \prod_{s\in S}\bF_s/\cD$. Let $\fK^{\alg}$ be the set of all elements in $\bK = \prod_{s\in S}\bF_s^{\alg}/\cD$ that are algebraic over $\fK$. Since $\bK$ is algebraically closed, it is known (see Lang \cite{lang-algebra} or Zariski--Samuel \cite{Zariski-Samuel}) that $\fK^{\alg}$ is an algebraic closure of $\fK$. Throughout this paper, we fix such an algebraic closure $\fK^{\alg}$ of $\fK$.

  Throughout this paper, we often work with an algebraic (possibly infinite) extension $\fL_s$  over $\fF_s$ for $\cD$-almost all $s \in S$, in which case, we always denote by $\fL$ an ultra-field extension of $\cU(\F)$ of the form 
  \begin{align*}
  	\fL = \prod_{s\in S}\fL_s/\cD.
  \end{align*}
  
  When $\fL_s$ is a finite extension of degree $n_s \in \bZ_{>0}$ over $\fF_s$ for $\cD$-almost all $s \in S$, $\fL$ is an ultra-field extension of ultra-degree $n = \ulim_{s \in S}n_s \in \bZ^{\#, >0}$ over $\cU(\F)$ (see Subsection \ref{subsubsec-ultra-fields}). A typical diagram of fields and ultra-fields used throughout this paper is illustrated in Figure \ref{fig-diagram-of-fields}.

  \begin{center}
  \begin{figure}
  	\label{fig-diagram-of-fields}
	\begin{tikzpicture}

	\node (Q-3) at (-5,-1) {$\cU(\A) = \prod_{s\in S}\fA_s/\cD$};

	\node (Q-1) at (-2,-2) {$\A = \fK[t]$};

    \node (Q1) at (0,0) {$\F = \fK(t)$};
     \node (Q2) at (2,2) {$\fL_{\sep} = \fL \cap \F^{\sep}$};
\node (Q4) at (3.5,3.5) {$\fL_{\alg} = \fL \cap \F^{\alg}$};
          \node (Q6) at (4.5,4.5) {$\F^{\sep}$};

     \node (Q8) at (6,6) {$\F^{\alg}$};
     
        \node (Q10) at (7.5, 7.5) {$\cU(\F)^{\alg}_{\ultra} = \prod_{s\in S}\fF_s^{\alg}/\cD$};
        
    \node (Q3) at (-3,1) {$\cU(\F) = \prod_{s\in S}\fF_s/\cD$};

    \node (Q5) at (0,4) {$\fL = \prod_{s\in S}\fL_s/\cD$};
     
        \node (Q7) at (2.5, 6.5) {$\cU(\F)^{\sep}_{\ultra} = \prod_{s\in S}\fF_s^{\sep}/\cD$};
        
        \draw (Q-1)--(Q-3) ;
\draw (Q-1)--(Q1)  ;
    \draw (Q1)--(Q2) ;
      \draw (Q2)--(Q4)  ;
    \draw (Q4)--(Q6)  ;
    \draw (Q6)--(Q8)  ;
\draw (Q8)--(Q10)  ;
    \draw (Q10)--(Q7)  ;

     \draw (Q6)--(Q7) ;

    \draw (Q1)--(Q3)  ;
    
   \draw (Q-3)--(Q3)  ;

    \draw (Q3)--(Q5)  ;

    \draw (Q5)--(Q2) ;
    \draw (Q5)--(Q4) ;
      \draw (Q5)--(Q7) ;
       \end{tikzpicture}

      \caption{Diagram of Fields and Ultra-fields}
      \end{figure}
    \end{center}

\begin{definition}
	\label{def-algebraic-part-of-ultra-finite-extension}
	(algebraic and separable parts)
	
	\begin{itemize}
		\item []
		\item [(i)] 	The fields $\fL_{\alg} = \fL \cap \F^{\alg}$ and $\fL_{\sep} = \fL \cap \F^{\sep}$ are called the \textbf{algebraic and separable parts of $\fL$ over $\F$}, respectively.

\item [(ii)] More generally, for any ultra-subset $\Omega = \prod_{s\in S}\Omega_s/\cD$ of $\fL$ for some subsets $\Omega_s$ of $\fL_s$, the sets $\Omega_{\alg} = \Omega \cap \F^{\alg}$ and $\Omega_{\sep} = \Omega \cap \F^{\sep}$ are called the \textbf{algebraic and separable parts of $\Omega$ over $\F$}, respectively.

	\end{itemize}

\end{definition}

\begin{remark}
	
	In most important applications in this paper, we only deal with separable extensions $\fL_s$ of $\fF_s$ for $\cD$-almost all $s \in S$, in which case $\fL_{\sep} = \fL_{\alg}$. 
	
\end{remark}

We prove a key lemma that we will need in many places of this paper. In fact, this is the most important lemma that permits moving back and forth between Galois extensions of $\F$ and Galois extensions of the family of function fields $\fF_s$, which in turn permits a comparison between explicit class field theories for $\fF_s$ and for $\F$. 

Roughly speaking, the proof of the following lemma shows that the shape of the Newton polygon for the $\infty$-adic completion of $\F$ at the infinite prime $\infty = 1/t$ is similar to the shape of the Newton polygon for $\infty$-adic completion of $\fF_s$ at the infinite prime $\infty$ for $\cD$-almost all $s \in S$.

\begin{lemma}
	\label{rem-irreducibility-is-the-same-in-F-and-U(F)}

    A polynomial $P$ in $\F[x]$ is irreducible over $\F$ if and only if $P$ is irreducible over its ultra-hall $\cU(\F)$.

\end{lemma}

\begin{proof}

Since $\cU(\F)$ contains $\F$, it is obvious that if $P$ is irreducible over $\cU(\F)$, then $P$ is irreducible over $\F$.

Suppose that $P(x) = a_0 + a_1x + \cdots + a_nx^n \in \F[x]$ is irreducible over $\F$, where the $a_i$ belong to $\F$ and $a_n \ne 0$. If $\deg(P) = n = 0$ or $\deg(P) = n = 1$, then it is obvious that $P(x)$ is also irreducible over $\cU(\F)$. Thus without loss of generality, we assume that $n > 1$. 

Note that $P(x)$ is irreducible over $\F$ if and only if $\alpha P(x)$ is irreducible over $\F$ for every nonzero element $\alpha \in \F$. Thus, by multiplying a suitable element in $\A = \fK[t]$, we can assume that the $a_i$ belong to $\A = \fK[t]$ such that $\gcd(a_0, \ldots, a_n) = 1$, i.e., $P(x)$ is a primitive polynomial in $\A[x]$ such that $P(x)$ is irreducible over $\F$.

For each $0 \le i \le n$, one can, using Lemma \ref{lem-elementary-lemma}, write $a_i = \ulim_{s\in S}a_{i, s}$ for some elements $a_{i, s} \in \fA_s = \bF_s[t]$. Set
\begin{align*}
M = \max\{\deg(a_i) : 0 \le i \le n\}.
\end{align*}
We see that for all $0 \le i \le n$,
\begin{align}
\label{e1-lem-irreducibility-is-the-same-for-F-and-U(F)}
\deg(a_{i, s}) \le M
\end{align}
for $\cD$-almost all $s \in S$.

For each $s \in S$, let 
\begin{align*}
P_s(x) = a_{0, s} + a_{1, s}x + \cdots + a_{n, s}x^n \in \fA_s[x].
\end{align*}

It is clear that $P(x) = \ulim_{s\in S}P_s(x)$. Since $a_n \ne 0$ and $a_n = \ulim_{s\in S}a_{n, s}$, \L{}o\'s' theorem implies that $a_{n, s}\ne 0$ for $\cD$-almost all $s \in S$, and thus $P_s(x)$ is of degree $n$ for $\cD$-almost all $s \in S$.

Since $\gcd(a_0, \ldots, a_n) = 1$ and $\A$ is a principal ideal domain, there exist elements $\beta_0, \ldots, \beta_n$ in $\A = \fK[t]$ such that
\begin{align}
\label{e2-lem-irreducibility-is-the-same-for-F-and-U(F)}
\beta_0a_0 + \beta_1a_1 + \cdots + \beta_na_n = 1.
\end{align}

For each $0 \le i \le n$, one can write $\beta_i = \ulim_{s\in S}\beta_{i, s}$ for some elements $\beta_{i, s} \in \fA_s = \bF_s[t]$. We see from (\ref{e2-lem-irreducibility-is-the-same-for-F-and-U(F)}) that
\begin{align*}
\ulim_{s\in S}(\beta_{0, s}a_{0, s} + \beta_{1, s}a_{1, s} + \cdots + \beta_{n, s}a_{n, s}) = 1,
\end{align*}
and it thus follows from \L{}o\'s' theorem that
\begin{align*}
\beta_{0, s}a_{0, s} + \beta_{1, s}a_{1, s} + \cdots + \beta_{n, s}a_{n, s} = 1
\end{align*}
for $\cD$-almost all $s \in S$, i.e., $\gcd(a_{0, s}, \ldots, a_{n, s}) = 1$ for $\cD$-almost all $s \in S$. Thus $P_s(x)$ is a primitive polynomial in $\fA_s[x]$ for $\cD$-almost all $s \in S$.

We are now ready to prove that $P(x)$ is irreducible over $\cU(\F)$. Assume the contrary, i.e., $P(x)$ is reducible over $\cU(\F)$. Since $P(x) = \ulim_{s\in S}P_s(x)$ and $\cU(\F) = \prod_{s\in S}\fF_s/\cD$, \L{}o\'s' theorem implies that $P_s(x)$ is reducible over $\fF_s$ for $\cD$-almost all $s \in S$. Since $P_s(x)$ is primitive in $\fA_s[x]$, and reducible over $\fF_s$, and $\fA_s = \bF_s[t]$ is a principal ideal domain with field of fractions $\fF_s$ for $\cD$-almost all $s \in S$, Gauss' lemma (see Lang \cite{lang-algebra}) implies that $P_s(x)$ is reducible in $\fA_s[x]$ for $\cD$-almost all $s \in S$. Thus there exist polynomials $F_s(x), G_s(x) \in \fA_s[x]$ that satisfy the following:

\begin{description}[style=multiline, labelwidth=1.5cm]

\item [\namedlabel{itm:C1'}{C1'}] $F_s(x) = c_{0, s} + c_{1, s}x + \cdots + c_{h_s, s}x^{h_s}$ and $G_s(x) = d_{0, s} + d_{1, s}x + \cdots + d_{m_s, s}x^{m_s}$, where the $c_{i, s}, d_{j, s}$ belong to $\fA_s$ such that $c_{h_s, s}, d_{m_s, s} \ne 0$, and $1 \le h_s, m_s \le n - 1$ such that $n = h_s + m_s$; and

\item [\namedlabel{itm:C2'}{C2'}] $P_s(x) = F_s(x)G_s(x)$ for $\cD$-almost all $s \in S$.

\end{description}

Since $1 \le h_s, m_s \le n -1$ for $\cD$-almost all $s \in S$, we see that if $A_{h, m} = \{s \in S\; | \; \text{$\deg(F_s) = h_s = h$ and $\deg(G_s) = m_s = m$}\}$ for all $1 \le h, m \le n - 1$ with $h + m = n$, then
\begin{align*}
\bigcup_{\substack{1 \le h, m \le n ,\\ h + m = n}}A_{h, m} \in \cD
\end{align*}

Since the above union only has finitely many sets $A_{h, m}$, Lemma \ref{lem-at-least-one-set-in-th-union-of-sets-is-in-D} implies that there exist integers $1 \le h, m \le n - 1$ such that 
\begin{align*}
A_{h, m} =  \{s \in S\; | \; \text{$\deg(F_s) = h_s = h$ and $\deg(G_s) = m_s = m$}\} \in \cD,
\end{align*}
and thus conditions (\ref{itm:C1'}) and (\ref{itm:C2'}) implies that there exist polynomials $F_s(x), G_s(x) \in \fA_s[x]$ such that the following are true.

\begin{description}[style=multiline, labelwidth=1.5cm]

\item [\namedlabel{itm:C1}{C1}] $F_s(x) = c_{0, s} + c_{1, s}x + \cdots + c_{h, s}x^{h}$ and $G_s(x) = d_{0, s} + d_{1, s}x + \cdots + d_{m, s}x^{m}$, where the $c_{i, s}, d_{j, s}$ belong to $\fA_s$ such that $c_{h, s}, d_{m, s} \ne 0$, and $1 \le h, m \le n - 1$ such that $n = h + m$; and

\item [\namedlabel{itm:C2}{C2}] $P_s(x) = F_s(x)G_s(x)$ for $\cD$-almost all $s \in S$.

\end{description}

We contend that there exists an absolute number $N$ (only depending on $M$ and $n$) such that
\begin{align}
\label{e2-1/2-lem-irreducibility-is-the-same-for-F-and-U(F)}
\deg(c_{i, s}), \deg(d_{j, s}) < N
\end{align}
for all $0 \le i \le h$, $0 \le j \le m$ for $\cD$-almost all $s \in S$. 

Let $\fS_s$ be the splitting field of $P_s(x)$ over $\fF_s$ that is contained in $\fF_s^{\alg}$. Since $\deg(P_s) = n$ for $\cD$-almost all $s \in S$, the degree of $\fS_s$ over $\fF_s$ is at most $n!$. Let $\infty = 1/t$ denote the infinite prime of $\fF_s = \bF_s(t)$ \footnote{By abuse of notation and for simplicity, we use the same symbol $\infty$ for the infinite prime of \textit{all} rational function fields $\fF_s$.}, and let $v_{\infty}$ denote the normalized valuation of $\fF_s$ at $\infty$, that is, 
\begin{align*}
v_{\infty}(\alpha) = -\deg(\alpha)
\end{align*}
for each nonzero polynomial $\alpha \in \fF_s = \bF_s(t)$, and $v_{\infty}(0) = \infty$. Let $w_{\infty}$ denote the normalized valuation of $\fS_s$ that extends $v_{\infty}$, and let $\fe(\fS_s/\fF_s)$ denote the ramification index of $\infty$ in $\fS_s/\fF_s$, i.e.,
\begin{align*}
w_{\infty}(\alpha) = \fe(\fS_s/\fF_s)v_{\infty}(\alpha) = -\fe(\fS_s/\fF_s)\deg(\alpha)
\end{align*}
for every nonzero polynomial $\alpha \in \fF_s = \bF_s(t)$. A simple estimate gives
\begin{align}
\label{e3-lem-irreducibility-is-the-same-for-F-and-U(F)}
1 \le \fe(\fS_s/\fF_s) \le [\fS_s : \fF_s] \le n!
\end{align}

We now estimate the valuations of roots of $P_s(x)$ in $\fS_s$. We first note that $P_s(0) \ne 0$ for $\cD$-almost all $s \in S$; otherwise, $P_s(0) = 0$ for $\cD$-almost all $s \in S$, and thus
\begin{align*}
P(0) = \ulim_{s\in S}P_s(0) = 0.
\end{align*}
Since $\deg(P) = n \ge 2$, we deduce that $P(x) = xQ(x)$, where $Q(x)$ is a polynomial in $\F[x]$ such that $\deg(Q) = \deg(P) - 1 \ge 1$. Thus $P(x)$ is reducible over $\F$, which is a contradiction. Therefore $x = 0$ is not a root of $P_s(x)$ for $\cD$-almost all $s \in S$.

Let $\alpha$ be an arbitrary root of $P_s(x)$ in $\fS_s$. By assumption, we know that $\alpha \ne 0$. Then
\begin{align*}
P_s(\alpha) = a_{0, s} + a_{1, s}\alpha + \cdots + a_{n, s}\alpha^n = 0.
\end{align*}
If there exists a unique minimum $w_{\infty}(a_{i, s}\alpha^i)$ among all the valuations $w_{\infty}(a_{j, s}\alpha^j)$ with $0 \le j \le n$, then the above equation implies that
\begin{align*}
\infty = w_{\infty}(0) = w_{\infty}(P_s(\alpha)) = \min\{w_{\infty}(a_{j, s}\alpha^j) \; |\; 0 \le j \le n\} = w_{\infty}(a_{i, s}\alpha^i),
\end{align*}
which is a contradiction. Thus there exist integers $0 \le i < j \le n$ such that
\begin{align*}
w_{\infty}(a_{i, s}\alpha^i) = w_{\infty}(a_{j, s}\alpha^j)
\end{align*}
and thus
\begin{align*}
w_{\infty}(\alpha) = \dfrac{w_{\infty}(a_{i, s}) - w_{\infty}(a_{j, s})}{j - i} = \dfrac{\fe(\fS_s/\fF_s)(\deg(a_{j, s}) - \deg(a_{i, s}))}{j - i}
\end{align*}

Since $\alpha \ne 0$, the above equation implies that both $a_{i, s}$ and $a_{j, s}$ are nonzero. Thus (\ref{e1-lem-irreducibility-is-the-same-for-F-and-U(F)}) implies that
\begin{align*}
\deg(a_{j, s}) - \deg(a_{i, s}) \ge 0 - M = -M.
\end{align*}
Since $0 < j - i \le n$, we deduce from (\ref{e3-lem-irreducibility-is-the-same-for-F-and-U(F)}) that
\begin{align}
\label{e4-lem-irreduciblity-is-the-same-for-F-and-U(F)}
w_{\infty}(\alpha) \ge -\fe(\fS_s/\fF_s)\left(\dfrac{M}{n}\right).
\end{align}
for an arbitrary root $\alpha$ of $P_s(x)$ in the splitting field $\fS_s$. 

We now prove (\ref{e2-1/2-lem-irreducibility-is-the-same-for-F-and-U(F)}). By (\ref{itm:C2}), and by comparing the coefficient of $x^n$, we see that $a_{n, s} = c_{h, s}d_{m, s}$. By (\ref{e1-lem-irreducibility-is-the-same-for-F-and-U(F)}), $\deg(a_{n, s}) \le M$, and thus since $\deg(c_{h, s}) + \deg(d_{m, s}) = \deg(a_{n, s})$ and $c_{h, s} \ne 0$, $d_{m, s} \ne 0$ for $\cD$-almost all $s \in S$, we deduce that 
\begin{align}
\label{e5.1-lem-irreduciblity-is-the-same-for-F-and-U(F)-leading-coefficients}
\text{$0 \le \deg(c_{h, s}) \le M$ and $0 \le \deg(d_{m, s}) \le M$}.
\end{align}

By (\ref{itm:C2}), all the roots of $F_s$ and $G_s$ in the algebraic closure $\fF_s^{\alg}$ are roots of $P_s$, and conversely every root of $P_s$ is a root of $F_s$ or $G_s$. Thus all the roots of $F_s$ and $G_s$ belong to the splitting field $\fS_s$ of $P_s$. Let $\alpha_1, \ldots, \alpha_h$ be the roots of $F_s$ in $\fS_s$, and let $\beta_1, \ldots, \beta_m$ be the roots of $G_s$ in $\fS_s$.  

We first consider the coefficients of $F_s$. By Vi\'ete's theorem, for any $1 \le k \le h$, 
\begin{align*}
\sum_{1 \le i_1 < i_2 < \cdots < i_k \le h}\alpha_{i_1}\alpha_{i_2}\cdots\alpha_{i_k} = (-1)^k\dfrac{c_{h - k, s}}{c_{h, s}}
\end{align*}
Thus since $w_{\infty}((-1)^k) = 0$, we deduce that
\begin{align}
\label{e5.2-lem-irreduciblity-is-the-same-for-F-and-U(F)-leading-coefficients}
w_{\infty}(c_{h - k, s}(-1)^k) &= w_{\infty}(c_{h - k, s})  \nonumber \\
&= w_{\infty}\left(c_{h, s}\sum_{1 \le i_1 < i_2 < \cdots < i_k \le h}\alpha_{i_1}\alpha_{i_2}\cdots\alpha_{i_k}   \right) \nonumber \\
&= w_{\infty}(c_{h, s}) +  w_{\infty}\left(\sum_{1 \le i_1 < i_2 < \cdots < i_k \le h}\alpha_{i_1}\alpha_{i_2}\cdots\alpha_{i_k}   \right) \nonumber \\
&= -\fe(\fS_s/\fF_s)\deg(c_{h, s}) + w_{\infty}\left(\sum_{1 \le i_1 < i_2 < \cdots < i_k \le h}\alpha_{i_1}\alpha_{i_2}\cdots\alpha_{i_k}   \right)
\end{align}

By (\ref{e4-lem-irreduciblity-is-the-same-for-F-and-U(F)}), we deduce that for any integers $1 \le i_1 < i_2 < \cdots < i_k \le h$, 
\begin{align*}
w_{\infty}\left(\alpha_{i_1}\alpha_{i_2}\cdots\alpha_{i_k}\right) &= w_{\infty}(\alpha_{i_1}) + \cdots + w_{\infty}(\alpha_{i_k}) \nonumber \\
&\ge \underbrace{\left(-\fe(\fS_s/\fF_s)\left(\dfrac{M}{n}\right)\right) + \cdots + \left(-\fe(\fS_s/\fF_s)\left(\dfrac{M}{n}\right)\right)}_{\text{$k$ copies}} \nonumber \\
&= -\fe(\fS_s/\fF_s)\left(\dfrac{kM}{n}\right),
\end{align*}
and thus
\begin{align}
\label{e5.3-lem-irreduciblity-is-the-same-for-F-and-U(F)-leading-coefficients}
w_{\infty}\left(\sum_{1 \le i_1 < i_2 < \cdots < i_k \le h}\alpha_{i_1}\alpha_{i_2}\cdots\alpha_{i_k}   \right) &\ge \min\{w_{\infty}\left(\alpha_{i_1}\alpha_{i_2}\cdots\alpha_{i_k}\right)\; |\; 1 \le i_1 < i_2 < \cdots < i_k \le h \} \nonumber \\
&\ge -\fe(\fS_s/\fF_s)\left(\dfrac{kM}{n}\right).
\end{align}
By (\ref{e5.1-lem-irreduciblity-is-the-same-for-F-and-U(F)-leading-coefficients}), (\ref{e5.2-lem-irreduciblity-is-the-same-for-F-and-U(F)-leading-coefficients}), (\ref{e5.3-lem-irreduciblity-is-the-same-for-F-and-U(F)-leading-coefficients}), and since $\fe(\fS_s/\fF_s) \ge 1$, we deduce that
\begin{align*}
w_{\infty}(c_{h - k, s}) \ge -\fe(\fS_s/\fF_s)M  -\fe(\fS_s/\fF_s)\left(\dfrac{kM}{n}\right)
\end{align*}
for any $1 \le k \le h$. Since $w_{\infty}(c_{h - k, s}) = -\fe(\fS_s/\fF_s)\deg(c_{h - k, s})$, we see that
\begin{align*}
-\fe(\fS_s/\fF_s)\deg(c_{h - k, s}) \ge -\fe(\fS_s/\fF_s)M  -\fe(\fS_s/\fF_s)\left(\dfrac{kM}{n}\right),
\end{align*}
Since $\fe(\fS_s/\fF_s) \ge 1$, it follows from the above inequality that
\begin{align*}
\deg(c_{h - k, s}) \le M + \dfrac{kM}{n} = \dfrac{M(k + n)}{n}
\end{align*}
for every $1 \le k \le h$. In particular, since $k \le h \le n - 1$ (see (\ref{itm:C1}), we obtain
\begin{align}
\label{e6-lem-irreduciblity-is-the-same-for-F-and-U(F)}
\deg(c_{h - k, s}) \le \dfrac{M(h + n)}{n} \le \dfrac{M(2n - 1)}{n}
\end{align}
for all $1 \le k \le h$ for $\cD$-almost all $s \in S$.

For each $0 \le i \le h$, set
\begin{align}
\label{e7-lem-irreduciblity-is-the-same-for-F-and-U(F)}
c_i = \ulim_{s \in S}c_{i, s} \in \prod_{s\in S}\fA_s/\cD = \cU(\A).
\end{align}
By (\ref{e5.1-lem-irreduciblity-is-the-same-for-F-and-U(F)-leading-coefficients}), (\ref{e6-lem-irreduciblity-is-the-same-for-F-and-U(F)}), we deduce from Lemma \ref{lem-elementary-lemma} that $c_i \in \A = \fK[t]$ for all $0 \le i \le h$.

Repeating the same arguments as above, we can show that 
\begin{align*}
d_{m - k, s} \le \dfrac{M(m + n)}{n} \le \dfrac{M(2n - 1)}{n}
\end{align*}
for all $1 \le k \le m$ for $\cD$-almost all $s \in S$. Thus in setting 
\begin{align}
\label{e8-lem-irreduciblity-is-the-same-for-F-and-U(F)}
d_j = \ulim_{s \in S}d_{j, s} \in \cU(\A)
\end{align}
for each $0 \le j \le m$, we deduce from (\ref{e5.1-lem-irreduciblity-is-the-same-for-F-and-U(F)-leading-coefficients}) and (\ref{e6-lem-irreduciblity-is-the-same-for-F-and-U(F)}) that $d_j \in \A = \fK[t]$ for all $0 \le j \le m$.

Let 
\begin{align*}
F(x) = c_0 + c_1x + \cdots + c_hx^h \in \A[x]
\end{align*}
and
\begin{align*}
G(x) = d_0 + d_1x + \cdots + d_mx^m \in \A[x].
\end{align*}

We see from (\ref{e7-lem-irreduciblity-is-the-same-for-F-and-U(F)}) and (\ref{e8-lem-irreduciblity-is-the-same-for-F-and-U(F)}) that
\begin{align*}
F(x) = \ulim_{s\in S}F_s(x)
\end{align*}
and 
\begin{align*}
G(x) = \ulim_{s\in S}G_s(x).
\end{align*}
Therefore it follows from (\ref{itm:C2}) that
\begin{align*}
P(x) = \ulim_{s\in S}P_s(x) = \ulim_{s\in S}F_s(x)G_s(x) = F(x)G(x),
\end{align*}
which implies that $P(x)$ is reducible in $\A[x]$. Since $P(x)$ is primitive, Gauss' lemma (see Lang \cite{lang-algebra}) implies that $P(x)$ is reducible in $\F[x]$, a contradiction. Thus $P(x)$ is irreducible over $\cU(\F)$.

\end{proof}

The following elementary result will be useful in many places in this paper.

\begin{lemma}
 \label{lem-elementary-lemma1}

Let $P(x) \in \A[x]$ be a polynomial of degree $m \in \bZ_{\ge 0}$ over $\A$ of the form 
 \begin{align*}
 P(x) = a_0 + a_1x + \cdots + a_mx^m,
 \end{align*}
 where the $a_i$ belong to $\A$ such that $a_m \ne 0$. Then 
 \begin{itemize}
 
 \item [(i)]
 \begin{align*}
 P(x) = \ulim_{s\in S}P_s(x),
 \end{align*}
 where
 \begin{align*}
 P_s(x) = a_{0, s} + a_{1, s}x + \cdots + a_{m, s}x^m \in \fA_s[x],
 \end{align*}
 and the $a_{i, s}$ are elements in $\fA_s = \bF_s[t]$ such that
 \begin{align*}
 a_i = \ulim_{s \in S}a_{i, s}
 \end{align*}
 for all $0 \le i \le m$, and $a_{m, s} \ne 0$ for $\cD$-almost all $s \in S$. 
 
 \item [(ii)]  $P(x)$ is irreducible over $\F$ if and only if $P_s(x)$ is irreducible over $\fF_s$ for $\cD$-almost all $s \in S$.

 \item [(iii)] if $P(x)$ is primitive, then $P_s(x)$ is primitive for $\cD$-almost all $s \in S$. 
     
 \end{itemize}

 \end{lemma}
 
 \begin{proof}
 
It is straightforward to verify that $P(x) = \ulim_{s\in S}P_s(x)$, where the $P_s(x)$ are defined as in part (i). 

Since $a_m = \ulim_{s\in S}a_{m, s} \ne 0$, we deduce that $a_{m, s} \ne 0$ for $\cD$-almost all $s \in S$, which proves part (i).

By \L{}o\'s' theorem, $P(x)$ is irreducible over $\cU(\F) = \prod_{s\in S}\fF_s/\cD$ if and only if $P_s(x)$ is irreducible over $\fF_s$ for $\cD$-almost all $s \in S$. Thus part (ii) follows immediately from Lemma \ref{rem-irreducibility-is-the-same-in-F-and-U(F)}.

We now prove part (iii). Suppose that $P(x)$ is primitive, i.e., $\gcd(a_0, a_1, \ldots, a_m) = 1$. Since $\A = \fK[t]$ is a principal ideal domain, B\'ezout's identity implies that there exist elements $b_0, b_1, \ldots, b_m \in \A$ such that
\begin{align}
\label{e1-in-elementary-lemma1}
b_0a_0 + b_1a_1 + \cdots + b_ma_m = 1.
\end{align}

Using Lemma \ref{lem-elementary-lemma}, one can write $b_i = \ulim_{s\in S}b_{i, s}$ for each $0 \le i \le m$, where $b_{i, s} \in \fA_s = \bF_s[t]$. It thus follows from (\ref{e1-in-elementary-lemma1}) that
\begin{align*}
1 = b_0a_0 + b_1a_1 + \cdots + b_ma_m = \ulim_{s\in S}(b_{0, s}a_{0, s} + \cdots + b_{m, s}a_{m, s}),
\end{align*}
and therefore
\begin{align*}
b_{0, s}a_{0, s} + \cdots + b_{m, s}a_{m, s} = 1
\end{align*}
for $\cD$-almost all $s \in S$. Hence $\gcd(a_{0, s}, \ldots, a_{m, s}) = 1$ for $\cD$-almost all $s \in S$, and so $P_s(x)$ is primitive for $\cD$-almost all $s \in S$.

 \end{proof}

  \begin{corollary}
\label{cor-F^sep-is-contained-in-U(F)^sep_ultra}

$\F^{\sep}$ is contained in $\cU(\F)^{\sep}_{\ultra}$.

\end{corollary}

\begin{proof}

Take an arbitrary element $\alpha \in \F^{\sep}$.  Let $f(x) = a_0 + a_1x + \cdots + a_{m - 1}x^{m - 1} + x^m \in \F[x]$ be the minimal polynomial of $\alpha$ over $\F$, where the $a_i$ belong to $\F$. Since $\alpha \in \F^{\sep}$, $f(x)$ is a separable polynomial over $\F$. 

For each $0 \le i \le m - 1$, since $a_i \in \F \subset \cU(\F) = \prod_{s\in S}\fF_s/\cD$, we can write $a_i = \ulim_{s\in S}a_{i, s}$, where the $a_{i, s}$ belong to $\fF_s$ for $\cD$-almost all $s \in S$. Set
\begin{align*}
f_s(x) = a_{0, s} + a_{1, s}x + \cdots + a_{m - 1, s}x^{m - 1} + x^m \in \fF_s[x].
\end{align*}
It is straightforward to verify that $f(x) = \ulim_{s\in S}f_s(x)$. Since $f(x)$ is irreducible over $\F$, Lemma \ref{lem-elementary-lemma1} implies that $f_s(x)$ is irreducible over $\fF_s$ for $\cD$-almost all $s \in S$. 

Since 
\begin{align*}
\F^{\sep} \subset \F^{\alg} \subset \cU(\F)^{\alg}_{\ultra} = \prod_{s\in S}\fF_s^{\alg}/\cD,
\end{align*}
we can write $\alpha = \ulim_{s\in S}\alpha_s$, where the $\alpha_s$ belong to $\fF_s^{\alg}$ for $\cD$-almost all $s \in S$. Since
\begin{align*}
0 = f(\alpha) = \ulim_{s\in S}f_s(\alpha_s),
\end{align*}
we deduce that $f_s(\alpha_s) = 0$ for $\cD$-almost all $s \in S$, and thus $f_s(x)$ is the minimal polynomial of $\alpha_s$ for $\cD$-almost all $s \in S$. 

Since $\alpha \in \F^{\sep}$, $f(x)$ is a separable polynomial over $\F$. Since $\deg(f) = m$, there exist exactly $m - 1$ distinct elements $\alpha_2, \ldots, \alpha_m \in \F^{\sep}$ such that $\alpha \ne \alpha_i$ for all $2 \le i \le m$ and 
\begin{align*}
f(x) = (x - \alpha)(x-\alpha_2)\cdots (x - \alpha_m).
\end{align*}

Since $\alpha_i \in \F^{\sep}$ for all $2 \le i \le m$, using the same arguments as above, one can write $\alpha_i = \ulim_{s\in S}\alpha_{i, s}$, where $\alpha_{i, s}\in \fF_s^{\alg}$ for $\cD$-almost all $s \in S$. Thus we deduce that
\begin{align*}
\ulim_{s\in S}f_s(x)  = f(x) &= (x - \alpha)(x-\alpha_2)\cdots (x - \alpha_m) \\
&= \ulim_{s\in S}(x - \alpha_s)(x - \alpha_{2, s})\cdots (x - \alpha_{m, s}),
\end{align*}
and therefore
\begin{align}
\label{e1-cor-F^sep-is-contained-in-U(F)^sep_ultra}
f_s(x) = (x - \alpha_s)(x - \alpha_{2, s})\cdots (x - \alpha_{m, s})
\end{align}
for $\cD$-almost all $s \in S$. 

Since $\alpha\ne \alpha_j$ for all $2 \le j \le m$, we see that for all $2 \le j \le m$, $\alpha_s \ne \alpha_{j, s}$ for $\cD$-almost all $s \in S$, i.e.,
\begin{align*}
A_{1, j} = \{s \in S\; | \; \alpha_s \ne \alpha_{j, s}\} \in \cD.
\end{align*}
Similarly, since  $\alpha_i \ne \alpha_j$ for all $2 \le i < j\le m$, we see that for all $2 \le i < j \le m$, $\alpha_{i, s} \ne \alpha_{j, s}$ for $\cD$-almost all $s \in S$, i.e.,
\begin{align*}
A_{i, j} = \{s \in S\; | \; \alpha_{i, s} \ne \alpha_{j, s}\} \in \cD.
\end{align*}
Since $\cD$ is an ultrafilter on $S$, we deduce that
\begin{align*}
\bigcap_{\substack{1 \le i < j \le m}}A_{i, j} \in \cD,
\end{align*}
that is,
\begin{align*}
\{s \in S \; | \; \text{$\alpha_s,\alpha_{2, s}, \ldots, \alpha_{m, s}$ are all distinct}\}\in \cD.
\end{align*}

By (\ref{e1-cor-F^sep-is-contained-in-U(F)^sep_ultra}), we deduce that all the roots of $f_s(x)$ are distinct for $\cD$-almost all $s \in S$, and thus $f_s(x)$ is separable for $\cD$-almost all $s \in S$. Therefore $\alpha_s$ belongs to $\fF_s^{\sep}$ for $\cD$-almost all $s\in S$, and thus $\alpha = \ulim_{s\in S}\alpha_s\in \prod_{s\in S}\fF_s^{\sep}/\cD = \cU(\F)^{\sep}_{\ultra}$ as required. 

\end{proof}

 We describe the algebraic parts of some ultra-field extensions at the constant field and base field levels.

\begin{lemma}
\label{lem-algebraic-part-of-ultra-algebraic-closure-of-constant-fields}

Let $\bK = \prod_{s \in S}\bF_s^{\alg}/\cD$. Then $\bK_{\alg} = \bK \cap \F^{\alg} = \fK^{\alg}$.

\end{lemma}

\begin{proof}

We know that $\fK^{\alg}$ is a subset of both $\F^{\alg}$ and $\bK$, and thus $\fK^{\alg} \subset \bK_{\alg}$.

We now prove that $\bK_{\alg} \subset \fK^{\alg}$. Take an arbitrary element $\alpha \in \bK_{\alg} = \bK \cap \F^{\alg}$. Since $\alpha \in \F^{\alg}$, there exists an irreducible polynomial $P(x) \in \A[x]$ of the form
\begin{align*}
P(x) = a_0(t) + a_1(t)x + \cdots + a_d(t)x^d
\end{align*}
such that $P(\alpha) = 0$, where the $a_i(t)$ are polynomials in $\A = \fK[t]$, $a_d(t)$ is a nonzero polynomial, and $d$ is a positive integer. For each $0 \le i \le d$, one can write $a_i(t) = \ulim_{s\in S}a_{i, s}(t)$, where $a_{i, s}(t) \in \fA_s = \bF_s[t]$ for $\cD$-almost all $s \in S$. Set
\begin{align*}
P_s(x) = a_{0, s}(t) + a_{1, s}(t)x + \cdots + a_{d, s}(t)x^d \in \fA_s[x].
\end{align*}
By Lemma \ref{lem-elementary-lemma1}, $P_s$ is irreducible over $\fF_s$ for $\cD$-almost all $s \in S$. 

Since $a_d(t)$ is a nonzero polynomial, \L{}o\'s' theorem implies that $a_{d, s}(t)$ is a nonzero polynomial for $\cD$-almost all $s \in S$. 

Since $\alpha \in \bK$, one can write $\alpha = \ulim_{s\in S}\alpha_s$, where $\alpha_s \in \bF_s^{\alg}$ for $\cD$-almost all $s \in S$. Since 
\begin{align*}
0 = P(\alpha) = \ulim_{s\in S}P_s(\alpha_s),
\end{align*}
\L{}o\'s' theorem implies that $P_s(\alpha_s) = 0$ for $\cD$-almost all $s \in S$. Thus
\begin{align}
\label{e1-lem-algebraic-parts-of-ultra-algebraic-closure-of-constant-fields}
a_{0, s}(t) + a_{1, s}(t)\alpha_s + \cdots + a_{d, s}(t)\alpha_s^d = 0
\end{align}
for $\cD$-almost all $s \in S$. Since $a_{d, s}(t)$ is a nonzero polynomial in $\fA_s = \bF_s[t]$ for $\cD$-almost all $s \in S$, there exists an element $t_s \in \bF_s$ such that $a_{d, s}(t_s) \ne 0$. Specializing $t$ to $t_s$ in (\ref{e1-lem-algebraic-parts-of-ultra-algebraic-closure-of-constant-fields}), we deduce that
\begin{align*}
a_{0, s}(t_s) + a_{1, s}(t_s)\alpha_s + \cdots + a_{d, s}(t_s)\alpha_s^d = 0
\end{align*}
for $\cD$-almost all $s\in S$. Taking the ultraproduct on both sides of the above equation, we deduce that
\begin{align*}
b_0 + b_1\alpha + \cdots + b_d\alpha^d = 0,
\end{align*}
where $b_i = \ulim_{s\in S}a_{i, s}(t_s) \in \prod_{s\in S}\bF_s/\cD = \fK$ for each $0 \le i \le d$ such that $b_d = \ulim_{s\in S}a_{d, s}(t_s) \ne 0$. Thus $\alpha \in \fK^{\alg}$, which verifies the lemma.

\end{proof}

\begin{lemma}
\label{lem-U(F)-algebraic-part-is-F}

\begin{itemize}

\item []

\item [(i)] $\cU(\F)_{\alg} = \cU(\F)_{\sep} = \F$.

\item [(ii)] $\cU(\A)_{\alg} = \cU(\A)_{\sep} = \A$.

\end{itemize}

\end{lemma}

\begin{proof}
	
	We first prove part (i). It is obvious that $\F$ is a subset of $\cU(\F)_{\alg}$. Take any element $\alpha \in \cU(\F)_{\alg} =\cU(\F) \cap \F^{\alg}$. Thus one can write $\alpha = \ulim_{s\in S}\alpha_s$ for some elements $\alpha_s \in \fF_s$, and there exists a minimal polynomial $f(x) \in \F[x]$ of $\alpha$ that is of the form
	\begin{align*}
		f(x) = a_0 + a_1x + \cdots + a_{m - 1}x^{m - 1} + x^m,
	\end{align*} 
	where the $a_i$ belong to $\F = \fK(t)$ and $m \in \bZ_{>0}$. For each $0 \le i \le m - 1$, we write $a_i = \ulim_{s\in S}a_{i, s}$ for some elements $a_{i, s} \in \fF_s = \bF_s(t)$, and define
	\begin{align*}
		f_s(x) = a_{0, s} + a_{1, s}x + \cdots + a_{m - 1, s}x^{m - 1} + x^m \in \fF_s[x].
	\end{align*}
	It is clear that $f(x) = \ulim_{s\in S}f_s(x)$, and thus
	\begin{align*}
		0 = f(\alpha) = \ulim_{s\in S}f_s(\alpha_s),
	\end{align*}
	and thus \L{}o\'s' theorem implies that $f_s(\alpha_s) = 0$ for $\cD$-almost all $s \in S$. Since $f$ is irreducible over $\F$, Lemma \ref{lem-elementary-lemma1} implies that $f_s$ is irreducible over $\fF_s$ for $\cD$-almost all $s \in S$, and hence $f_s$ is the minimal polynomial of $\alpha_s$ over $\fF_s$ for $\cD$-almost all $s \in S$. Since $\alpha_s \in \fF_s$, we see that 
	\begin{align*}
		m = \deg(f_s) = 1,
	\end{align*}
	and thus $\deg(f) = m = 1$. So $\alpha$ belongs to $\F$, and hence
	\begin{align*}
		\cU(\F)_{\alg} = \cU(\F) \cap \F^{\alg} \subset \F.
	\end{align*}
	Therefore
	\begin{align*}
		\cU(\F)_{\alg}= \F.
	\end{align*}
	
	Since $\F \subset \cU(\F)_{\sep} \subset \cU(\F)_{\alg}$, it follows immediately that
	\begin{align*}
		\cU(\F)_{\alg} = \cU(\F)_{\sep} = \F.
	\end{align*}

For part (ii), we see that $\A \subset \cU(\A)_{\sep} \subset \cU(\A)_{\alg}$. So it suffices to prove that
\begin{align*}
\cU(\A)_{\alg} \subset \A.
\end{align*}
Indeed, let $\alpha$ be an arbitrary element in $\cU(\A)_{\alg} = \cU(\A) \cap \F^{\alg}$. By part (i), we know that 
\begin{align*}
\cU(\A)_{\alg} \subset \cU(\F)_{\alg} = \F,
\end{align*}
and thus $\alpha \in \F$.

If $\alpha = 0$, it is trivial that $\alpha \in \A$. 

Suppose that $\alpha \ne 0$. Since $\F = \fK(t)$ is the quotient field of $\A = \fK[t]$, there exist two elements $\beta, \lambda \in \A$ such that $\beta, \lambda \ne 0$, $\alpha = \beta/\lambda$ and $\gcd(\beta, \lambda) = 1$. One can write $\beta = \ulim_{s\in S}\beta_s$ and $\lambda = \ulim_{s\in S}\lambda_s$, where $\beta_s, \lambda_s$ are elements in $\fA_s = \bF_s[t]$ for $\cD$-almost all $s \in S$. Since $\gcd(\beta, \lambda) = 1$ and $\lambda \ne 0$, \L{}o\'s' theorem implies that $\gcd(\beta_s, \lambda_s) = 1$ and $\lambda_s \ne 0$ for $\cD$-almost all $s \in S$. Since $\alpha \in \cU(\A)$, $\alpha = \ulim_{s\in S}\alpha_s$ for some elements $\alpha_s \in \fA_s$, and thus
\begin{align*}
\alpha = \ulim_{s\in S}\dfrac{\beta_s}{\lambda_s} = \ulim_{s\in S}\alpha_s.
\end{align*}
Therefore
\begin{align*}
\beta_s = \alpha_s\lambda_s
\end{align*}
for $\cD$-almost all $s \in S$. Since $\gcd(\beta_s, \lambda_s) = 1$, we deduce that $\beta_s$ divides $\alpha_s$, and thus $\alpha_s/\beta_s \in \fA_s$. Therefore
\begin{align*}
(\alpha_s/\beta_s)\lambda_s = 1,
\end{align*}
which implies that $\lambda_s$ is a unit in $\fA_s = \bF_s[t]$ for $\cD$-almost all $s \in S$, and hence $\lambda_s \in \bF_s^{\times}$. Thus $\lambda = \ulim_{s\in S}\lambda_s \in \prod_{s\in S}\bF_s^{\times}/\cD = \fK^{\times}$, and therefore $1/\lambda \in \fK^{\times}$. Hence $\alpha = (1/\lambda)\beta \in \A$, and so $\cU(\A)_{\alg} \subset \A$ as required.

\end{proof}

\begin{lemma}
\label{lem-algebraic-part-of-principal-ideal-in-U(F)-is-principal-ideal-in-F}

Let $\alpha$ be a polynomial in $\A = \fK[t]$. Then the algebraic part of the ideal $\alpha \cU(\A)$ over $\F$ is the principal ideal $\alpha \A$, i.e.,
\begin{align*}
(\alpha\cU(\A))_{\alg} = \alpha\cU(\A) \cap \F^{\alg} = \alpha \A.
\end{align*}

\end{lemma}

\begin{proof}

If $\alpha$ is the zero polynomial, then the assertion is trivial. Suppose that $\alpha$ is a nonzero polynomial in $\A$. 

Since $\A \subset \cU(\A)$, it is obvious that $\alpha\A$ is a subset of $(\alpha\cU(\A))_{\alg}$. 

For the opposite implication, take an arbitrary element $\beta \in (\alpha\cU(\A))_{\alg}$. By part (ii) of Lemma \ref{lem-U(F)-algebraic-part-is-F}, $\cU(\A)_{\alg} = \A$,and thus $\beta \in (\alpha\cU(\A))_{\alg} \subset \cU(\A)_{\alg} = \A$. Since $\beta \in \alpha\cU(\A)$, there exists an element $\gamma = \ulim_{s\in S}\gamma_s \in \cU(\A)$ for some elements $\gamma_s \in \fA_s$ such that $\beta = \alpha\gamma$. Thus $\gamma = \beta/\alpha \in \F$. Thus it follows from Lemma \ref{lem-U(F)-algebraic-part-is-F} that $\gamma \in \cU(\A) \cap \F \subset \cU(\A) \cap \F^{\alg} = \cU(\A)_{\alg} = \A$, and therefore $\beta = \alpha\gamma \in \alpha\A$. Hence $(\alpha\cU(\A))_{\alg} = \alpha\cU(\A) \cap \F^{\alg} \subset \alpha \A$ as required.

\end{proof}

The following result concerns separability of algebraic parts of ultra-field extensions.

\begin{theorem}
\label{thm-main-thm1-Algebraic-part-of-L-is-a-finite-separable-extension}
	
Let $\fL_s$ be a (possibly infinite) algebraic extension over $\fF_s = \bF_s(t)$ for $\cD$-almost all $s \in S$, and let $\fL = \prod_{s \in S}\fL_s/\cD$ be an ultra-field extension of $\cU(\F)$.  Then 
\begin{itemize}
	\item [(i)] if $\fL_s$ is of degree $n_s \in \bZ_{>0}$ over $\fF_s$ for $\cD$-almost all $s \in S$, and there are only finitely many integers dividing the hyperinteger $n = \ulim_{s\in S}n_s$ in $\bZ^{\#}$, then $\fL_{\alg}$ is a finite extension of $\F$ and $\fL_{\sep}$ is a finite separable extension of $\F$.

	\item [(ii)] if $\fL_s$ is a separable extension of $\fF_s$ for $\cD$-almost all $s \in S$, then $\fL_{\alg}$ is a separable extension of $\F$, and $\fL_{\alg} = \fL_{\sep}$. 
	
	\item [(iii)] if the characteristic of $\bF_s$, being the same as that of $\fF_s = \bF_s(t)$, is distinct for $\cD$-almost all $s \in S$, then  $\fL_{\alg}$ is a  separable extension of $\F$ and $\fL_{\alg} = \fL_{\sep}$.

\end{itemize}	
	
\end{theorem}

\begin{proof}
	
	For part (i), since $\F \subset \fL_{\sep} \subset \fL_{\alg}$, it suffices to prove that $\fL_{\alg}$ is a finite extension of $\F$.
	
	Take an arbitrary element $\alpha \in \fL_{\alg} = \fL \cap \F^{\alg}$. Then $\alpha$ is algebraic over $\F$ of degree $r$ for some positive integer $r$, and thus there is a monic, irreducible polynomial $P(x) = \sum_{i = 0}^ra_ix^i \in \F[x]$ such that $P(\alpha) = 0$, where the $a_i$ belong to $\F$ such that $a_r \ne 0$. We can write $\alpha$ in the form $\alpha = \ulim_{s\in S}\alpha_s$ for some elements $\alpha_s \in \fL_s$. Since each $a_i \in \F \subset \cU(\F) = \prod_{s\in S}\fF_s/\cD$, one can write $a_i = \ulim_{s\in S}a_{i, s}$ for some elements $a_{i, s}\in \fF_s = \bF_{s}(t)$. Therefore
	\begin{align*}
		0 = P(\alpha) &= \sum_{i = 0}^r a_i\alpha^i  = \sum_{i = 0}^r\ulim_{s\in S}a_{i, s}\ulim_{s\in S}\alpha_s^i \\
		&= \sum_{i = 0}^r\ulim_{s\in S}a_{i, s}\alpha_s^i = \ulim_{s\in S}\left(\sum_{i = 0}^r a_{i, s}\alpha_s^i\right) \\
		&= \ulim_{s\in S}P_s(\alpha_s),
	\end{align*}
	where $P_s(x) = \sum_{i = 0}^ra_{i, s}x^i \in \fF_s[x]$. Thus $P_s(\alpha_s) = 0$ for $\cD$-almost all $s \in S$. Clearly $P(x) = \ulim_{s\in S}P_s(x)$, and since $P$ is a monic irreducible polynomial over $\F$, Lemma \ref{lem-elementary-lemma1} implies that $P_s$ is a monic, irreducible of degree $r$ over $\fF_s$ for $\cD$-almost all $s \in S$. Therefore $P_s$ is the minimal polynomial of $\alpha_s$ for $\cD$-almost all $s \in S$, and so $\fF_s(\alpha_s)$ is a finite extension of degree $r$ over $\fF_s$.

	Assume the contrary to the assertion of the theorem, i.e., $\fL_{\alg}$ is an infinite extension of $F$. Then there is an ascending sequence of subfields $(\F_i)_{i \ge 0}$ of $\fL_{\alg}$ such that
	\begin{align*}
		\F = \F_0 \subsetneq \F_1 \subsetneq \cdots \subsetneq \F_i \subsetneq \F_{i + 1} \subsetneq \cdots \subsetneq \cdots \fL_{\alg}
	\end{align*}
	such that for each $i \ge 1$, $\F_i$ is a finite extension of degree $r_i > 1$ over $\F_{i - 1}$ of the form $\F_i = \F_{i - 1}(\alpha_i)$ for some algebraic element $\alpha_i \in \fL_{\alg}$ of degree $r_i$ over $\F_{i - 1}$. 
	
	Write $\alpha_i = \ulim_{s\in S}\alpha_{i, s}$ for some elements $\alpha_{i, s}\in \fL_s$. By the transitivity property, $\alpha_i$ is algebraic of degree $r_1r_2\cdots r_i$ over $\F$, and thus it follows from the result proved in the beginning of the proof that the field $\fF_s(\alpha_{i, s})$ is a finite extension of degree $r_1r_2\cdots r_i$ over $\fF_s$ for $\cD$-almost all $s \in S$. Thus $r_1r_2 \cdots r_i$ divides $n = \ulim_{s\in S}n_s$ for all $i \ge 1$. Since the product $\prod_{j = 1}^i r_j$ goes to infinity whenever $i$ goes to infinity and each integer $r_j$ is strictly greater than one, $n$ is divisible by infinitely many integers in $\bZ^{\#}$, which is a contradiction. Thus $\fL_{\alg}$ is a finite extension of $\F$.

	For part (ii), it suffices to prove that $\fL_{\alg}$ is a separable extension of $\F$, which in particular implies that $\fL_{\alg} = \fL_{\sep}$. 
	
	Assume the contrary, i.e., $\fL_{\alg}$ is not separable, and thus the characteristic of $\F$, being the same as $\fK$ equals some prime $p > 0$. By \L{}o\'s' theorem, $\fF_s$ is of the same characteristic $p$ for $\cD$-almost all $s \in S$.

	Let $\alpha = \ulim_{s\in S}\alpha_s \in \fL_{\alg}$ be an inseparable element over $\F$ for some elements $\alpha_s \in \fL_s$. Thus the minimal polynomial of $\alpha$ is of the form $Q(x^p)$ for some polynomial $Q(x) = \sum_{i = 0}^ra_i x^i \in \F[x]$. Letting $a_i = \ulim_{s\in S}a_{i, s}$ for some elements $a_{i, s}\in \fF_s$ and repeating the same arguments as in the beginning of the proof that make use of Lemma \ref{lem-elementary-lemma1}, we know that $Q_s(x^p)$ is the minimal polynomial of $\alpha_s$, where $Q_s(x) = \sum_{i = 0}^r a_{i, s}x^i$. Thus the extension $\fF_s(\alpha_s)$ is not a separable extension of $\fF_s$ for $\cD$-almost all $s \in S$, which is a contradiction since $\fF_s(\alpha_s)$ is a subextension of the separable extension $\fL_s$. Therefore $\fL_{\alg}$ is separable, and part (ii) follows immediately.

	For part (iii), \L{}o\'s' theorem implies that the characteristic of $\fK$, being the same as that of $\F = \fK(t)$ is $0$. Thus every extension over $\F$ is separable; in particular, $\fL_{\alg} = \fL_{\sep}$ and $\fL_{\alg}$ is a separable extension of $\F$.

\end{proof}

\begin{remark}
	\label{rem-L_alg=L_sep}
	
	Part (iii) of Theorem \ref{thm-main-thm1-Algebraic-part-of-L-is-a-finite-separable-extension} is most frequently used throughout this paper. In most applications in this paper, we are concerned with fields $\bF_s$ of distinct characteristics for $\cD$-almost all $s \in S$, which in particular implies that the characteristic of $\F = \fK(t)$ is $0$, and thus every extension over $\F$ is separable.

\end{remark}

\begin{remark}
	\label{rem-divisibility-of-global-degree-and-ultra-degrees}
	
If $\fL_s$ is of degree $n_s \in \bZ_{>0}$ over $\fF_s$ for $\cD$-almost all $s \in S$, following the proof of Theorem \ref{thm-main-thm1-Algebraic-part-of-L-is-a-finite-separable-extension}, we deduce that if $\alpha$ is an element in $\fL_{\alg}$ of degree $m \in \bZ_{>0}$ over $\F$, then $m$ divides $n_s$ for $\cD$-almost all $s \in S$, i.e., $m$ divides $n = \ulim_{s\in S}n_s$ in $\bZ^{\#}$. 
	In particular, if $m = [\fL_{\alg}: \F]$, then $m$ divides $n$ in $\bZ^{\#}$.

\end{remark}

\begin{theorem}
	\label{thm-main-thm2-Galois-property-of-algebraic-parts}
	
	Assume that $\fL_s$ is a (possibly infinite) Galois extension of $\fF_s$ for $\cD$-almost all $s \in S$. Then $\fL_{\alg} = \fL_{\sep}$ and $\fL_{\alg}$ is a Galois extension of $F$.

\end{theorem}

\begin{proof}
	
	By Theorem \ref{thm-main-thm1-Algebraic-part-of-L-is-a-finite-separable-extension}, $\fL_{\alg} = \fL_{\sep}$ and $\fL_{\alg}$ is a separable extension of $\F$. So it suffices to prove that $\fL_{\alg}$ is a normal extension of $\F$.
	
	Let $f(x) = \sum_{i = 0}^ma_ix^i$ be an irreducible polynomial in $\F[x]$ that has a root $\alpha = \ulim_{s\in S}\alpha_s \in \fL_{\alg}$ for some elements $\alpha_s \in \fL_s$, where the $a_i$ belong to $\F$ such that $a_m \ne 0$. For each $0 \le i \le m$, one can write $a_i = \ulim_{s\in S}a_{i, s}$, where the $a_{i, s}$ belong to $\fF_s$. Thus 
 	\begin{align*}
 		f(x) = \ulim_{s\in S}f_s(x),
 	\end{align*}
 	where 
 	\begin{align*}
 		f_s(x) = \sum_{i = 0}^ma_{i, s}x^i \in \fF_s[x].
 	\end{align*}

Since $f(x)$ is irreducible over $\F$, by Lemma \ref{lem-elementary-lemma1}, $f_s(x)$ is irreducible over $\fF_s$ for $\cD$-almost all $s \in S$. 

We see that
\begin{align*}
	0 = f(\alpha) = \sum_{i = 0}^m a_i\alpha^i = \ulim_{s\in S}\sum_{i = 0}^ma_{i, s}\alpha_s^i = \ulim_{s\in S}f_s(\alpha_s),
\end{align*}
which implies that $f_s(\alpha_s) = 0$ for $\cD$-almost all $s \in S$. Since $\fL_s$ is a Galois extension of $\fF_s$ and the irreducible polynomial $f_s$ has a root $\alpha_s$ in $\fL_s$, all the roots of $f_s$ belong to $\fL_s$. Let $A_s$ denote the set of all roots of $f_s$ that is a subset of $\fL_s$. Take an arbitrary root $\beta_s$ of $f_s$ in $A_s$  for $\cD$-almost all $s \in S$. Using similar arguments above, we deduce that the element $\beta = \ulim_{s\in S}\beta_s \in \fL$ satisfies 
\begin{align*}
	0 = f(\beta) = \ulim_{s\in S}f_s(\beta_s) = 0,
\end{align*}
and thus $\beta$ is a root of $f$. Since $f \in \F[x]$, $\beta$ is algebraic over $\F$, and thus $\beta$ also belongs to $\F^{\alg}$, and therefore $\beta \in \fL \cap \F^{\alg} = \fL_{\alg}$. Thus we have showed that every element in the ultraproduct $\prod_{s\in S}A_s/\cD$ is a root of $f$. 

Since $\fL_{\alg}$ is a separable extension of $\F$, $f$ is separable, and thus there are exactly $m$ roots of $f$ in $\F^{\alg}$. Since $A_s$ has cardinality $m$ for $\cD$-almost all $s \in S$, the cardinality of $\prod_{s\in S}A_s/\cD$ is $\ulim_{s\in S}m = m$ (see \cite[Lemma 3.7]{bell-slomson}), we deduce that the ultraproduct $\prod_{s\in S}A_s/\cD$ consists of exactly $m$ roots of $f$. Since $\prod_{s\in S}A_s/\cD$ is a subset of $\fL_{\alg}$, $f$ splits into linear factors in $\fL_{\alg}$. Thus $\fL_{\alg}$ is a normal extension of $\F$ and therefore $\fL_{\alg}$ is a Galois extension of $\F$.

\end{proof}

\begin{remark}

\begin{itemize}

\item []
	
\item [(i)]	For the rest of this paper, unless otherwise stated, we only work with separable extensions $\fL_s$ of $\fF_s$ for $\cD$-almost all $s \in S$ or the characteristic of $\F$ is $0$, both of which cases imply that $\fL_{\alg} = \fL_{\sep}$ and $\fL_{\alg}$ is a separable extension of $\F$. Thus it suffices to work with algebraic parts $\fL_{\alg}$ of certain ultra-field extensions $\fL$ over $\cU(\F)$ instead of separable parts $\fL_{\sep}$. 

\item [(ii)] Even under the assumption that $\fL_s$ is a finite Galois extension of $\fF_s$ for $\cD$-almost all $s \in S$, the algebraic part $\fL_{\alg}$ may be an infinite Galois extension of $\F$. In Proposition \ref{prop-structure-of-subfields-of-L-alg-ininite-over-F} in Subsection \ref{subsubsec-L_alg-is-Z-hat-profinite-Galois-over-F}, we will provide an example in which $\fL_s$ is a finite Galois extension of $\fF_s$ for $\cD$-almost all $s \in S$ but $\fL_{\alg}$ is an infinite Galois extension of $\F$ whose Galois group of $\fL_{\alg}$ over $\F$ is isomorphic to the profinite completion $\widehat{\bZ}$.

\end{itemize}
	
\end{remark}

Combining Theorem \ref{thm-main-thm1-Algebraic-part-of-L-is-a-finite-separable-extension} and Theorem \ref{thm-main-thm2-Galois-property-of-algebraic-parts}, we obtain the following.

\begin{corollary}
	\label{cor-main-cor1-L_alg-is-a-finite-Galois-extension-of-F}

Let $\fL_s$ be a finite Galois extension of degree $n_s \in \bZ_{>0}$ over $\fF_s$ for $\cD$-almost all $s \in S$, and let $\fL = \prod_{s\in S}\fL_s/\cD$ be the ultra-field extension of ultra-degree $n = \ulim_{s\in S}n_s \in \bZ_{>0}^{\#}$ over $\cU(\F)$. Assume that there are only finite many integers dividing $n$ in $\bZ^{\#}$. Then $\fL_{\alg}$ is a finite Galois extension of $\F$.

\end{corollary}

\subsection{Shadows of separable extensions of $\F = \fK(t)$}
\label{subsec-shadows-and-ultra-shadows}

Let $\H$ be a finite separable extension of degree $m$ over $\F$. In this subsection, we will construct, for $\cD$-almost all $s\in S$, a finite separable extension $\fH_s$ of the same degree $m$ over $\fF_s$ such that the ultra-finite extension $\fH = \prod_{s\in S}\fH_s/\cD$ satisfies
\begin{align*}
	\H = \fH_{\sep} = \fH_{\alg},
\end{align*}
that is, $\H$ is the algebraic part of $\fH$ over $\F$. It turns out that $\fH$ is also a finite separable extension of the same degree $m$ over $\cU(\F)$. Such a finite separable extension $\fH_s$ of degree $m$ over $\fF_s$ is unique for $\cD$-almost all $s\in S$ in the sense that if there exists a finite separable extension $\fG_s$ of degree $m$ over $\fF_s$ for $\cD$-almost all $s\in S$ such that the algebraic part $\fG_{\alg}$ over $\F$ of the ultra-finite extension $\fG = \prod_{s\in S}\fG_s/\cD$ is $\H$, then 
\begin{align*}
	\fG_s = \fH_s
\end{align*}
for $\cD$-almost all $s \in S$.

\begin{definition}
	\label{def-shadows-of-algebraic-extensions}
	(shadow and ultra-shadow)
	
	For $\cD$-almost all $s\in S$, the unique separable extension $\fH_s$ of degree $m$ over $\fF_s$ described above is called the \textbf{$s$-th shadow of $\H$}. The ultra-field extension $\fH = \prod_{s\in S}\fH_s/\cD$ whose algebraic part $\fH_{\alg}$ over $\F$ coincides with $\H$, is called the \textbf{ultra-shadow of $\H$}.

\end{definition}

For an illustration of shadows of separable extensions of $\F$, see Figure \ref{fig-diagram-of-shadows}.

\begin{center}
  \begin{figure}
  	\label{fig-diagram-of-shadows}
	\begin{tikzpicture}

\node (Q1) at (0,0) {$\F = \fK(t)$};
     \node (Q2) at (3, 3 ) {$\H = \fH_{\alg} = \fH_{\sep}= \fH \cap \F^{\alg}$};

          \node (Q4) at (4.5,4.5) {$\F^{\sep}$};

     \node (Q6) at (6,6) {$\F^{\alg}$};
     
        \node (Q10) at (7.5, 7.5) {$\cU(\F)^{\alg}_{\ultra} = \prod_{s\in S}\fF_s^{\alg}/\cD$};
        
    \node (Q3) at (-3,1) {$\cU(\F) = \prod_{s\in S}\fF_s/\cD$};

    \node (Q5) at (0,4) {$\fH = \prod_{s\in S}\fH_s/\cD$};
     
        \node (Q7) at (2.5, 6.5) {$\cU(\F)^{\sep}_{\ultra} = \prod_{s\in S}\fF_s^{\sep}/\cD$};

     \draw (Q1)--(Q2)  ;
      \draw (Q2)--(Q4)  ;
    \draw (Q4)--(Q6) ;
    
        \draw (Q4)--(Q7)  ;

    \draw (Q6)--(Q8)  ;
         
       \draw (Q8)--(Q10)  ;
    \draw (Q10)--(Q7)  ;

    \draw (Q1)--(Q3)  ;

    \draw (Q3)--(Q5)  ;

    \draw (Q5)--(Q2) ;

      \draw (Q5)--(Q7);
       \end{tikzpicture}

      \caption{Shadows of algebraic extensions of $\F$}
      \end{figure}
    \end{center}

We now describe $\fH_s$ explicitly and prove its uniqueness. Since $\H$ is a finite separable extension of degree $m$ over $\F$, the primitive element theorem implies that $\H = \F(\alpha)$ for some element $\alpha \in \H$ whose minimal polynomial is of the form
\begin{align*}
	P(x) = a_0 + a_1x + \cdots + a_{m - 1}x^{m - 1} + x^m \in \F[x],
\end{align*} 
where the $a_i$ belong to $\F$. 

Since $\alpha$ is a separable element, $P(x)$ is a separable polynomial. Since $\H$ is a subfield of $\F^{\sep}$, Corollary \ref{cor-F^sep-is-contained-in-U(F)^sep_ultra} implies that $\H$ is contained in $\cU(\F)^{\sep}_{\ultra}$, and thus one can write $\alpha = \ulim_{s\in S}\alpha_s$ for some elements $\alpha_s \in \fF_s^{\sep}$. Since $a_i \in \F$, one can write $a_i = \ulim_{s\in S}a_{i, s}$ for some elements $a_{i, s} \in \fF_s$. Letting
\begin{align*}
	P_s(x) = a_{0, s} + a_{1, s}x + \cdots + a_{m - 1, s}x^{m - 1} + x^m \in \fF_s[x],
\end{align*}
we deduce that 
\begin{align*}
	P(x) = \ulim_{s\in S}P_s(x).
\end{align*}
Since $P(\alpha) = 0$, it follows that $P_s(\alpha_s) = 0$ for $\cD$-almost all $s \in S$. Since $P(x)$ is irreducible over $\F$, Lemma \ref{lem-elementary-lemma1} implies that $P_s(x)$ is irreducible over $\fF_s$ for $\cD$-almost all $s \in S$. Therefore $P_s(x)$ is the minimal polynomial of $\alpha_s$ for $\cD$-almost all $s\in S$. 

   Let $\fH_s = \fF_s(\alpha_s)$ be the field obtained from $\fF_s$ by adjoining $\alpha_s$. By what we have showed above, $\fH_s$ is a finite separable extension of degree $m$ over $\fF_s$ for $\cD$-almost all $s\in S$. 
   
   We contend that $\fH_s$ is the $s$-th shadow of $\H$. Indeed, let $\fH = \prod_{s\in S}\fH_s/\cD$ be the ultra-field extension of $\cU(\F)$. Since $\alpha_s \in \fH_s$ for $\cD$-almost all $s \in S$, $\alpha = \ulim_{s\in S}\alpha_s \in \prod_{s\in S}\fH_s/\cD = \fH$, and thus
   \begin{align*}
   	\alpha \in \fH \cap \H \subset \fH \cap \F^{\alg} = \fH_{\alg}.
   \end{align*}
   Therefore
   \begin{align*}
   	\H = \F(\alpha) \subset \fH_{\alg}.
   \end{align*}
   
   By part (ii) of Theorem \ref{thm-main-thm1-Algebraic-part-of-L-is-a-finite-separable-extension}, $\fH_{\sep} = \fH_{\alg}$, and $\fH_{\alg}$ is a finite separable extension of $\F$. Let $h = [\fH_{\alg} : \F]$. Since $\H$ is a subfield of $\fH_{\alg}$,
   \begin{align*}
   	h = [\fH_{\alg}: \F] \ge [\H : \F] = m.
   \end{align*} 
   
   By Remark \ref{rem-divisibility-of-global-degree-and-ultra-degrees} and since the degree of $\fH_s$ over $\fF_s$ is $m$, $h$ divides $m$, which implies that $h = m$. Thus
   \begin{align*}
   	\H = \fH_{\alg} = \fH_{\sep}.
   \end{align*}

  We now prove the uniqueness of $\fH_s$ for $\cD$-almost all $s \in S$. Assume that there exists another separable extension $\fG_s$ of degree $m$ over $\fF_s$ for $\cD$-almost all $s\in S$ such that the algebraic part of the ultra-field extension $\fG = \prod_{s\in S}\fG_s/\cD$ over $\F$ is $\H$. Then
  \begin{align*}
  	\alpha =\ulim_{s\in S}\alpha_s \in \H = \fG_{\alg} = \fG \cap \F^{\alg} \subset \fG,
  \end{align*}
  and thus $\alpha_s \in \fG_s$ for $\cD$-almost all $s \in S$. Thus
  \begin{align*}
  	\fH_s = \fF_s(\alpha_s) \subset \fG_s.
  \end{align*}
  
  Since $m = [\fG_s : \fF_s]  \ge [\fH_s : \fF_s] = m$, it follows that 
  \begin{align*}
  	[\fH_s : \fF_s] =[\fG_s : \fF_s] = m,
  \end{align*}
  which implies that $\fG_s = \fH_s$ for $\cD$-almost all $s\in S$. Thus $\fH_s$ is a unique separable extension of degree $m$ over $\fF_s$ such that the algebraic part of $\fH$ over $\F$ is $\H$. Therefore $\fH_s$ is the $s$-th shadow of $\H$. 
  
  We now prove that $\fH = \cU(\F)(\alpha)$. Since $\alpha \in \H = \fH_{\alg} = \fH \cap \F^{\alg}$, $\alpha$ belongs to $\fH$, and thus $\cU(\F)(\alpha) \subset \fH$. By Lemma \ref{rem-irreducibility-is-the-same-in-F-and-U(F)}, the minimal polynomial $P(x)$ of $\alpha$ is also irreducible over $\cU(\F)$. Since $P(x)$ is separable and of degree $m$, $\cU(\F)(\alpha)$ is a finite separable extension of degree $m$ over $\cU(\F)$. 
  
  Since $\fH_s$ is a finite separable extension of degree $m$ over $\fF_s$ for $\cD$-almost all $s \in S$, we know from Proposition \ref{prop-explicit-description--for--algebraic-extension-of-degree-d-of-ultra-fields} that $\fH = \prod_{s\in S}\fH_s/\cD$ is a finite algebraic extension of degree $m$ over $\cU(\F)$. Since 
  \begin{align*}
  [\cU(\F)(\alpha): \cU(\F)] = [\fH: \cU(\F)] = m
  \end{align*}
  and $\cU(\F)(\alpha) \subset \fH$, we deduce that $\fH = \cU(\F)(\alpha)$.

 We summarize the discussion above in the following.
  
  \begin{proposition}
  	\label{prop-explicit-descriptions-of-shadows-and-ultra-shadows}
  	
  Let $\fH_s$ be the $s$-th shadow of $\H$ for $\cD$-almost all $s \in S$, and let $\fH = \prod_{s\in S}\fH_s/\cD$ be the ultra-shadow of $\H$. Let $\H = \F(\alpha)$ for some primitive element of degree $m$ over $\F$, and let $\alpha = \ulim_{s\in S}\alpha_s$ for some elements $\alpha_s \in \fF_s^{\sep}$ \footnote{Such elements $\alpha_s$ exist by Corollary \ref{cor-F^sep-is-contained-in-U(F)^sep_ultra}.}. Then
\begin{itemize}
  
  	\item [(i)] $\fH_s = \fF_s(\alpha_s)$ for $\cD$-almost all $s \in S$.
  
  \item [(ii)]	$\fH$ is a finite separable extension of degree $m$ over $\cU(\F)$, and can be written in the form $\fH = \cU(\F)(\alpha)$.
  	
  	\end{itemize}

  \end{proposition}

  We prove the main result in this subsection.
  
  \begin{theorem}
  	\label{thm-main-thm3-Galois-group-structures-of-shadows}
  	
  	Let $\H$ be a Galois extension of degree $m$ over $\F$, and let $G$ denote the Galois group of $\H$ over $\F$.   Let $\fH_s$ be the $s$-th shadow of $\H$ for $\cD$-almost all $s \in S$, and let $\fH = \prod_{s\in S}\fH_s/\cD$ be the ultra-shadow of $\H$. Then
  	\begin{itemize}
  		
  		\item [(i)] $\fH_s$ is a finite Galois extension of degree $m$ over $\fF_s$, and the Galois group $\Gal(\fH_s/\fF_s)$ of $\fH_s$ over $\fF_s$ is isomorphic to $G$ for $\cD$-almost all $s\in S$.
  		
  		\item [(ii)] $\fH$ is a finite Galois extension of degree $m$ over $\cU(\F)$, and the Galois group $\Gal(\fH/\cU(\F))$ of $\fH$ over $\cU(\F)$ is isomorphic to 
  		$G$. 

\item [(iii)] $G = \{\sigma_{|\H} \; |\; \sigma \in \Gal(\fH/\cU(\F))\}$ and $\Gal(\fH/\cU(\F)) = \Gal_{\ultra}(\fH/\cU(\F)) = \prod_{s\in S}\Gal(\fH_s/\fF_s)/\cD$ \footnote{See Subsection \ref{subsubsec-ultra-Galois-extensions} for a notion of ultra-Galois groups}.

  	\end{itemize}

  \end{theorem}

  \begin{proof}
  	
  	We first prove part (i). By the primitive element theorem, $\H = \F(\alpha_1)$ for some element $\alpha_1 \in \H$. Let $\alpha_1, \alpha_2, \ldots, \alpha_m$ be all the distinct Galois conjugates of $\alpha_1$, i.e., 
  	\begin{align*}
  		\{\alpha_1, \ldots, \alpha_m\} = \{\sigma(\alpha_1) \; |\; \sigma \in G\}.
  	\end{align*}
  	
  	Since $\H$ is Galois over $\F$, all the Galois conjugates of $\alpha_1$ belong to $\H$, and thus one can write
  \begin{align}
  \label{e1-main-thm3-Galois-structures-of-shadows}
  \H = \F(\alpha_i)
  \end{align}
  for all $1 \le i \le m$.

  Let $\fH_s$ be the $s$-th shadow of $\H$, and let $\fH = \prod_{s \in S}\fH_s/\cD$ be the ultra-shadow of $\H$.

  	For each $1 \le i \le m$, since $\alpha_i \in \H \subset \cU_{\ultra}^{\sep}(\F)$ (see Corollary \ref{cor-F^sep-is-contained-in-U(F)^sep_ultra}), one can write $\alpha_i = \ulim_{s \in S}\alpha_{i, s}$ for elements $\alpha_{i, s}\in \fF^{\sep}_s$. By the construction of the $s$-th shadows as in Proposition \ref{prop-explicit-descriptions-of-shadows-and-ultra-shadows}, $\fH_s = \fF_s(\alpha_{1, s})$.
  
  Let
  \begin{align}
  \label{e1-1/2-main-thm3-Galois-structures-of-shadows-minimal-polynomial}
  P(x) = \sum_{i = 0}^ma_ix^i \in \F[x] 
  \end{align}
  be the minimal polynomial of $\alpha_1$ over $\F$, where the $a_i$ belong to $\F$. Note that $P(x)$ is also the minimal polynomial of $\alpha_i$ for all $1 \le i \le m$. 
  
  For each $s \in S$, write $a_i = \ulim_{s\in S}a_{i, s}$ for some elements $a_{i, s} \in \fF_s$. Note that all the roots of $P(x)$ are all the distinct Galois conjugates of $\alpha_1$, say $\{\alpha_1, \ldots, \alpha_m\}$.
 
 By Lemma \ref{lem-elementary-lemma1}, we deduce that
 \begin{align*}
 	P(x) = \ulim_{s\in S}P_s(x),
 \end{align*}
where 
\begin{align*}
P_s(x) = \sum_{i = 0}^ma_{i, s}x^i \in \fF_s[x].
\end{align*}
Since $P(\alpha_i) = 0$ for every $1 \le i \le m$, we deduce that for all $1 \le i \le m$, $P_s(\alpha_{i, s}) = 0$ for $\cD$-almost all $s \in S$. Since $P$ is irreducible over $\F$, Lemma \ref{lem-elementary-lemma1} again implies that $P_s$ is irreducible over $\fF_s$ for $\cD$-almost all $s \in S$, and therefore $P_s$ is the minimal polynomial of $\alpha_{i, s}$ for all $1 \le i \le m$.

By the construction of the $s$-th shadow of $\fH_s$ described in Proposition \ref{prop-explicit-descriptions-of-shadows-and-ultra-shadows} and (\ref{e1-main-thm3-Galois-structures-of-shadows}), $\fH_s = \fF_s(\alpha_{i, s})$ for any $1 \le i \le m$. In particular, since $\fH_s$ is separable and $P_s$ is the minimal polynomial of $\alpha_{i, s}$, $P_s$ is separable for $\cD$-almost all $s \in S$. 

If $\beta_s$ is any element in $\fF_s^{\sep}$ such that $P_s(\beta_s) = 0$ for $\cD$-almost all $s\in S$, then $P(\beta) = 0$, where $\beta = \ulim_{s\in S}\beta_s \in \cU(\F)^{\sep}_{\ultra}$. Thus $\beta = \alpha_i$ for some $1 \le i \le m$, and thus $\beta_s = \alpha_{i, s}$ for $\cD$-almost all $s \in S$. Therefore, for $\cD$-almost all $s \in S$, all the $m$ distinct roots of $P_s(x)$ are $\alpha_{1, s}, \ldots, \alpha_{m, s}$. Since $\fH_s = \fF_s(\alpha_{i, s})$ for any $1 \le i \le m$, we deduce that
\begin{align}
 	\label{e2-main-thm3-Galois-structures-of-shadows}
\fH_s = \fF_s(\alpha_{1, s}, \ldots, \alpha_{m, s}),
\end{align}
and thus $\fH_s$ is the splitting field of the separable and irreducible polynomial $P_s(x)$. Therefore $\fH_s$ is a Galois extension of degree $m$ over $\fF_s$ for $\cD$-almost all $s \in S$.

  We now prove that $\Gal(\fH_s/\fF_s)$ is isomorphic to $G$.
  
 For a given element $\sigma \in G = \Gal(\H/\F)$, let $\iota$ be a unique permutation of $\{1, 2, \ldots, m\}$ such that 
 \begin{align}
 	\label{e3-main-thm3-Galois-structures-of-shadows}
 	\sigma(\ulim_{s\in S}\alpha_{i, s}) = \sigma(\alpha_i) = \alpha_{\iota(i)} = \ulim_{s\in S}\alpha_{\iota(i), s}
 \end{align}
 for each $1 \le i \le m$.  We define an automorphism $\sigma_s \in \Gal(\fH_s/\fF_s)$ by letting 
\begin{align}
\label{e4-main-thm3-Galois-structures-of-shadows}
\begin{cases}
	\sigma_s(\alpha_{i, s}) &= \alpha_{\iota(i), s}, \; \; \text{for all $1 \le i \le m$}, \\
	\sigma_s(a) &= a, \; \; \text{for all $s \in \fF_s$}.
\end{cases}
\end{align}

Let $\Gamma : G \to \Gal(\fH_s/\fF_s)$ be the map defined by
\begin{align*}
\Gamma(\sigma) = \sigma_s
\end{align*}
for each $\sigma \in G$, where $\sigma_s$ is defined as in (\ref{e4-main-thm3-Galois-structures-of-shadows}). We prove that $\Gamma$ is a group isomorphism. For any elements $\sigma, \delta \in G$, let $\iota_{\sigma}, \iota_{\delta}$ be the unique permutations of $\{1, \ldots, m\}$ that define $\sigma, \delta$, respectively. We see that
\begin{align*}
\sigma\delta(\alpha_i) = \sigma(\alpha_{\iota_{\delta}(i)}) = \alpha_{\iota_{\sigma}(\iota_{\delta}(i))}
\end{align*}
for all $1 \le i \le m$, and thus
\begin{align*}
\iota_{\sigma\delta} = \iota_{\sigma}\iota_{\delta},
\end{align*}
where $\iota_{\sigma\delta}$ is the unique permutation of $\{1, \ldots, m\}$ that defines $\sigma\delta$. Thus in setting $\gamma = \sigma\delta \in G$, we deduce that
\begin{align*}
\gamma_s(\alpha_{i, s}) &= \alpha_{\iota_{\gamma}(i), s} \\
&= \alpha_{\iota_{\sigma\delta}(i), s} \\
&= \alpha_{\iota_{\sigma}(\iota_{\delta}(i)), s}\\
&= \sigma_s(\delta_s(\alpha_{i, s}))
\end{align*} 
for all $1 \le i \le m$. Therefore
\begin{align*}
\gamma_s = \sigma_s\delta_s,
\end{align*}
which implies that 
\begin{align*}
\Gamma(\sigma\delta) = \Gamma(\sigma)\Gamma(\delta).
\end{align*}
Hence $\Gamma$ is a group morphism. 

We contend that $\Gamma$ is injective. Indeed, let $\sigma$ be an element in $G$ such that $\Gamma(\sigma) = 1_{\Gal(\fH_s/\fF_s)}$. Note that the permutation of $\{1, \ldots, m\}$ that defines $1_{\Gal(\fH_s/\fF_s)}$ is the trivial one, i.e., it sends $i$ to itself for every $1 \le i \le m$. Since
\begin{align*}
\Gamma(\sigma) = \sigma_s = 1_{\Gal(\fH_s/\fF_s)},
\end{align*}
we deduce that
\begin{align*}
\iota_{\sigma}(i) = i
\end{align*}
for every $1 \le i \le m$, where $\iota_{\sigma}$ is the unique permutation of $\{1, \ldots, m\}$ that defines $\sigma$. Thus $\iota_{\sigma}$ is the trivial permutation, and therefore $\sigma = 1_G$. Therefore $\Gamma$ is injective. 

Since both $G$ and $\Gal(\fH_s/\fF_s)$ have the same cardinality $m$, $\Gamma$ must be surjective, and thus it is a group isomorphism, which proves part (i).

We now prove part (ii). By Proposition \ref{prop-explicit-descriptions-of-shadows-and-ultra-shadows}, and since $\H = \F(\alpha_i)$ for every $1 \le i \le m$, $\fH$ is a finite separable extension of degree $m$ over $\cU(\F)$ of the form $\fH = \cU(\F)(\alpha_i)$ for every $1 \le i \le m$, and thus
\begin{align}
\label{e5-main-thm3-Galois-structures-of-shadows}
\fH = \cU(\F)(\alpha_1, \ldots, \alpha_m).
\end{align}

Since $P(x)$ is irreducible in $\F[x]$, we know from Lemma \ref{rem-irreducibility-is-the-same-in-F-and-U(F)}, $P(x)$ is irreducible in $\cU(\F)$. let $\fG$ denote the splitting field of $P(x)$ over $\cU(\F)$. Thus $\fG$ contains all $m$ distinct roots $\alpha_1, \ldots, \alpha_m$ of $P(x)$, and it therefore follows from (\ref{e5-main-thm3-Galois-structures-of-shadows}) that
\begin{align*}
\fH = \cU(\F)(\alpha_1, \ldots, \alpha_m) \subset \fG.
\end{align*}
Hence $\fG = \fH$ since $\fG$ is the smallest field extension of $\cU(\F)$ over which $P(x)$ splits into linear factors, and $\fH$ contains all the roots of $P(x)$. Thus $\fH$ is normal over $\cU(\F)$, and therefore $\fH$ is a Galois extension of degree $m$ over $\cU(\F)$.

It remains to show that the Galois group $\Gal(\fH/\cU(\F))$ is isomorphic to $G$.

Since $\H, \fH$ are the Galois extensions of $\F, \cU(\F)$, respectively, generated by the same set of generators $\alpha_1, \ldots, \alpha_m$ that are exactly all the $m$ roots of $P(x)$, we can associate the elements in their corresponding Galois groups, based on the permutations of the generators. Indeed, take an arbitrary element $\delta \in \Gal(\fH/\cU(\F))$. Let $\iota_{\delta}$ denote the unique permutation of $\{1, \ldots, m\}$ such that
\begin{align*}
\delta(\alpha_i) = \alpha_{\iota_{\delta}(i)}
\end{align*}
for all $1 \le i \le m$. Let $\sigma_{\iota_{\delta}}$ be the unique element in the Galois group $G = \Gal(\H/\F)$ such that
\begin{align*}
\sigma_{\iota_{\delta}}(\alpha_i) = \alpha_{\iota_{\delta}(i)}
\end{align*}
for all $1 \le i \le m$, and
\begin{align*}
\sigma_{\iota_{\delta}}(a) = a
\end{align*}  	
for all $a \in \F$.

We contend that the map $\Sigma : \Gal(\fH/\cU(\F)) \to G$ that sends each element $\delta \in \Gal(\fH/\cU(\F))$ to $\sigma_{\iota_{\delta}} \in G$, is a group isomorphism. 
  	
  Let $\delta, \lambda$ be arbitrary elements in $\Gal(\fH/\cU(\F))$. We see that
  \begin{align*}
  \delta\lambda(\alpha_i) &= \delta(\lambda(\alpha_i))  \\
  &= \delta(\alpha_{\iota_{\lambda}(i)}) \\
  &= \alpha_{\iota_{\delta}(\iota_{\lambda}(i))}
  \end{align*}	
  for all $1 \le i \le m$. We know from the definition of $\iota_{\delta\lambda}$ that
  \begin{align*}
\delta\lambda(\alpha_i) = \alpha_{\iota_{\delta\lambda}}
  \end{align*}
  for all $1 \le i \le m$, and thus
   \begin{align*}
    \iota_{\delta\lambda}(i) &= \iota_{\delta}( \iota_{\lambda}(i)
  \end{align*}
  for all $1 \le i \le m$. Therefore 
  \begin{align} 	
  \label{e6-main-thm3-Galois-structures-of-shadows}
\iota_{\delta\lambda} = \iota_{\delta}\iota_{\lambda}.
  \end{align}

We see that
  \begin{align*}
  \sigma_{\iota_{\delta\lambda}}(\alpha_i) &= \alpha_{\iota_{\delta\lambda}(i)} \\
  &= \alpha_{\iota_{\delta}(\iota_{\lambda}(i))} \; \; \text{see (\ref{e6-main-thm3-Galois-structures-of-shadows}}) \\
  &= \sigma_{\iota_{\delta}}\left(\alpha_{\iota_{\lambda}(i)}\right) \\
&= \sigma_{\iota_{\delta}}\left(\sigma_{\iota_{\lambda}}(\alpha_i)\right)  \\
&= \sigma_{\iota_{\delta}}\sigma_{\iota_{\lambda}}(\alpha_i)
  \end{align*}
  for all $1 \le i \le m$. Since $\alpha_1, \ldots, \alpha_m$ generates the Galois extension $\H$ and it is trivial that $\sigma_{\iota_{\delta\lambda}}$ and $\sigma_{\iota_{\delta}}\sigma_{\iota_{\lambda}}$ fix every element of $\F$, we deduce that
  \begin{align*}
  \sigma_{\iota_{\delta\lambda}} = \sigma_{\iota_{\delta}}\sigma_{\iota_{\lambda}}, 
  \end{align*}
  which immediately implies  
  \begin{align*}
  \Sigma(\delta\lambda) = \Sigma(\delta)\Sigma(\lambda).
  \end{align*}
  Thus $\Sigma$ is a group homomorphism.

Let $\delta \in \Gal(\fH/\cU(\F))$ such that $\Sigma(\delta) = 1_G$, where $1_G$ is the identity element in $G = \Gal(\H/\F)$, i.e., $1_G(a) = a$ for all $a \in \H$. Thus
\begin{align*}
\sigma_{\iota_{\delta}} = 1_G,
\end{align*}
which implies that 
\begin{align*}
\alpha_{\iota_{\delta}(i)} = \sigma_{\iota_{\delta}}(\alpha_i) = 1_G(\alpha_i) = \alpha_i
\end{align*}
for all $1 \le i \le m$. Thus $\iota_{\delta}$ is the trivial permutation of $\{1, \ldots, m\}$, i.e., $\iota_{\delta}(i) = i$ for all $1 \le i \le m$, and therefore $\delta$ is the trivial element in the Galois group $\Gal(\fH/\cU(\F))$. Thus $\Sigma$ is injective. Since both $\Gal(\fH/\cU(\F))$ and $G$ have $m$ elements, $\Sigma$ is surjective, and thus $\Sigma$ is an isomorphism.

We now prove part (iii). 

For an arbitrary element $\delta \in \Gal(\fH/\cU(\F))$, we know from part (ii) that
\begin{align*}
\Sigma(\delta) = \sigma_{\iota_{\delta}},
\end{align*}
where $\iota_{\delta}$ is the unique permutation of $\{1, \ldots, m\}$ such that 
\begin{align*}
\delta(\alpha_i) = \alpha_{\iota_{\delta}(i)}
\end{align*}
for all $1 \le i \le m$. Thus
\begin{align}
  \label{e7-main-thm3-Galois-structures-of-shadows}
\delta(\alpha_i) = \alpha_{\iota_{\delta}(i)} = \sigma_{\iota_{\delta}}(\alpha_i)
\end{align}
for all $1 \le i \le m$. 

On the other hand, since $\F \subset \cU(\F)$, $\delta(a) = a = \sigma_{\iota_{\delta}}(a)$ for all $a \in \F$. Since $\H = \F(\alpha_1, \ldots, \alpha_m)$, we deduce from (\ref{e7-main-thm3-Galois-structures-of-shadows}) that $\delta(\beta) = \sigma_{\iota_{\delta}}(\beta)$ for any element $\beta \in \H$, and thus since $\H = \fH_{\alg} \subset \fH$, it follows that
\begin{align*}
\sigma_{\iota_{\delta}} = \delta_{|\H}.
\end{align*}
Therefore
\begin{align*}
G = \Sigma(\Gal(\fH/\cU(\F)) = \{\sigma_{\iota_{\delta}} \; |\; \delta \in \Gal(\fH/\cU(\F))\} = \{\delta_{|\H} \; |\; \delta \in \Gal(\fH/\cU(\F))\},
\end{align*}
which proves the first assertion of part (iii). The last assertion of part (iii) follows immediately from Proposition \ref{prop-ultra-Galois-extensions-are-Galois}.

  \end{proof}
  
  \begin{remark}
  \label{rem-s-th-components-of-a-Galois-element}
  
  The Galois elements $\sigma_s \in \Gal(\fH_s/\fF_s)$ defined by (\ref{e4-main-thm3-Galois-structures-of-shadows}) are called the \textbf{$s$-th components of $\sigma \in \Gal(\H/\F)$}.

  \end{remark}
\begin{proposition}
\label{prop-shadows-of-the-algebraic-part-are-inside}

Let $\fL_s$ be a (possibly infinite) separable algebraic extension of $\fF_s$ for $\cD$-almost all $s \in S$, and let $\fL = \prod_{s\in S}\fL_s/\cD$. Let $\H$ be a subextension of $\fL_{\alg}$ that is a finite separable extension of $\F$. Let $\fH_s$ denote the $s$-th shadow of $\H$ for $\cD$-almost all $s \in S$, and let $\fH  = \prod_{s\in S}\fH_s/\cD$ be the ultra-shadow of $\H$. Then

\begin{itemize}

\item []

\item [(i)] $\fH_s$ is a subfield of $\fL_s$ for $\cD$-almost all $s \in S$.

\item [(ii)] $\fH$ is a subfield of $\fL$.

\end{itemize}

\end{proposition}

\begin{proof}

Since $\fH = \prod_{s\in S}\fH_s/\cD$ and $\fL = \prod_{s\in S}\fL_s/\cD$, part (ii) follows trivially from part (i).

For part (i), the primitive element theorem implies that there exists an element $\alpha \in \H$ such that $\H = \F(\alpha)$. Since $\alpha \in \H \subset \fL_{\alg} \subset \fL$, one can write $\alpha = \ulim_{s\in S}\alpha_s$ for some elements $\alpha_s \in \fL_s$. By part (i) of Theorem \ref{thm-main-thm3-Galois-group-structures-of-shadows}, and since $\fL_s \subset \fF_s^{\sep}$, we know that $\fH_s = \fF_s(\alpha_s)$ for $\cD$-almost all $s \in S$. Since $\alpha_s  \in \fL_s$ for $\cD$-almost all $s \in S$, we deduce that $\fH_s \subset \fL_s$, which proves part (i).

\end{proof}

  The following result follows immediately from the above proposition and Theorem \ref{thm-main-thm3-Galois-group-structures-of-shadows}.
\begin{corollary}
\label{cor-Galois-group-structure-of-shadows-of-algebraic-part-in-a-given-ultra-extension}

Let $\fL_s$ be a (possibly infinite) Galois extension of $\fF_s$ for $\cD$-almost all $s \in S$, and let $\fL = \prod_{s\in S}\fL_s/\cD$. Assume that $\fL_{\alg}$ is a finite Galois extension of $\F$ \footnote{By Theorem \ref{thm-main-thm2-Galois-property-of-algebraic-parts}, $\fL_{\alg}$ is a Galois extension of $\F$, but it may be an infinite extension of $\F$.}. Let $\fS_s$ denote the $s$-th shadow of $\fL_{\alg}$ for $\cD$-almost all $s \in S$, and let $\fS = \prod_{s\in S}\fS_s/\cD$ denote the ultra-shadow of $\fL_{\alg}$. Then
\begin{itemize}

\item [(i)] $\fS_s/\fF_s$ is a finite Galois subextension of $\fL_s/\fF_s$ for $\cD$-almost all $s \in S$ whose Galois group $\Gal(\fS_s/\fF_s)$ is isomorphic to $\Gal(\fL_{\alg}/\F)$.

\item [(ii)] $\fS$ is a subfield of $\fL$ such that $\fS_{\alg} = \fL_{\alg}$, and $\fS$ is a finite Galois extension of $\cU(\F)$ whose Galois group $\Gal(\fS/\cU(\F))$ is isomorphic to $\Gal(\fL_{\alg}/\F)$.

\end{itemize}

\end{corollary}

\begin{remark}
\label{rem-relation-between-s-th-component-of-lambda-and-its-image-in-Gal(fS/cU(F))}

Following the proof of Theorem \ref{thm-main-thm3-Galois-group-structures-of-shadows}, we see that $\Sigma : \Gal(\fS/\cU(\F)) \to \Gal(\fL_{\alg}/\F)$ that sends each $\delta$ to $\sigma_{\iota_{\delta}}$ is a group isomorphism, where $\iota_{\delta}$ is the unique permutation of $\{1, \ldots, m\}$ that determines $\delta$. Furthermore, for every element $\lambda \in \Gal(\fL_{\alg}/\F)$, it is easy to verify that
\begin{align*}
\Sigma^{-1}(\lambda) = \ulim_{s\in S}\lambda_s \in \Gal(\fS/\cU(\F)),
\end{align*}
where $\lambda_s \in \Gal(\fS_s/\fF_s)$ is the $s$-th component of $\lambda$ as in Remark \ref{rem-s-th-components-of-a-Galois-element}.

\end{remark}

\begin{corollary}
\label{cor-main-corollary-1-abelian-Galois-group-of-L-alg-and-shadows}

Assume that $\fL_s$ is a (possibly infinite) abelian extension over $\fF_s$ for $\cD$-almost all $s \in S$. Then $\fL_{\alg}$ is an abelian extension over $\F$.

\end{corollary}

\begin{proof}

By Theorem \ref{thm-main-thm2-Galois-property-of-algebraic-parts}, $\fL_{\alg}$ is a Galois extension over $\F$. Note that $\fL_{\alg}$ may be an infinite extension of $\F$.

We can write 
\begin{align*}
\fL_{\alg} = \bigcup_{i \in I}\H_i,
\end{align*}
where $\H_i/\F$ ranges over every finite Galois subextension of $\fL_{\alg}/\F$. We know that
\begin{align*}
\Gal(\fL_{\alg}/\F) = \varprojlim\Gal(\H_i/\F),
\end{align*}
where $\Gal(\fL_{\alg}/\F)$ is the Galois group of $\fL_{\alg}$ over $\F$, and $\Gal(\H_i/\F)$ denotes the Galois group of $\H_i$ over $\F$. 

Take an arbitrary finite Galois subextension $\H_i/\F$ of $\fL_{\alg}/\F$. Let $\fH_{i, s}$ denote the $s$-th shadow of $\F$ for $\cD$-almost all $s \in S$, and $\fH_i = \prod_{s\in S}\fH_{i, s}/\cD$ denotes the ultra-shadow of $\H_i$. By Proposition \ref{prop-shadows-of-the-algebraic-part-are-inside}, $\fH_{i, s} \subset \fL_s$ for $\cD$-almost all $s \in S$. By Theorem \ref{thm-main-thm3-Galois-group-structures-of-shadows}, $\Gal(\H_i/\F)$ is isomorphic to the Galois group $\Gal(\fH_{i, s}/\fF_s)$ for $\cD$-almost all $s \in S$. 

Since $\fH_{i, s}/\fF_s$ is a Galois subextension of $\fL_s/\fF_s$ and $\fL_s$ is abelian over $\fF_s$, $\fH_{i, s}$ is an abelian extension of $\fF_s$. Thus $\Gal(\fH_{i, s}/\fF_s)$ is an abelian group, and therefore $\Gal(\H_i/\F)$ is also abelian. Thus the profinite Galois group $\Gal(\fL_{\alg}/\F) = \varprojlim\Gal(\H_i/\F)$ is also abelian, which proves the corollary.

\end{proof}

Using similar arguments as in the proof of the above corollary, we immediately obtain the following.

\begin{corollary}
\label{cor-main-corollary-1-procyclic-Galois-group-of-L-alg-and-shadows}

Assume that $\fL_s$ is a cyclic extension of degree $n_s \in \bZ_{>0}$ over $\fF_s$ for $\cD$-almost all $s \in S$. Then 
\begin{itemize}

\item [(i)] $\fL_{\alg}$ is a Galois extension of $\F$, and $\Gal(\fL_{\alg}/\F)$ is a procyclic group.

\item [(ii)] if $\fL_{\alg}$ is a finite extension of $\F$, then $\fL_{\alg}$ is a cyclic extension of $\F$.

\end{itemize}

\end{corollary}

We obtain the following result that describes a relation between the maximal abelian extensions of $\fF_s = \bF_s(t)$ and the maximal abelian extension of $\F = \fK(t)$.

\begin{corollary}
\label{cor-relation-between-maximal-abelian-extensions-of-F_s-and-F}

Let $\cA_s$ be the maximal abelian extension of $\fF_s = \bF_s(t)$, and let $\cA = \prod_{s\in S}\cA_s/\cD$ be the ultraproduct of the $\cA_s$ with respect to $\cD$. Then the algebraic part of $\cA$ over $\F$, i.e., $\cA_{\alg} = \cA \cap \F^{\alg}$ is the maximal abelian extension of $\F$.

\end{corollary}

\begin{proof}

We know from Corollary \ref{cor-main-corollary-1-abelian-Galois-group-of-L-alg-and-shadows} that the algebraic part $\cA_{\alg}$ of $\cA= \prod_{s\in S}\cA_s/\cD$ over $\F = \fK(t)$ is an abelian extension of $\F$.

We prove that $\cA_{\alg}$ is the maximal abelian extension of $\F$. Indeed, take an arbitrary finite abelian extension, say $\H$, of $\F$. Let $\fH_s$ be the $s$-th shadow of $\H$ for $\cD$-almost all $s \in S$, and let $\fH = \prod_{s\in S}\fH_s/\cD$ be the ultra-shadow of $\H$. 

By Theorem \ref{thm-main-thm3-Galois-group-structures-of-shadows}, the Galois group $\Gal(\fH_s/\fF_s)$ is isomorphic to $\Gal(\H/\F)$ for $\cD$-almost all $s \in S$. Since $\Gal(\H/\F)$ is a finite abelian group, $\Gal(\fH_s/\fF_s)$ is also a finite abelian group, and thus $\fH_s$ is a finite abelian extension of $\fF_s$ for $\cD$-almost all $s \in S$. Since $\cA_s$ is the maximal abelian extension of $\fF_s$ for $\cD$-almost all $s \in S$, $\fH_s$ is contained in $\cA_s$ for $\cD$ for $\cD$-almost all $s \in S$. Thus the ultra-shadow $\fH = \prod_{s\in S}\fH_s/\cD$ is contained in $\cA = \prod_{s\in S}\cA_s/\cD$. Since the algebraic part of the ultra-shadow $\fH$ over $\F$ is $\H$, i.e., $\fH_{\alg} = \fH \cap \F^{\alg} = \H$ (see Definition \ref{def-shadows-of-algebraic-extensions}), we deduce that
\begin{align*}
\H = \fH_{\alg} \subset \cA_{\alg},
\end{align*}
which proves that $\cA_{\alg}$ contains $\H$ as a subfield. Since $\H$ is an arbitrary finite abelian extension of $\F$, $\cA_{\alg}$ is the maximal abelian extension of $\F$.

\end{proof}

\subsection{Examples of algebraic parts}
\label{subsec-examples-of-algebraic-parts}

In this subsection, we list several examples of algebraic parts.

\subsubsection{An example in which $\fL_{\alg}$ is a finite separable extension of degree $n$ over $\F$ for each integer $n > 0$.}

In this subsection, we provide an example in which $\fL_{\alg}$ is a finite separable extension of degree $n$ over $\F$ for each integer $n > 0$. We begin by recalling the following. 

\begin{lemma}
	\label{lem-halter-koch}
	(see Halter--Koch \cite[Theorem 1.7.9]{Halter-Koch} or Lang \cite{lang-algebra})
	
	Let $F$ be a field, and let $a$ be an element in $F$. Let $n$ be a positive integer. Then the polynomial $x^n - a$ is irreducible over $F$ if and only if $a \not\in F^p$ for all primes $p$ dividing $n$ and $a \not\in -4F^4$ whenever $4$ divides $n$.
	
\end{lemma}

Throughout this example, let $\bF_s$ denote the finite field of $q_s$ elements, where $q_s$ is a power of a prime $p_s > 0$ for $\cD$-almost all $s \in S$, and let $\fK = \prod_{s\in S}\bF_s/\cD$. As in previous sections, let $\fA_s = \bF_s[t]$, $\fF_s = \bF_s(t)$ and let $\A = \fK[t]$ and $\F = \fK(t)$. We further assume that the primes $p_s$ are distinct for $\cD$-almost all $s \in S$, which implies that $\fK$ is of characteristic $0$.

\begin{corollary}
	
	For any positive integer $n$, there exists an ultra-field extension $\fL = \prod_{s\in S}\fL_s/\cD$ of $\cU(\F) = \prod_{s\in S}\fF_s/\cD$ such that $\fL_{\alg} = \fL \cap \F^{\alg}$ is a finite separable extension of degree $n$ over $\F$.

\end{corollary}

\begin{proof}

	For $\cD$-almost all $s \in S$, choose a polynomial $a_s \in \fA_s = \bF_s[t]$ such that $a \not\in \fF_s^p$ for all primes $p$ dividing $n$ and $a \not\in -4\fF_s^4$ whenever $4$ divides $n$, and the degree of the $a_s$ is bounded for $\cD$-almost all $s \in S$, i.e., there exists an absolute constant $M > 0$ such that 
\begin{align*}
\{s \in S\; | \; \deg(a_s) < M\} \in \cD.
\end{align*}
There are infinitely many choices of $a_s$; for example, one can choose $a_s$ as an irreducible polynomial in $\bF_s[t]$ of degree bounded by an absolute constant for $\cD$-almost all $s \in S$. Thus we deduce from Lemma \ref{lem-halter-koch} that the polynomial $P_s(x) = x^n - a_s \in \fA_s[x]$ is irreducible over $\fF_s$. Since there are only finitely many primes dividing $n$, the polynomial $P_s$ is separable for $\cD$-almost all $s \in S$. 
	
	For each $s \in S$, let $\alpha_s \in \fF_s^{\alg}$ be a root of $P_s$ in the algebraic closure $\fF_s^{\alg}$ of $\fF_s$. Let $\fL_s = \fF_s(\alpha_s)$ be the field obtained from $\fF_s$ by adjoining $\alpha_s$ to $\fF_s$. Thus $\fL_s$ is a finite separable extension of degree $n$ over $\fF_s$ for $\cD$-almost all $s \in S$. 
      
     Let $\fL = \prod_{s\in S}\fL_s/\cD$ be the ultraproduct of $\fL_s$. We contend that $\fL_{\alg} = \fL \cap \F^{\alg}$ is a finite separable extension of degree $n$ over $\F$. Indeed, let $P(x) = \ulim_{s\in S}P_s(x) = x^n - a$, where $a = \ulim_{s\in S}a_s$. Since the degree of $a_s$ is bounded by an absolute constant for $\cD$-almost all $s \in S$, we see from Lemma \ref{lem-elementary-lemma} that $a$ belongs to $\A$. Thus $P(x)$ is a polynomial of degree $n$ in $\A[x]$. By \L{}o\'s' theorem, and since the $P_s$ are separable and irreducible over $\fF_s$ for $\cD$-almost all $s \in S$ and $\F$ is of characteristic $0$, Lemma \ref{lem-elementary-lemma1} implies that $P(x)$ is irreducible and separable over $\F$. Note that $\alpha = \ulim_{s\in S}\alpha_s \in \fL$ is a root of $P(x)$ since $P(\alpha) = \ulim_{s\in S}P_s(\alpha_s) = 0$. Thus $\alpha \in \fL \cap \F^{\alg} = \fL_{\alg}$. Hence the field $\F(\alpha)$ is a finite separable extension of degree $n$ over $\F$. We know that $\F(\alpha)$ is a subfield of $\fL_{\alg}$ and thus $n$ divides $m$, where $m$ is the degree of $\fL_{\alg}$ over $\F$. By Remark \ref{rem-divisibility-of-global-degree-and-ultra-degrees}, $m$ also divides $n$ (that is the degree of $\fL_s$ over $\fF_s$ for $\cD$-almost all $s \in S$). Thus $n = m$, and so $\fL_{\alg}$ is a finite separable extension of degree $n$ over $\F$, as required.

\end{proof}

\subsubsection{An example in which $\fL_{\alg}$ is an infinite Galois extension of $\F$ whose Galois group $\Gal(\fL_{\alg}/\F) \cong \widehat{\bZ}$}
\label{subsubsec-L_alg-is-Z-hat-profinite-Galois-over-F}

Let $S = \bZ_{>0}$ be the set of positive integers, and let $\cD$ be a nonprincipal ultrafilter on $S$. For each $s \in S$, let $p_s = s!r_s + 1 = (1\cdot 2\cdots s)r_s + 1$ for some integer $r_s$. By Dirichlet's theorem on primes in arithmetic progressions (see Serre \cite{serre-arithmetic}), one can choose infinitely many integers $r_s$ such that the $p_s$ are distinct primes. Let $(h_s)_{s \in S}$ be a sequence of positive integers, and for each positive integer $s \in S$, set $q_s = p_s^{h_s}$. 

By our construction of the primes $p_s$, it is straightforward that for any positive integer $n > 0$, $n$ divides $p_s - 1$ for all $s \ge n$, and thus $n$ divides $q_s - 1 = p_s^{h_s} - 1$ for all $s \ge n$.

\begin{lemma}
\label{lem-example-L_alg-infinite-extension-of-F}

For every positive integer $n > 0$, $n$ divides $q_s - 1$ for all $s \ge n$. In particular, for every positive integer $n > 0$,
\begin{align*}
\{s \in S\; | \; \text{$n$ divides $q_s - 1$}\} \in \cD.
\end{align*}

\end{lemma}

Throughout this example, for each $s \in S = \bZ_{>0}$, let $\bF_s$ be the finite field of $q_s$ elements, and let $\fK = \prod_{s\in S}\bF_s/\cD$ be the $1$st level ultra-finite field (see Subsection \ref{subsec-nth-level-ultra-finite-fields}). Note that in this case, $\fK$ is of characteristic $0$ since the $p_s$ are distinct primes.

Take an arbitrary monic irreducible polynomial $P$ of degree $d > 0$ in $\A = \fK[t]$. By Lemma \ref{lem-elementary-lemma}, one can write $P = \ulim_{s\in S}P_s$, where $P_s$ is a monic irreducible polynomial of degree $d$ in $\fA_s = \bF_s[t]$ for $\cD$-almost all $s \in S$. 

Let $\fL_s$ denote the $P_s$-th cyclotomic function field that is obtained from $\fF_s = \bF_s(t)$ by adjoining the set of $P_s$-torsion elements of the Carlitz module $C^{(s)}$ for $\fF_s$ (see Rosen \cite{rosen} or Subsection \ref{subsec-KW-theorem-for-Fq[t]} below for a review of this theory). It is known that the Galois group $\Gal(\fL_s/\fF_s)$ is a cyclic group of order $q_s^d - 1$ (see Hayes \cite[Proposition 2.2]{hayes}). Thus for every positive integer $n$ dividing $q_s^d - 1$, there exists a unique subfield $\fH_s(n)$ of $\fL_s$ such that $[\fH_s(n) : \fF_s] = n$. We recall the following known result whose proof for the rational function field $\bF_q(t)$ over a finite field $\bF_q$ can be found in Zhao--Wu \cite{Zhao-Wu}. 

\begin{proposition}
\label{prop-unique-subfield-of-P_s-cyclotomic-function-field-of-degree-n}

Let $n$ be a positive integer dividing $q_s^d - 1$ for $\cD$-almost all $s \in S$. Then for $\cD$-almost all $s \in S$, the unique subfield $\fH_s(n)$ of $\fL_s$ that is of degree $n$ over $\fF_s$, coincides with $\fF_s(P_s^{1/n})$ or $\fF_s((-P_s)^{1/n})$ according as the degree $d$ of $P_s$ is even or odd.

\end{proposition}

We now characterize $\fL_{\alg} = \fL\cap \F^{\alg}$, where $\fL = \prod_{s\in S}\fL_s/\cD$. By Corollary \ref{cor-main-corollary-1-abelian-Galois-group-of-L-alg-and-shadows}, $\fL_{\alg}$ is an abelian extension of $\F$. 

\begin{proposition}
\label{prop-structure-of-subfields-of-L-alg-ininite-over-F}
\begin{itemize}

\item []

\item [(i)] For every positive integer $n$, there exists a unique subfield $\fH(n)$ of $\fL_{\alg}$ that is of degree $n$ over $\F$ such that $\fH(n) = \F((\epsilon(d)P)^{1/n})$, where 
\begin{align*}
\epsilon(d) =
\begin{cases}
1 \; \; &\text{if $d$ is even,} \\
-1 \; \; &\text{if $d$ is odd.}
\end{cases}
\end{align*}
In particular, $\fH(n)$ is a cyclic extension of degree $n$ over $\F$, and $\Gal(\fH(n)/\F) \cong \bZ/n\bZ$.

\item [(ii)] $\fL_{\alg} = \bigcup_{n \in \bZ_{>0}}\fH(n)$ and $\Gal(\fL_{\alg}/\F) \cong \widehat{\bZ}$. 

\end{itemize}

\end{proposition}

\begin{proof}

Let $\alpha = (\epsilon(d)P)^{1/n} \in \F^{\alg}$. By Eisenstein's criterion (see Lang \cite{lang-algebra}), the polynomial $f(x) = x^n - \epsilon(d)P$ is irreducible over $\F$, and thus $f(x)$ is the minimal polynomial of $\alpha$ over $\F$, and thus the field $\fH(n) = \F(\alpha) = \F((\epsilon(d)P)^{1/n})$ is an algebraic extension of degree $n$ over $\F$. We see that
\begin{align*}
f(x) = \ulim_{s\in S}f_s(x),
\end{align*}
where $f_s(x) = x^n - \epsilon(d)P_s \in \fA_s[x] = \bF_s[t][x]$, and $\alpha = \ulim_{s\in S}\alpha_s$, where $\alpha_s = (\epsilon(d)P_s)^{1/n} \in \fF_s^{\alg}$. By Lemma \ref{lem-elementary-lemma1} or Eisenstein's criterion, $f_s(x)$ is irreducible over $\fF_s$ for $\cD$-almost all $s \in S$, and thus $f_s(x)$ is the minimal polynomial of $\alpha_s$ for $\cD$-almost all $s \in S$. Thus by Proposition \ref{prop-explicit-descriptions-of-shadows-and-ultra-shadows}, the field $\fH_s(n) = \fF_s(\alpha_s) = \fF_s((\epsilon(d)P_s)^{1/n})$ is the $s$-th shadow of $\fH(n)$ for $\cD$-almost all $s \in S$, and $\fH^{(n)} = \prod_{s\in S}\fH_s(n)/\cD$ is the ultra-shadow of $\fH(n)$ such that $\fH^{(n)}_{\alg} = \fH(n)$. 

Since $\fH_s(n) \subset \fL_s$, $\fH^{(n)} = \prod_{s\in S}\fH_s(n)/\cD$ is a subset of $\fL = \prod_{s\in S}\fL_s/\cD$. Thus
\begin{align*}
\fH^{(n)}_{\alg} = \fH(n) \subset \fL_{\alg},
\end{align*}
which proves that $\fH(n)$ is an algebraic subextension of $\fL_{\alg}$ that is of degree $n$ over $\F$.

If there exists another subextension, say $\fG(n)$ of $\fL_{\alg}$ that is of degree $n$ over $\F$. Then the $s$-th shadow $\fG_s$ of $\fG(n)$ is a subfield of $\fL_s$ that is of degree $n$ over $\fF_s$. By Lemma \ref{lem-example-L_alg-infinite-extension-of-F}, $n$ divides $q_s^d - 1$ for $\cD$-almost all $s \in S$, and thus the subfield $\fH_s(n)$ of $\fL_s$ is a unique extension of degree $n$ over $\fF_s$ for $\cD$-almost all $s \in S$. Therefore $\fG_s = \fH_s(n)$ since both fields are algebraic extensions of degree $n$ over $\fF_s$ for $\cD$-almost all $s \in S$. By Definition \ref{def-shadows-of-algebraic-extensions} and Proposition \ref{prop-shadows-of-the-algebraic-part-are-inside}, we deduce that
\begin{align*}
\fG(n) = \fG_{\alg} = \fH^{(n)}_{\alg} = \fH(n),
\end{align*}
where $\fG = \prod_{s\in S}\fG_s/\cD$ is the ultra-shadow of $\fG(n)$. Thus $\fH(n)$ is the unique subfield of $\fL_{\alg}$ that is of degree $n$ over $\F$. Therefore $\fH(n)$ is a cyclic extension of degree $n$ over $\F$ whose Galois group $\Gal(\fH(n)/\F) \cong \bZ/n\bZ$ (see Szamuely \cite{szamuely}).

Part (ii) follows immediately from infinite Galois theory (see \cite{szamuely}).

\end{proof}

  \begin{remark}
  
  The field $\fL_{\alg}$ is an example of the \textbf{$P$-th cyclotomic function field for $\F$} that will be introduced in Definition \ref{def-a-th-cyclotomic-function-field-for-ultra-finite-fields}.

  \end{remark}

 \subsubsection{An example in which $\fL_{\alg} = \F$}
 \label{subsubsec-Artin-Schreier-extension-algebraic-parts}

 Let $K$ be a field of characteristic $p > 0$ such that $\bF_p \subset K$. Let $\rho(x) = x^p - x \in \bF_p[x]$. Let $a \in K$, and let $\alpha$ be a root of $\rho(x) - a$ in an algebraic closure $ K^{\alg}$ of $K$ such that $\alpha \not\in K$, then the field $L = K(\alpha)$ is a cyclic extension of degree $p$ over $K$. We call $L$ an \textbf{Artin--Schreier extension of $K$} (see Halter--Koch \cite[Theorem 1.7.10]{Halter-Koch}).

In this example, let $\bF_s$ denote the finite field of $q_s$ elements, where $q_s$ is a power of prime $p_s > 0$ such that the primes $p_s$ are distinct for $\cD$-almost all $s \in S$.

Let $\fL_s$ be an Artin--Schreier extension of $\fF_s$ so that $\fL_s$ is of degree $p_s$ over $\fF_s$ for $\cD$-almost all $s \in S$, and let $\fL = \prod_{s\in S}\fL_s/\cD$. 

\begin{proposition}
\label{prop-L_alg-is-trivial}

$\fL_{\alg} = \F$.

\end{proposition} 
 
 \begin{proof}
 
 We know from Halter--Koch \cite[Theorem 1.7.10]{Halter-Koch} that $\fL_s$ is a cyclic extension of degree $p_s$ over $\fF_s$, and thus Theorem \ref{thm-main-thm2-Galois-property-of-algebraic-parts} implies that $\fL_{\alg}$ is a Galois extension of $\F$.
 
 We contend that $1$ is the only positive integer that divides the hyperinteger $\ulim_{s\in S}p_s$ in $\bZ^{\#}$. Indeed, let $h$ be a positive integer dividing $\ulim_{s\in S}p_s$ in $\bZ^{\#}$. Thus $h$ divides $p_s$ for $\cD$-almost all $s \in S$. Since the primes $p_s$ are distinct for $\cD$-almost all $s \in S$, there exist distinct primes $p_{s_1}, p_{s_2}$ such that $h$ divides both primes $p_{s_1}$ and $p_{s_2}$. If $h > 1$, then $h = p_{s_1}$ and $h = p_{s_2}$, which is a contradiction. Thus $h$ must equal $1$. Thus Corollary \ref{cor-main-cor1-L_alg-is-a-finite-Galois-extension-of-F} implies that $\fL_{\alg}$ is a finite Galois extension of degree $m$ over $\F$ for some positive integer $m$. By Remark \ref{rem-divisibility-of-global-degree-and-ultra-degrees}, $m$ divides $\ulim_{s\in S}p_s$ in $\bZ^{\#}$, and thus $m = 1$. Thus $\fL_{\alg} = \F$ as required.

 \end{proof}

 \subsubsection{A relation between the maximal abelian extensions of rational function fields over algebraically closed fields of characteristics $0$ and $p > 0$.}
 \label{subsubsection-relation-the-maximal-abelian-extension-of-function-fields-overACFs-of-p-and-0-chars}

 The main aim of this subsection is to explain, via the theory of algebraic parts, a relation between the maximal abelian extensions of a family of rational function fields over algebraically closed fields of distinct positive characteristics and the maximal abelian extension of the rational function field over a corresponding algebraically closed field of characteristic $0$.
 
 Throughout this subsection, the constant field $\bF_s$ denotes an algebraically closed field of characteristic $p_s > 0$ such that the $p_s$ are distinct for $\cD$-almost all $s \in S$, i.e., there does not exist a prime $p > 0$ such that $\{s \in S\; | \; p_s = p\} \in \cD$. Thus \L{}o\'s' theorem implies that the ultra-field $\fK = \prod_{s\in S}\bF_s/\cD$ is an algebraically closed field of characteristic $0$. As in previous sections, let $\fF_s = \bF_s(t)$ be the rational function field over $\bF_s$ for $\cD$-almost all $s\in S$, and let $\F = \fK(t)$ be the rational function field over $\fK$.

 There is an explicit description of the maximal abelian extension $\cA_s$ of the rational function field $\fF_s = \bF_s(t)$. Since $\bF_s$ is algebraically closed of characteristic $p_s > 0$, the maximal abelian extension $\cA_s$ of $\fF_s$ is the compositum of the maximal abelian pro-$p_s$ extension $\cB_s$ of $\fF_s$ with the maximal pro-prime-to-$p_s$ extension $\cC_s$ of $\fF_s$. The extension $\cC_s$ is obtained as in characteristic $0$, i.e., $\cC_s$ is obtained from $\fF_s$ by adjoining all the $n$-th roots of the elements $t - a_s$ for all positive integers $n$ that are prime to $p_s$, and all elements $a_s \in \bF_s$. As for the former extension, $\cB_s$ is obtained from $\fF_s$ by taking all possible abelian finite towers of successive Artin--Schreier extensions. Another way to obtain $\cB_s$ is to take the compositum of all $p_s^n$-cyclic extensions of $\fF_s$ for all $n$ that are obtained via Witt vectors (see Serre \cite{serre-local-fields}).  
 
 Turning to the maximal abelian extension $\fM$ of $\F = \fK(t)$, one can obtain $\fM$ from $\F$ by adjoining all the $n$-th roots of the elements $t - a$ for all positive integers $n$, and all elements $a \in \fK$, which is very similar to the construction of the maximal pro-prime-to-$p_s$ extension $\cC_s$ above. 
 
 By Corollary \ref{cor-relation-between-maximal-abelian-extensions-of-F_s-and-F}, one obtains the following relation between the maximal abelian extension $\cA_s$ of $\fF_s$ and the maximal abelian extension $\fM$ of $\F$.
 
 \begin{proposition}
 \label{prop-relation-between-rational-function-fields-over-ACFs-of-arbitrary-characteristics}
 
 Let $\cA = \prod_{s\in S}\cA_s/\cD$, where $\cA_s = \cB_s \cdot \cC_s$ is the maximal abelian extension of $\fF_s = \bF_s(t)$ for $\cD$-almost all $s \in S$. Then the algebraic part of $\cA$ over $\F$ coincides with the maximal abelian extension $\fM$ of $\F$, that is, $\cA_{\alg} = \cA \cap \F^{\alg} = \fM$.

 \end{proposition} 
 
 In view of Proposition \ref{prop-relation-between-rational-function-fields-over-ACFs-of-arbitrary-characteristics}, the maximal abelian extension $\fM$ is built via taking the algebraic part of the ultra-field whose $s$-th component $\cA_s$ is the compositum of the maximal pro-prime-to-$p_s$ extension $\cC_s$ of $\fF_s$ that is constructed using the same procedure as $\fM$, with the maximal abelian pro-$p_s$ extension $\cB_s$ that comes from Artin--Schreier extensions of $\fF_s$. In view of the same construction of the extensions $\cC_s$ and $\fM$, i.e., both obtained from their ground fields by adjoining all the $n$-th roots of all monic linear polynomials over their constant fields for all integers $n$ prime to their constant fields' characteristics,  it is natural to ask \textit{why the structures of the maximal abelian pro-$p_s$ extensions $\cB_s$ disappears from the structure of $\fM$ after taking the algebraic part of the ultra-field extension $\cA$ of $\cU(\F)$?} In what follows, we will briefly explain, from a model-theoretic viewpoint, why $\fM$ does not contain an algebraic part of the ultraproduct of the $\cB_s$, i.e., there is no analogue of Artin--Schreier extensions in characteristic $0$, and provide an insight into why the constructions of the maximal pro-prime-to-$p_s$ extensions $\cC_s$ and $\fM$ are carried out in complete analogy.
 
 Let $\cB = \prod_{s\in S}\cB_s/\cD$ and $\cC = \prod_{s\in S}\cC_s/\cD$. Since $\cB_s \subset \cB_s \cdot \cC_s = \cA_s$ for $\cD$-almost all $s \in S$, we deduce that
 \begin{align*}
 \cB = \prod_{s\in S}\cB_s/\cD \subset \cA = \prod_{s\in S}\cA_s/\cD,
 \end{align*}
 and thus
 \begin{align*}
 \cB_{\alg} \subset \cA_{\alg} = \fM.
 \end{align*}
 
 Similarly, one can show that
\begin{align*}
 \cC_{\alg} \subset \cA_{\alg} = \fM,
 \end{align*}
 and thus
 \begin{align}
 \label{e1-example-relation-between-rational-function-fields-over-ACFs-of-characteristics}
\cC_{\alg} \subset \cB_{\alg}\cC_{\alg} \subset \cA_{\alg} = \fM.
 \end{align}
 
\begin{lemma}
\label{lem-example-relation-between-rational-function-fields-over-ACFs-of-characteristics}

\begin{itemize}
  
  \item []
  
  \item [(i)] $\cB_{\alg} = \F$.
  
  \item [(ii)] $\fM = \cC_{\alg}$.
  
\end{itemize}

\end{lemma}
 
 \begin{proof}
 
 Since $\cB_s$ is obtained from $\fF_s$ by taking all possible abelian finite towers of successive Artin-Schreier extensions for $\cD$-almost all $s \in S$, using the same arguments as in Subsection \ref{subsubsec-Artin-Schreier-extension-algebraic-parts}, we deduce that $\cB_{\alg} = \F$, which verifies part (i).
 
 By (\ref{e1-example-relation-between-rational-function-fields-over-ACFs-of-characteristics}), it suffices to prove that $\fM \subset \cC_{\alg}$.
 
 Take an arbitrary $n$-th root $\zeta_n$ of an element $\alpha = t - a \in \A = \fK[t] \subset \F$ for some positive integer $n$ and some element $a \in \fK$. By Einsenstein's Irreducibility Criterion (see Lang \cite{lang-algebra}), the polynomial $P(x) = x^n - \alpha \in \A[x]$ is irreducible over $\F$, and thus $P(x)$ is the minimal polynomial of the $n$-th root $\zeta_n$ of $\alpha$. 
 
 Since $a \in \fK = \prod_{s\in S}\bF_s/\cD$, we can write $a = \ulim_{s\in S}a_s$ for some elements $a_s \in \bF_s$ for $\cD$-almost all $s \in S$. Letting $\alpha_s = t - a_s \in \fA_s = \bF_s[t]$, we see that $P = \ulim_{s\in S}P_s$, where $P_s(x) = x^n - \alpha_s$. Thus for $\cD$-almost all $s \in S$, one can choose a suitable $n$-th root $\zeta_{n, s}$ of $\alpha_s$ such that $\zeta_n = \ulim_{s\in S}\zeta_{n, s}$ and $P_s(x)$ is the minimal polynomial of $\zeta_{n, s}$ (see Lemma \ref{lem-elementary-lemma1}). Note that for every positive integer $n$, there are only finitely many primes dividing $n$, and thus
 \begin{align*}
 \{s \in S\; | \; \text{$n$ is prime to the characteristic $p_s$ of $\bF_s$}\} \in \cD.
 \end{align*}
 Therefore by the construction of $\cC_s$, $\zeta_{n, s} \in \cC_s$ for $\cD$-almost all $s \in S$, and thus $\zeta_n = \ulim_{s\in S}\zeta_{n, s}$ belongs to $\cC = \prod_{s\in S}\cC_s/\cD$. It is trivial that $\zeta_n \in \F^{\alg}$, and thus $\zeta_n \in \cC \cap \F^{\alg} = \cC_{\alg}$. 
 
 Since $n$-th roots $\zeta_n$ generate the maximal abelian extension $\fM$ of $\F$, we deduce that $\fM \subset \cC_{\alg}$, which proves part (ii).

 \end{proof}

 \begin{remark}
 
 \begin{itemize}
 
 \item []
 
 \item [(i)] In view of part (ii) of the above lemma, the maximal abelian extension $\fM$ of $\F$ only depends on the algebraic part of the ultraproduct $\cC$ of the maximal pro-prime-to-$p_s$ extensions $\cC_s$, which explains why the construction of $\fM$ is carried out in complete analogy with the constructions of the extensions $\cC_s$ although $\fM$ is of characteristic $0$ and the extensions $\cC_s$ are of positive characteristics.  
 
 \item [(ii)] By the definitions of $\cB_s$ and $\cC_s$, $\cB_s \cap \cC_s = \fF_s$ for $\cD$-almost all $s \in S$, and thus $\cB \cap \cC = \prod_{s\in S}\fF_s/\cD = \cU(\F)$. Thus it follows from Lemma \ref{lem-U(F)-algebraic-part-is-F} that
     \begin{align}
     \label{e2-example-relation-between-rational-function-fields-over-ACFs-of-characteristics}
     \cB_{\alg} \cap \cC_{\alg} = \cU(\F)_{\alg} = \F.
     \end{align}
     
     By part (ii) of Lemma \ref{lem-example-relation-between-rational-function-fields-over-ACFs-of-characteristics} and (\ref{e1-example-relation-between-rational-function-fields-over-ACFs-of-characteristics}), we deduce that
     \begin{align*}
     \cC_{\alg} = \cB_{\alg}\cC_{\alg} = \fM,
     \end{align*}
     which implies that $\cB_{\alg} \subset \cC_{\alg}$. Thus we deduce from (\ref{e2-example-relation-between-rational-function-fields-over-ACFs-of-characteristics}) that $\cB_{\alg}$ must coincide with $\F$, which provides another proof to part (i) of Lemma \ref{lem-example-relation-between-rational-function-fields-over-ACFs-of-characteristics}.

 \item [(iii)] In view of Proposition \ref{prop-relation-between-rational-function-fields-over-ACFs-of-arbitrary-characteristics} and part (ii) of Lemma \ref{lem-example-relation-between-rational-function-fields-over-ACFs-of-characteristics}, we know that 
     \begin{align*}
     \fM = \cA_{\alg} = \cC_{\alg}.
     \end{align*}
     
     The above equations signify that although $\fM$ can be obtained by taking the algebraic part of the ultra-field $\cA$ whose $s$-th component is the compositum of $\cB_s$ with $\cC_s$ for $\cD$-almost all $s \in S$, the resulting maximal abelian extension $\fM$ of $\F$ only coincides with the algebraic part $\cC_{\alg}$, which suggests that the structures of the extensions $\cB_s$ that involve Artin--Schreier extensions, completely disappear from the structure of $\fM$ after taking algebraic parts.

 \end{itemize}

 \end{remark}
 
 \subsection{Applications to the Inverse Galois Problem}
 \label{subsec-inverse-Galois-problem}
 
 The inverse Galois problem for a field $K$ concerns whether or not every finite group occurs as the Galois group of some Galois extension of $K$. Let $G$ be a given finite group, and $K$ a field. If there exists a Galois extension $L/K$ whose Galois group $\Gal(L/K)$ is isomorphic to $G$, we say that \textbf{$G$ is realizable over $K$}. The \textbf{inverse Galois problem is solvable for $K$} if every finite group is realizable over $K$.

 The inverse Galois problem is still unsolved over $\bQ$ and over the rational function field $\bF_q(t)$ over a finite field $\bF_q$. 
In this subsection, we apply the theory of shadows developed in Subsection \ref{subsec-shadows-and-ultra-shadows} to study the inverse Galois problem over certain fields. 

 \begin{theorem}
 \label{thm-main-thm-for-inverse-Galois-problem}
 
 \begin{itemize}
 
 \item []
 
 \item [(i)] If a finite group $G$ is realizable over $\F = \fK(t)$, then $G$ is realizable over $\fF_s = \bF_s(t)$ for $\cD$-almost all $s \in S$.

 \item [(ii)] If a finite group $G$ is realizable over $\F = \fK(t)$, then $G$ is realizable over the ultra-hull $\cU(\F) = \prod_{s\in S}\fF_s/\cD$. In particular, if the inverse Galois problem is solvable over $\F$, then the inverse Galois problem is solvable over the ultra-hull $\cU(\F)$.

 \end{itemize}

 \end{theorem}

\begin{proof}

Let $G$ be a finite group that is realizable over $\F$, and let $\H$ be a Galois extension of $\F$ whose Galois group $\Gal(\H/\F)$ is isomorphic to $G$. Let $\fH_s$ be the $s$-th shadow of $\H$ for $\cD$-almost all $s \in S$, and let $\fH = \prod_{s\in S}\fH_s/\cD$ be the ultra-shadow of $\H$. By Theorem \ref{thm-main-thm3-Galois-group-structures-of-shadows}, $\fH_s$ is a finite Galois extension of $\fF_s$ whose Galois group $\Gal(\fH_s/\fF_s)$ is isomorphic to $\Gal(\H/\F)$, and thus $\Gal(\fH_s/\fF_s)$ is isomorphic to $G$ for $\cD$-almost all $s \in S$, which proves part (i).

Similarly, by Theorem \ref{thm-main-thm3-Galois-group-structures-of-shadows}, the ultra-shadow $\fL$ is a finite Galois extension of $\cU(\F)$ whose Galois group $\Gal(\fL/\cU(\F))$ is isomorphic to $\Gal(\H/\F)$, and thus $\Gal(\fL/\cU(\F))$ is isomorphic to $G$. Thus part (ii) follows immediately.

\end{proof}

\subsubsection{An asymptotic version of the inverse Galois problem for the rational function field $\bF_q(t)$ over a finite field $\bF_q$}

In this example, $S$ will be an infinite set of prime powers $q$. For each $q \in S$, let $\bF_q$ denote a finite field of $q$ elements. We will determine a set $S$ in the proof of Corollary \ref{cor-inverse-galois-problem-for-function-fields-over-finite-fields}.

Recall that a field $K$ is \textbf{pseudo-algebraically closed (PAC)}  if each absolutely irreducible variety $V$ defined over $K$ has a $K$-rational point (see \cite{FJ}). 

It is well-known that the ultraproduct $\fK = \prod_{s\in S}\bF_s/\cD$ is PAC (see Ax \cite{ax-1968}). Since $\F = \fK(t)$ is a rational function field over the PAC field $\fK$ of characteristic $0$, we know from Fried--V\"{o}lklein \cite[Theorem 2]{FV} or Pop \cite{pop-1996} that every finite group is realizable over $\F$. Using Theorem \ref{thm-main-thm-for-inverse-Galois-problem}, we obtain another proof to the following.

\begin{corollary}
\label{cor-inverse-galois-problem-for-function-fields-over-finite-fields}
(Fried--V\"{o}lklein, Jarden, Pop)

For every finite group $G$, there exists an absolute constant $N > 0$ such that $G$ is realizable over the rational function field $\bF_q(t)$ over a finite field $\bF_q$ for all prime powers $q > N$. 

\end{corollary}

\begin{proof}

Assume that there exists a finite group $G$ such that the assertion fails for infinitely many prime powers $q$. Let $S$ be the set of such prime powers $q$. Let $\cD$ be an arbitrary nonprincipal ultrafilter on $S$. Then the ultraproduct $\fK = \prod_{q \in S}\bF_q/\cD$ is a PAC of characteristic $0$. By Fried--V\"{o}lklein \cite[Theorem 2]{FV} (see also Pop \cite{pop-1996}), $G$ is realizable over $\F = \fK(t)$. By Theorem \ref{thm-main-thm-for-inverse-Galois-problem}, $G$ is realizable over $\fF_q = \bF_q(t)$ for $\cD$-almost all $q \in S$, i.e.,
\begin{align*}
A = \{q \in S \; | \; \text{$G$ is realizable over $\bF_q(t)$}\}\in \cD.
\end{align*}

Since $\cD$ is a nonpricipal ultrafilter, $A$ must be an infinite subset of $S$, which is a contradiction since $G$ is not realizable over $\bF_q(t)$ for all $q \in S$.

\end{proof}

The above theorem is due to Fried--V\"{o}lklein \cite{FV}, Jarden, and Pop (see \cite[Proposition 3.3.9]{harbater-2003}). In \cite{FV}, Fried and V\"{o}lklein provide two proofs to the above result, one of which uses ultraproducts. Our proof above, in spite of relying on shadows that are constructed using ultraproducts, uses a completely different approach, and provides an explicit construction of a Galois extension over $\bF_q(t)$ for all sufficiently large prime powers $q$ whose Galois group is isomorphic to a given finite group.

\section{Ramifications of primes in the algebraic part of an ultra-field extension}
\label{sec-ramification-of-primes-Hilbert-theory}

Throughout this section, assume that $\fL_s$ is a (possibly infinite) Galois extension over $\fF_s = \bF_s(t)$. Let $\Gal(\fL_s/\fF_s)$ denote the Galois group of $\fL_s$ over $\fF_s$ for $\cD$-almost all $s \in S$, and let $\Gal_{\ultra}(\fL/\cU(\F))$ denote the ultra-Galois group of the ultra-field extension $\fL = \prod_{s\in S}\fL_s/\cD$ over $\cU(\F)$. Let $\fL_{\alg}$ be the algebraic part of $\fL$ over $\F$. By Theorem \ref{thm-main-thm2-Galois-property-of-algebraic-parts}, $\fL_{\alg}$ is a (possibly infinite) Galois extension of $\F$. Let $\Gal(\fL_{\alg}/\F)$ denote the Galois group of $\fL_{\alg}$ over $\F$. The main results in this section provide insights into ramifications of primes of $\F$ in $\fL_{\alg}$. 

Let $\cO_{\fL_{\alg}}$ denote the ring of integers of $\fL_{\alg}$ that consists of all elements $\alpha \in \fL_{\alg}$ such that the minimal polynomial of $\alpha$ is a monic, irreducible polynomial with coefficients in $\A = \fK[t]$. 

Let $\cO_{\fL_s}$ denote the ring of integers of $\fL_s$, i.e., $\cO_{\fL_s}$ is the integral closure of $\fA_s = \bF_s[t]$ in $\fL_s$ for $\cD$-almost all $s \in S$. Let
\begin{align}
\label{def-ultra-ring-of-integers-O_L}
\cO_{\fL} = \prod_{s\in S}\cO_{\fL_s}/\cD
\end{align}
be the ultraproduct of the $\cO_{\fL_s}$ with respect to the nonprincipal ultrafilter $\cD$. 

When studying ramifications of primes of $\F$ in $\fL_{\alg}$, we mainly use shadows and the ultra-shadow of $\fL_{\alg}$, and their rings of integers, instead of the ring of integers $\cO_{\fL_s}$ and the ultraproduct $\cO_{\fL}$.

We first consider the case when $\fL_{\alg}$ is a finite Galois extension of $\F$, and then move to the case when $\fL_{\alg}$ is an infinite Galois extension of $\F$.

\subsection{Prime ramifications in $\fL_{\alg}$ when $\fL_{\alg}$ is a finite Galois extension of $\F$}
\label{subsec-prime-ramification-in-L-alg-finite-case}

Throughout this subsection, we further assume that $\fL_{\alg}$ is a finite Galois extension of degree $m$ over $\F$. This holds, for example, if $\fL_s$ is a finite Galois extension of degree $n_s \in \bZ_{>0}$ over $\fF_s$ for $\cD$-almost all $s \in S$ such that there are only finitely many integers dividing the hyperinteger $n = \ulim_{s\in S}n_s$ in $\bZ^{\#}$ (see Theorem \ref{cor-main-cor1-L_alg-is-a-finite-Galois-extension-of-F}). In this case, the degree $m$ of $\fL_{\alg}$ over $\F$ divides $n$ in $\bZ^{\#}$ (see Remark \ref{rem-divisibility-of-global-degree-and-ultra-degrees}).

Let $\fS_s$ denote the $s$-th shadow of $\fL_{\alg}$ for $\cD$-almost all $s \in S$, and let $\fS = \prod_{s\in S}\fG_s/\cD$ denote the ultra-shadow of $\fL_{\alg}$. By Theorem \ref{thm-main-thm3-Galois-group-structures-of-shadows}, $\fS_s/\fF_s$ is a finite Galois subextension of $\fL_s/\fF_s$ whose Galois group $\Gal(\fS_s/\fF_s)$ is isomorphic to $\Gal(\fL_{\alg}/\F)$, and $\fS$ is an ultra-subfield of $\fL$ such that
\begin{align}
\label{e1-subsection-prime-ramifications-L_alg-is-finite-extension}
\fS_{\alg} = \fS \cap \F^{\alg} = \fL_{\alg},
\end{align}
and $\fS$ is a finite Galois extension of $\cU(\F)$ whose Galois group $\Gal(\fS/\cU(\F))$ is also isomorphic to $\Gal(\fL_{\alg}/\F)$.

Let $\cO_{\fS_s}$ denote the ring of integers of $\fS_s$, i.e., $\cO_{\fS_s}$ is the integral closure of $\fA_s = \bF_s[t]$ in $\fS_s$. Let
\begin{align}
\label{e2-def-of-O_S-subsection-prime-ramifications-L_alg-is-finite-extension}
\cO_{\fS} = \prod_{s\in S}\cO_{\fS_s}/\cD \subset \fS
\end{align}
be the ultraproduct of the $\cO_{\fS_s}$ with respect to $\cD$. It is trivial that
\begin{align*}
\cO_{\fS} \subset \cO_{\fL},
\end{align*}
where $\cO_{\fL}$ is defined in (\ref{def-ultra-ring-of-integers-O_L}).

\begin{remark}

$\cO_{\fS}$ is, in general, not a Dedekind domain, but is a Pr\"ufer domain (see Schoutens \cite{schoutens}).

\end{remark}

We will need the following definition to describe certain prime ideals in $\cO_{\fL_{\alg}}$.

\begin{definition}
\label{def-ultra-prime-ideals-in-O-S}
(ultra-prime ideals in $\cO_{\fS}$)

An ultraproduct of prime ideals $\fp_s$ in $\cO_{\fS_s}$ is called an \textbf{ultra-prime ideal of $\cO_{\fS}$}.

\end{definition}

\begin{proposition}
\label{prop-algebraic-part-of-ultra-integers}

The algebraic parts of $\cO_{\fS}$ and $\cO_{\fL}$ over $\F$ are the same, and coincide with $\cO_{\fL_{\alg}}$, that is,

\begin{align*}
(\cO_{\fS})_{\alg} = (\cO_{\fL})_{\alg} = \cO_{\fL_{\alg}},
\end{align*}
where $(\cO_{\fS})_{\alg} = \cO_{\fS} \cap \F^{\alg}$ and $(\cO_{\fL})_{\alg} = \cO_{\fL} \cap \F^{\alg}$.

\end{proposition}

\begin{proof}

Since $(\cO_{\fS})_{\alg} \subset (\cO_{\fL})_{\alg}$, we only need to prove that
\begin{align*}
(\cO_{\fL})_{\alg} \subset \cO_{\fL_{\alg}} \subset (\cO_{\fS})_{\alg}.
\end{align*}

We first prove that $(\cO_{\fL})_{\alg} \subset \cO_{\fL_{\alg}}$.

Take an arbitrary element $\alpha \in (\cO_{\fL})_{\alg} = \cO_{\fL} \cap \F^{\alg}$. We can write $\alpha = \ulim_{s\in S}\alpha_s$ for some elements $\alpha_s \in \cO_{\fL_s}$. Since $\alpha \in \F^{\alg}$ there exists a primitive polynomial $P(x) = \sum_{i = 0}^ha_ix^i \in \A[x]$ of degree $h$ such that $P(x)$ is irreducible over $\F$ and $P(\alpha) = 0$, where the $a_i \in \A$, $a_h \ne 0$, and $\gcd(a_0, \ldots, a_h) = 1$. By Lemma \ref{lem-elementary-lemma1}, one can write
\begin{align*}
P(x) = \ulim_{s\in S}P_s(x),
\end{align*}
where 
\begin{align*}
P_s(x) = a_{0, s} + a_{1, s}x + \cdots + a_{h, s}x^h 
\end{align*}
is a primitive polynomial of degree $h$ in $\fA_s[x]$, and irreducible over $\fF_s$ for $\cD$-almost all $s \in S$,
and the $a_{i, s}$ belong to $\fA_s$ such that 
\begin{align*}
a_i = \ulim_{s \in S}a_{i, s}
\end{align*}
for all $0 \le i \le h$, $a_{h, s} \ne 0$, and $\gcd(a_{0, s}, \ldots, a_{h, s}) = 1$ for $\cD$-almost all $s \in S$. 

Since $P(\alpha) = 0$ and $\alpha = \ulim_{s\in S}\alpha_s$, we see that $P_s(\alpha_s) = 0$ for $\cD$-almost all $s \in S$. Thus the minimal polynomial $f_s(x)$ of $\alpha_s$ over $\fF_s$ divides $P_s(x)$ for $\cD$-almost all $s \in S$. Since $\alpha_s \in \cO_{\fL_s}$, $f_s(x)$ is of the form
\begin{align*}
f_s(x) = b_{0, s} + b_{1, s}x + \cdots + b_{h_s - 1, s}x^{h_s - 1} + x^{h_s},
\end{align*}
where the $b_{i, s} \in \fA_s$ and $h_s \in \bZ_{>0}$. Since both $f_s(x)$ and $P_s(x)$ are irreducible over $\fF_s$, $\deg(f_s) = h_s = \deg(P_s) = h$ for $\cD$-almost all $s \in S$, and thus
\begin{align*}
\beta_s f_s(x) = b_{0, s}\beta_s + b_{1, s}\beta_s x + \cdots +  b_{h - 1, s}\beta_sx^{h - 1} + \beta_s x^{h_s} = P_s(x).
\end{align*}
for some nonzero element $\beta_s \in \fF_s$. Comparing the leading coefficients of $P_s(x)$ and $\beta_sf_s(x)$, we deduce that $\beta_s = a_{h, s}$, and thus comparing the remaining coefficients of $P_s(x)$ and $\beta_sf_s(x)$ implies that
\begin{align*}
a_{i, s} = \beta_s b_{i, s} = a_{h, s}b_{i, s}
\end{align*}
for all $0 \le i \le h - 1$. Therefore $a_{h, s}$ divides $\gcd(a_{0, s}, \ldots, a_{h, s}) = 1$, and thus $a_{h, s}$ is a unit in $\fA_s = \bF_s[t]$ for $\cD$-almost all $s \in S$. Therefore $a_{h, s} \in \bF_s^{\times}$ for $\cD$-almost all $s \in S$. Since $\fK = \prod_{s\in S}\bF_s/\cD$, we deduce that $a_h = \ulim_{s\in S}a_{h, s}\in \fK^{\times}$, and hence $P(x)$ is an irreducible over $\F$ with coefficients in $\A$ such that the leading coefficient $a_h$ is a unit in $\fK^{\times}$. Since $P(\alpha) = 0$ and $\alpha \in (\cO_{\fL})_{\alg} \subset \fL_{\alg}$, we deduce that $\alpha \in \cO_{\fL_{\alg}}$. Thus 
\begin{align*}
(\cO_{\fL})_{\alg} \subset \cO_{\fL_{\alg}}.
\end{align*}\

We now prove that $\cO_{\fL_{\alg}} \subset (\cO_{\fS})_{\alg}$. Take an arbitrary element $\alpha \in \cO_{\fL_{\alg}}$. Since $\fL_{\alg} = \fS_{\alg}$, $\cO_{\fL_{\alg}} = \cO_{\fS_{\alg}}$. Since $\alpha \in \cO_{\fL_{\alg}} = \cO_{\fS_{\alg}} \subset \fS_{\alg}$, one can write $\alpha = \ulim_{s\in S}\alpha_s$, where $\alpha_s \in \fS_s$ for $\cD$-almost all $s \in S$. Furthermore the minimal polynomial $P(x)$ of $\alpha$ over $\F$ is of the form
\begin{align*}
P(x) = a_0 + a_1x + \cdots + a_{h - 1}x^{h - 1} + x^h \in \A[x],
\end{align*}
where the $a_i$ belong to $\A$ and $h \in \bZ_{>0}$. By Lemma \ref{lem-elementary-lemma1}, one can write
\begin{align*}
P(x) = \ulim_{s\in S}P_s(x),
\end{align*}
where $P_s(x)$ is of the form
\begin{align*}
P_s(x) = a_{0, s} + a_{1, s}x + \cdots + a_{h - 1, s}x^{h - 1} + a_{h, s}x^h \in \fA_s[x]
\end{align*}
such that $P_s(x)$ is irreducible over $\fF_s$ for $\cD$-almost all $s \in S$, and the $a_{i, s}$ belong to $\fA_s$ such that $a_i = \ulim_{s\in S}a_{i, s}$ for $0 \le i \le h - 1$ and $1 = \ulim_{s\in S}a_{h, s}$. Thus $a_{h, s} = 1$ for $\cD$-almost all $s \in S$, which implies that $P_s(x)$ is a monic polynomial in $\fA_s[x]$ that is irreducible over $\fF_s$ for $\cD$-almost all $s \in S$. 

Since $P(\alpha) = 0$, $P_s(\alpha_s) = 0$ for $\cD$-almost all $s \in S$, and thus $P_s(x)$ is the minimal polynomial of $\alpha_s$ for $\cD$-almost all $s \in S$. Thus $\alpha_s \in \cO_{\fS_s}$ for $\cD$-almost all $s \in S$, and therefore $\alpha = \ulim_{s\in S}\alpha_s \in \prod_{s\in S}\cO_{\fS_s}/\cD = \cO_{\fS}$. Thus $\alpha \in \cO_{\fS} \cap \F^{\alg} = (\cO_{\fS})_{\alg}$, and therefore
\begin{align*}
\cO_{\fL_{\alg}} \subset (\cO_{\fS})_{\alg}.
\end{align*}

By what we have showed above, we deduce that
\begin{align*}
(\cO_{\fL})_{\alg} \subset \cO_{\fL_{\alg}} \subset (\cO_{\fS})_{\alg} \subset (\cO_{\fL})_{\alg},
\end{align*}
which yields
\begin{align*}
(\cO_{\fL})_{\alg} =  \cO_{\fL_{\alg}} = (\cO_{\fS})_{\alg},
\end{align*}
as desired.

\end{proof}

We generalize Lemma \ref{lem-algebraic-part-of-principal-ideal-in-U(F)-is-principal-ideal-in-F} to principal ideals in $\cO_{\fL_{\alg}}$.

\begin{corollary}
\label{cor-algebraic-part-of-principal-ideals-in-O_S-is-principal-ideal-in-O_L-alg}

For any element $\alpha \in \cO_{\fL_{\alg}}$, the algebraic part of $\alpha\cO_{\fS}$ over $\F$ coincides with $\alpha\cO_{\fL_{\alg}}$, i.e.,
\begin{align*}
(\alpha\cO_{\fS})_{\alg} = \alpha\cO_{\fS} \cap \F^{\alg} = \alpha\cO_{\fL_{\alg}}.
\end{align*}

\end{corollary}

\begin{proof}

The corollary is trivial for $\alpha = 0$. So we assume that $\alpha \ne 0$. 

We know from Proposition \ref{prop-algebraic-part-of-ultra-integers} that $\alpha\cO_{\fL_{\alg}} = \alpha(\cO_{\fS})_{\alg} = \alpha(\cO_{\fS}\cap \F^{\alg})$, which is a subset of $(\alpha\cO_{\fS})_{\alg}$. For the opposite implication, let $\beta$ be an arbitrary element in $(\alpha\cO_{\fS})_{\alg} = \alpha\cO_{\fS} \cap \F^{\alg}$. By Proposition \ref{prop-algebraic-part-of-ultra-integers}, $\cO_{\fS} \cap \F^{\alg} = (\cO_{\fS})_{\alg}= \cO_{\fL_{\alg}}$, and thus
\begin{align*}
\beta \in \alpha\cO_{\fS} \cap \F^{\alg} \subset \cO_{\fS} \cap \F^{\alg} = \cO_{\fL_{\alg}}.
\end{align*}

Since $\beta \in \alpha\cO_{\fS}$, $\beta = \alpha\gamma$ for some element $\gamma \in \cO_{\fS}$. Thus since $\alpha, \beta \in \cO_{\fL_{\alg}}$, $\gamma = \beta/\alpha \in \fL_{\alg}$. It thus follows from Proposition \ref{prop-algebraic-part-of-ultra-integers} that
\begin{align*}
\gamma \in \cO_{\fS}\cap \fL_{\alg} \subset \cO_{\fS} \cap \F^{\alg} = (\cO_{\fS})_{\alg}= \cO_{\fL_{\alg}},
\end{align*}
and therefore
\begin{align*}
\beta = \alpha\gamma \in \alpha\cO_{\fL_{\alg}},
\end{align*}
which implies that $(\alpha\cO_{\fS})_{\alg}$ is a subset of $\alpha\cO_{\fL_{\alg}}$ as required.

\end{proof}

We describe certain prime ideals in $\cO_{\fL_{\alg}}$ that arise from ultra-prime ideals in $\cO_{\fS}$.

\begin{proposition}
\label{prop-algebraic-part-of-ultra-prime-is-prime-in-O_L-alg}

Let $\fQ = \prod_{s\in S}\fq_s/\cD$ be an ultra-prime ideal in $\cO_{\fS} = \prod_{s\in S}\cO_{\fS_s}/\cD$, where $\fq_s$ is a nonzero prime ideal in $\cO_{\fS_s}$ for $\cD$-almost all $s \in S$. Then the algebraic part of $\fQ$ over $\F$, i.e., $\fQ_{\alg} = \fQ \cap \F^{\alg}$ is a prime ideal in $\cO_{\fL_{\alg}}$. Furthermore $\fQ_{\alg}$ is a nonzero prime ideal if and only if there exists a positive constant $M > 0$ such that
\begin{align*}
\{s \in S\; | \; \deg(Q_s) < M\} \in \cD,
\end{align*}
where $Q_s$ is an irreducible polynomial in $\fA_s = \bF_s[t]$ such that $\fq_s \cap \fA_s = Q_s\fA_s$ \footnote{Such an irreducible polynomial $Q_s$ exists since $\fA_s$ is a principal ideal domain, and $\fq_s \cap \fA_s$ is a prime ideal in $\fA_s$.} for $\cD$-almost all $s \in S$.

\end{proposition}

\begin{remark}

By Lemma \ref{prop-algebraic-part-of-ultra-integers}, $\cO_{\fL_{\alg}} = (\cO_{\fS})_{\alg} = \cO_{\fS} \cap \F^{\alg}$. Since $\fQ \subset \cO_{\fS}$, we deduce that $\fQ_{\alg} \subset (\cO_{\fS})_{\alg} = \cO_{\fL_{\alg}}$.

\end{remark}

\begin{proof}

We first prove that $\fQ_{\alg}$ is an ideal in $\cO_{\fL_{\alg}}$. Take an arbitrary element $\lambda \in \cO_{\fL_{\alg}} = (\cO_{\fS})_{\alg} = \cO_{\fS} \cap \F^{\alg}$. Then one can write
\begin{align*}
\lambda = \ulim_{s\in S}\lambda_s,
\end{align*}
where $\lambda_s \in \cO_{\fS_s}$ for $\cD$-almost all $s \in S$. Since $\fq_s$ is a prime ideal in $\cO_{\fS_s}$, we deduce that $\lambda_s \fq_s \subset \fq_s$, and thus
\begin{align*}
\lambda\fQ = \ulim_{s\in S}\lambda_s \prod_{s\in S}\fq_s/\cD = \prod_{s\in S}\lambda_s\fq_s/\cD \subset \prod_{s\in S}\fq_s/\cD = \fQ,
\end{align*}
and 
\begin{align*}
\lambda\F^{\alg} \subset \F^{\alg}.
\end{align*}
Thus
\begin{align*}
\lambda\fQ_{\alg} = \lambda(\fQ \cap \F^{\alg})\subset \lambda\fQ \cap \lambda \F^{\alg}\subset \fQ \cap \F^{\alg} = \fQ_{\alg}.
\end{align*}
Since $\lambda$ is arbitrary in $\cO_{\fL_{\alg}}$, we deduce that $\fQ_{\alg}$ is an ideal in $\cO_{\fL_{\alg}}$.

We now prove that $\fQ_{\alg}$ is a prime ideal in $\cO_{\fL_{\alg}}$. 

Let $\alpha\beta \in \fQ_{\alg} = \fQ \cap \F^{\alg}$ for elements $\alpha, \beta \in \cO_{\fL_\alg} = (\cO_{\fS})_{\alg}$. One can write 
\begin{align*}
\alpha &= \ulim_{s\in S}\alpha_s,\\
\beta  &= \ulim_{s\in S}\beta_s
\end{align*}
for some elements $\alpha_s, \beta_s \in \cO_{\fS_s}$. Thus $\alpha\beta = \ulim_{s\in S}\alpha_s\beta_s \in \fQ_{\alg} \subset \fQ = \prod_{s\in S}\fq_s/\cD$, and therefore
\begin{align*}
\alpha_s\beta_s \in \fq_s
\end{align*}
for $\cD$-almost all $s \in S$. Since $\fq_s$ is a prime ideal in $\cO_{\fS_s}$ for $\cD$-almost all $s \in S$, 
\begin{align*}
A = \{s \in S\; |\; \alpha_s\beta_s \in \fq_s\} = \{s\in S\; |\; \text{$\alpha_s \in \fq_s$ or $\beta_s \in \fq_s$}\} \in \cD.
\end{align*}

Let
\begin{align*}
B &= \{s\in S\; | \; \alpha_s \in \fq_s\}, \\
C &= \{s\in S\; | \; \beta_s \in \fq_s\}.
\end{align*}

It is clear that $A = B \cup C$. Since $A \in \cD$, we deduce from Lemma \ref{lem-at-least-one-set-in-th-union-of-sets-is-in-D}, $B \in \cD$ or $C \in \cD$. Thus $\alpha = \ulim_{s\in S}\alpha_s \in \prod_{s\in S}\fq_s/\cD = \fQ$ or $\beta = \ulim_{s\in S}\beta_s \in \prod_{s\in S}\fq_s/\cD = \fQ$. Thus $\alpha \in \fQ \cap \cO_{\fL_{\alg}} \subset \fQ \cap \F^{\alg} = \fQ_{\alg}$ or $\beta = \fQ \cap \cO_{\fL_{\alg}} \subset \fQ \cap \F^{\alg} = \fQ_{\alg}$. Thus $\fQ_{\alg}$ is a prime ideal in $\cO_{\fL_{\alg}}$.

We prove the last assertion of the lemma. 

Suppose that there exists a positive constant $M > 0$ such that
\begin{align*}
\{s \in S\; | \; \deg(Q_s) < M\} \in \cD.
\end{align*}
Thus Lemma \ref{lem-elementary-lemma} implies that $Q = \ulim_{s\in S}Q_s \in \A$. Since $\fq_s$ is a nonzero prime ideal for $\cD$-almost all $s \in S$, $Q_s$ is an irreducible polynomial in $\fA_s$ for $\cD$-almost all $s \in S$, and it thus follows from Lemma \ref{lem-elementary-lemma0} that $Q = \ulim_{s\in S}Q_s$ is an irreducible polynomial in $\A$. Since $Q_s \in Q_s\fA_s = \fq_s \cap \fA_s \subset \fq_s$, we deduce that 
\begin{align*}
Q = \ulim_{s\in S}Q_s \in \prod_{s\in S}\fq_s/\cD = \fQ.
\end{align*}
Since $Q \in \A$, we deduce that $Q \in \fQ\cap \A \subset \fQ \cap \F^{\alg} = \fQ_{\alg}$, and thus $\fQ_{\alg}$ is a nonzero prime ideal in $\cO_{\fL_{\alg}}$. 

Conversely, suppose that $\fQ_{\alg}$ is a nonzero prime ideal in $\cO_{\fL_{\alg}}$. Thus $\fQ_{\alg} \cap \A = Q\A$ for some irreducible polynomial $Q$ of degree $d > 0$ in $\A = \fK[t]$. By Lemma \ref{lem-elementary-lemma0}, $Q = \ulim_{s\in S}Q_s$ for some irreducible polynomials $Q_s$ of degree $d$ in $\fA_s = \bF_s[t]$ for $\cD$-almost all $s \in S$.

Since $Q = \ulim_{s\in S}Q_s \in \fQ_{\alg} \subset \fQ = \prod_{s\in S}\fq_s/\cD$, we deduce that $Q_s \in \fq_s$ for $\cD$-almost all $s \in S$. Thus $\fq_s \cap \fA_s = Q_s\fA_s$. Since $\{s \in S \; | \; \deg(Q_s) = d \} \in \cD$, the lemma follows immediately.

\end{proof}

\begin{proposition}
\label{prop-algebraic-parts-of-ultra-prime-ideals-invariant-under-Galois-action}

Let $\fp = \prod_{s\in S}\fp_s/\cD$ be an ultra-prime ideal of $\fS = \prod_{s\in S}\fS_s/\cD$, where $\fp_s$ is a prime ideal in $\cO_{\fS_s}$ for $\cD$-almost all $s \in S$. For all elements $\lambda \in \Gal(\fL_{\alg}/\F)$, 
\begin{align*}
\lambda(\fp_{\alg}) = \left(\Sigma^{-1}(\lambda)(\fp)\right)_{\alg} = \Sigma^{-1}(\lambda)(\fp) \cap \F^{\alg} = \left(\prod_{s\in S}\lambda_s(\fp_s)\right) \cap \F^{\alg},
\end{align*}
where $\lambda_s \in \Gal(\fS_s/\fF_s)$ is the $s$-th component of $\lambda$ as in the proof of Theorem \ref{thm-main-thm3-Galois-group-structures-of-shadows} and Remark \ref{rem-s-th-components-of-a-Galois-element}, $\Sigma^{-1} : \Gal(\fL_{\alg}/\F) = \Gal(\fS_{\alg}/\F) \to \Gal(\fS/\cU(\F))$ is the map given in the proof of Theorem \ref{thm-main-thm3-Galois-group-structures-of-shadows} and in Remark \ref{rem-relation-between-s-th-component-of-lambda-and-its-image-in-Gal(fS/cU(F))}, and $\fp_{\alg}$ is the algebraic part of $\fp$ over $\F$, i.e., $\fp_{\alg} = \fp \cap \F^{\alg}$.

\end{proposition}

\begin{proof}

By the primitive element theorem, $\fL_{\alg} = \F(\alpha_1)$ for some element $\alpha_1 \in \fL_{\alg}$. By Proposition \ref{prop-explicit-descriptions-of-shadows-and-ultra-shadows} and Theorem \ref{thm-main-thm3-Galois-group-structures-of-shadows}, $\fS = \cU(\F)(\alpha_1)$, where we recall that $\fS$ is the ultra-shadow of $\fL_{\alg}$. Let $\alpha_1, \ldots, \alpha_m$ be all the Galois conjugates of $\alpha_1$, i.e., 
\begin{align*}
\{\alpha_1, \ldots, \alpha_m\} = \{\sigma(\alpha_1) \; | \; \sigma \in \Gal(\fL_{\alg}/\F)\}.
\end{align*}
For each $\lambda \in \Gal(\fL_{\alg}/\F)$, let $\iota_{\lambda}$ be the unique permutation of $\{1, \ldots, m\}$ that determines $\lambda$, i.e.,
\begin{align*}
\lambda(\alpha_i) = \alpha_{\iota_{\lambda}(i)}
\end{align*}
for all $1 \le i \le m$. Following the proof of Theorem \ref{thm-main-thm3-Galois-group-structures-of-shadows}, the permutation $\iota_{\lambda}$ also uniquely determines $\Sigma^{-1}(\lambda) \in \Gal(\fS/\cU(\F))$ by letting
\begin{align}
\label{e0-prop-algebraic-part-of-ultra-prime-is-Galois-invariant}
\Sigma^{-1}(\lambda)(\alpha_i) = \alpha_{\iota_{\lambda}(i)}
\end{align}
for all $1 \le i \le m$, and 
\begin{align*}
\Sigma^{-1}(\lambda)(a) = a
\end{align*}
for all $a \in \cU(\F)$. 

By Proposition \ref{prop-algebraic-part-of-ultra-prime-is-prime-in-O_L-alg}, $\fp_{\alg}$ is a prime ideal in $\cO_{\fL_{\alg}}$, and thus since $\lambda \in \Gal(\fL_{\alg}/\F)$, $\lambda(\fp_{\alg})$ is also a prime ideal in $\cO_{\fL_{\alg}}$. 

On the other hand, we know from Remark \ref{rem-relation-between-s-th-component-of-lambda-and-its-image-in-Gal(fS/cU(F))}, $\Sigma^{-1}(\lambda) = \ulim_{s\in S}\lambda_s$, and thus since $\fp = \prod_{s\in S}\fp_s/\cD$, we deduce that
\begin{align*}
\Sigma^{-1}(\lambda)(\fp) = \prod_{s\in S}\lambda_s(\fp_s)/\cD.
\end{align*}
Since $\fp_s$ is a prime ideal in $\cO_{\fS_s}$ for $\cD$-almost all $s \in S$ and $\lambda_s \in \Gal(\fS_s/\fF_s)$, $\lambda_s(\fp_s)$ is also a prime ideal in $\cO_{\fS_s}$ for $\cD$-almost all $s \in S$. Thus $\prod_{s\in S}\lambda_s(\fp_s)/\cD$ is an ultra-prime ideal in $\cO_{\fS}$. Therefore it follows from Proposition \ref{prop-algebraic-part-of-ultra-prime-is-prime-in-O_L-alg} that its algebraic part $\left(\prod_{s\in S}\lambda_s(\fp_s)/\cD\right) \cap \F^{\alg}$ is a prime ideal in $\cO_{\fL_{\alg}}$, and thus $\Sigma^{-1}(\lambda)(\fp) \cap \F^{\alg}$ is a prime ideal in $\cO_{\fL_{\alg}}$.

Since $\cO_{\fL_{\alg}}$ is a Dedekind domain, it suffices to prove that 
\begin{align*}
\lambda(\fp_{\alg}) \subset \Sigma^{-1}(\lambda)(\fp) \cap \F^{\alg}.
\end{align*}

Let $\beta$ be an arbitrary element in $\lambda(\fp_{\alg})$. Then there exists an element $\epsilon \in \fp_{\alg} = \fp \cap \F^{\alg}$ such that 
\begin{align}
\label{e1-prop-algebraic-part-of-ultra-prime-is-Galois-invariant}
\beta = \lambda(\epsilon).
\end{align}

Since $\epsilon \in \fp_{\alg}$ and $\fp_{\alg}$ is a prime ideal in $\cO_{\fL_{\alg}}$, $\epsilon \in \cO_{\fL_{\alg}}$. Since $\fL_{\alg} = \F(\alpha_1)$ and $[\fL_{\alg}:\F] = m$, one can write
\begin{align}
\label{e2-prop-algebraic-part-of-ultra-prime-is-Galois-invariant}
\epsilon = \sum_{i = 0}^{m - 1}a_i \alpha_1^i
\end{align}
for some elements $a_i \in \F$. Since $\lambda(\alpha_1) = \alpha_{\iota_{\lambda}(1)}$, 
\begin{align}
\label{e3-prop-algebraic-part-of-ultra-prime-is-Galois-invariant}
\beta = \lambda(\epsilon) = \sum_{i = 0}^{m - 1}a_i\alpha_{\iota_{\lambda}(1)}^i.
\end{align} 

By (\ref{e0-prop-algebraic-part-of-ultra-prime-is-Galois-invariant}) and since $\Sigma^{-1}(\lambda)$ fixes every element in $\cU(\F)$, we deduce that $\Sigma^{-1}(\lambda)$ also fixes every element in $\F$, and it thus follows from (\ref{e2-prop-algebraic-part-of-ultra-prime-is-Galois-invariant}) and (\ref{e3-prop-algebraic-part-of-ultra-prime-is-Galois-invariant}) that
\begin{align*}
\Sigma^{-1}(\lambda)(\epsilon) = \sum_{i = 0}^{m - 1}a_i\Sigma^{-1}(\lambda)(\alpha_{1})^i = \sum_{i = 0}^{m - 1}a_i\alpha_{\iota_{\lambda}(1)}^i = \beta.
\end{align*}
Therefore since $\epsilon \in \fp_{\alg} \subset \fp$, $\beta$ belongs to $\Sigma^{-1}(\lambda)(\fp)$. Since $\lambda(\fp_{\alg})$ is a prime ideal in $\cO_{\fL_{\alg}}$, $\beta$ also belongs to $\cO_{\fL_{\alg}}$, and thus
\begin{align*}
\beta \in \Sigma^{-1}(\lambda)(\fp) \cap \cO_{\fL_{\alg}} \subset \Sigma^{-1}(\lambda)(\fp) \cap \F^{\alg}.
\end{align*}
Therefore $\lambda(\fp_{\alg})$ is a subset of $\Sigma^{-1}(\lambda)(\fp) \cap \F^{\alg}$, which verifies the proposition.

\end{proof}

We prove the converse to Proposition \ref{prop-algebraic-part-of-ultra-prime-is-prime-in-O_L-alg}.

\begin{proposition}
\label{prop-main-prop1-existence-of-ultra-prime-whose-algebraic-part-coincides-with-a-given-prime}

Let $\fp$ be a prime in $\cO_{\fL_{\alg}}$. Then there exists an ultra-prime ideal $\fP = \prod_{s\in S}\fP_s/\cD$ in $\cO_{\fS}$ for some prime ideals $\fP_s$ in $\fS_s$ such that the algebraic part of $\fP$ over $\F$ coincides with $\fp$, that is,
\begin{align*}
\fP_{\alg} = \fP \cap \F^{\alg} = \fp.
\end{align*}

\end{proposition}

\begin{proof}

Since $\A = \fK[t]$ is a Dedekind domain and $\fL_{\alg}$ is a finite Galois extension of degree $m$ over $\F$, the integral closure $\cO_{\fL_{\alg}}$ of $\A$ in $\fL_{\alg}$ is also a Dedekind domain. Thus $\P = \fp \cap \A$ is a prime ideal in $\A$. Since $\A = \fK[t]$ is a principal ideal domain, there exists an irreducible polynomial $P(t) \in \A$ such that $P$ generates the prime ideal $\P$, i.e., 
\begin{align}
\label{e1-prop-main-prop1-existence-of-ultra-prime-whose-algebraic-part-coincides-with-a-given-prime}
\P = \fp \cap \A = P\A.
\end{align} 

By Lemma \ref{lem-elementary-lemma0}, we can write $P = \ulim_{s\in S}P_s$, where $P_s(t)$ is an irreducible polynomial in $\fA_s = \bF_s[t]$ for $\cD$-almost all $s \in S$. Since $\fS_s$ is a finite Galois extension of degree $m$ over $\fF_s$ (see Corollary \ref{cor-Galois-group-structure-of-shadows-of-algebraic-part-in-a-given-ultra-extension}), $\cO_{\fS_s}$ is a Dedekind domain, and thus there exist (not necessarily distinct) prime ideals $\fp_{1, s}, \ldots, \fp_{h_s, s}$ in $\cO_{\fS_s}$ with $h_s \le m$ such that
\begin{align*}
P_s\cO_{\fS_s} = \fp_{1, s} \fp_{2, s}\cdots \fp_{h_s, s}.
\end{align*}

Let $\cA_s$ be the set consisting of prime ideals $\fp_{1, s}, \ldots, \fp_{h_s, s}$ as its elements. We form the ultraproduct of the $\cA_s$ with respect to $\cD$ of the form
\begin{align*}
\cA = \prod_{s\in S}\cA_s/\cD = \{\prod_{s\in S}\fp_{i_s, s}/\cD \; |\; \text{$\fp_{i_s, s} \in \cA_s$ for all $1 \le i_s \le h_s$} \} \subset \cO_{\fS}.
\end{align*}

The ultraproduct $\cA$ consists of ultra-prime ideals in $\cO_{\fS}$ that are ultraproducts of prime ideals in the sets $\cA_s$. Since $\card(\cA_s) \le m$, we know from Bell--Slomson \cite[Lemma 3.7]{bell-slomson} that $\card(\cA) \le \ulim_{s\in S}m = m$, and thus $\cA$ is a finite set of cardinality $\le m$. 

Take an arbitrary ultra-prime ideal $\fQ = \prod_{s\in S}\fp_{i_s, s}/\cD$ in $\cA$, and let 
\begin{align*}
\fQ_{\alg} = \fQ \cap \F^{\alg}
\end{align*}
be the algebraic part of $\fQ$ over $\F$. By Proposition \ref{prop-algebraic-part-of-ultra-prime-is-prime-in-O_L-alg}, $\fQ_{\alg}$ is a prime ideal in $\cO_{\fL_{\alg}}$. We contend that
\begin{align*}
\fQ_{\alg} \cap \A = P\A.
\end{align*}
Indeed, we see that
\begin{align*}
\fQ_{\alg} \cap \A &= \fQ \cap \F^{\alg} \cap \A \\
&= \fQ \cap \F^{\alg} \cap \cU(\A) \cap \F^{\alg} \; \; (\text{see Lemma \ref{lem-U(F)-algebraic-part-is-F}}) \\
&= (\fQ \cap \cU(\A)) \cap \F^{\alg} \\
&= \left(\prod_{s\in S}(\fp_{i_s, s} \cap \fA_s)/\cD\right) \cap \F^{\alg} \\
&= \left(\prod_{s\in S}P_s\fA_s/\cD\right) \cap \F^{\alg} \; \; (\text{since $\fp_{i_s, s}$ lies above $P_s\fA_s$}) \\
&= P\cU(\A) \cap \F^{\alg} \; \; (\text{since $P = \ulim_{s\in S}P_s$ and $\cU(\A) = \prod_{s\in S}\fA_s/\cD$}) \\
&= (P\cU(\A))_{\alg} \\
&= P\A \; \; (\text{see Lemma \ref{lem-algebraic-part-of-principal-ideal-in-U(F)-is-principal-ideal-in-F}})
\end{align*}
as required.

Thus both $\fQ_{\alg}$ and $\fp$ lie above $P\A$ (see (\ref{e1-prop-main-prop1-existence-of-ultra-prime-whose-algebraic-part-coincides-with-a-given-prime})), and thus there exists an element $\lambda \in \Gal(\fL_{\alg}/\F)$ such that 
\begin{align}
\label{e1-prop-main-prop1-existence-of-ultra-prime-whose-algebraic-part-coincides-with-a-given-prime}
\fp = \lambda(\fQ_{\alg}).
\end{align}

By Proposition \ref{prop-algebraic-parts-of-ultra-prime-ideals-invariant-under-Galois-action}, we know that
\begin{align*}
\lambda(\fQ_{\alg}) = \Sigma^{-1}(\lambda)(\fQ)_{\alg},
\end{align*}
where $\Sigma^{-1}$ is the isomorphism from $\Gal(\fL_{\alg}/\F)$ to $\Gal(\fS/\cU(\F))$ in the proof of Theorem \ref{thm-main-thm3-Galois-group-structures-of-shadows} and Remark \ref{rem-relation-between-s-th-component-of-lambda-and-its-image-in-Gal(fS/cU(F))}. Since $\fQ = \prod_{s\in S}\fp_{i_s, s}/\cD$, it follows from (\ref{e1-prop-main-prop1-existence-of-ultra-prime-whose-algebraic-part-coincides-with-a-given-prime}) and Proposition \ref{prop-algebraic-parts-of-ultra-prime-ideals-invariant-under-Galois-action} that
\begin{align}
\label{e2-prop-main-prop1-existence-of-ultra-prime-whose-algebraic-part-coincides-with-a-given-prime}
\fp = \Sigma^{-1}(\lambda)(\fQ)_{\alg} = \left(\prod_{s\in S}\lambda_s(\fp_{i_s, s})/\cD\right)\cap \F^{\alg},
\end{align}
where $\lambda_s$ is the $s$-th component of $\lambda$ in the proof of Theorem \ref{thm-main-thm3-Galois-group-structures-of-shadows} and Remark \ref{rem-s-th-components-of-a-Galois-element}.

Letting $\fP = \Sigma^{-1}(\lambda)(\fQ)$ and following the proof of Proposition \ref{prop-algebraic-parts-of-ultra-prime-ideals-invariant-under-Galois-action}, we know that  $\fP$ is an ultra-prime ideal in $\cO_{\fS}$ that is of the form
\begin{align*}
\fP = \prod_{s\in S}\fP_s/\cD,
\end{align*}
where $\fP_s = \lambda_s(\fp_{i_s, s})$ is a prime ideal in $\cO_{\fS_s}$. Equation (\ref{e2-prop-main-prop1-existence-of-ultra-prime-whose-algebraic-part-coincides-with-a-given-prime}) immediately implies that $\fp = \fP_{\alg} = \fP \cap \F^{\alg}$ as required.

\end{proof}

\begin{theorem}
\label{thm-unramified-primes-locally-imply-unramified-globally-in-algebraic-parts-finite-case}

Let $P$ be an irreducible polynomial, i.e., a prime in $\A = \fK[t]$. By Lemma \ref{lem-elementary-lemma0}, one can write $P = \ulim_{s\in S}P_s$, where $P_s$ is an irreducible polynomial in $\fA_s = \bF_s[t]$ for $\cD$-almost all $s \in S$. If $P_s$ is unramified in the $s$-th shadow $\fS_s$ of $\fL_{\alg}$ for $\cD$-almost all $s \in S$, then $P$ is unramified in $\fL_{\alg}$.

\end{theorem}

\begin{proof}

Assume the contrary, i.e., $P$ is ramified in $\fL_{\alg}$, and thus there exists a prime ideal $\fp$ in $\cO_{\fL_{\alg}}$ such that $P\cO_{\fL_{\alg}} \subset \fp^e$ for some integer $e \ge 2$. Since $\cO_{\fL_{\alg}}$ is a Dedekind domain, there exist elements $\alpha, \beta \in \cO_{\fL_{\alg}}$ such that $\fp = \alpha\cO_{\fL_{\alg}} + \beta\cO_{\fL_{\alg}}$, i.e., $\fp$ is generated by $\alpha$ and $\beta$ (see Swinnerton-Dyer \cite{swinnerton-dyer}). Thus
\begin{align*}
P \in \fp^e \subset \fp^2,
\end{align*}
and therefore 
\begin{align}
\label{e1-thm-unramified-primes-in-algebraic-parts-finite-case}
P = \sum_{i \in I}(\alpha a_i + \beta b_i)(\alpha c_i + \beta d_i)
\end{align}
for some elements $a_i, b_i, c_i, d_i \in \cO_{\fL_{\alg}}$ and some finite index set $I$. 

By Lemma \ref{prop-algebraic-part-of-ultra-integers}, $\cO_{\fL_{\alg}} = (\cO_{\fS})_{\alg} = \cO_{\fS} \cap \F^{\alg} \subset \cO_{\fS} = \prod_{s\in S}\cO_{\fS_s}/\cD$, and thus we can write $\alpha = \ulim_{s\in S}\alpha_s$, $\beta = \ulim_{s\in S}\beta_s$, $a_i = \ulim_{s\in S}a_{i, s}$, $b_i = \ulim_{s\in S}b_{i, s}$, $c_i = \ulim_{s\in S}c_{i, s}$, $d_i = \ulim_{s\in S}d_{i, s}$ for some elements $\alpha_s, \beta_s, a_{i, s}, b_{i, s}, c_{i, s}, d_{i, s} \in \cO_{\fS_s}$ for $\cD$-almost all $s \in S$. It thus follows from (\ref{e1-thm-unramified-primes-in-algebraic-parts-finite-case}) that
\begin{align*}
P = \ulim_{s\in S}P_s &= \sum_{i \in I}(\alpha a_i + \beta b_i)(\alpha c_i + \beta d_i) \\
&= \ulim_{s\in S}\left(\sum_{i \in I}(\alpha_s a_{i, s} + \beta_s b_{i, s})(\alpha_s c_{i, s} + \beta_s d_{i, s})\right), 
\end{align*}
and thus
\begin{align*}
P_s = \sum_{i \in I}(\alpha_s a_{i, s} + \beta_s b_{i, s})(\alpha_s c_{i, s} + \beta_s d_{i, s})
\end{align*}
for $\cD$-almost all $s \in S$. Therefore 
\begin{align}
\label{e2-thm-unramified-primes-in-algebraic-parts-finite-case}
P_s\cO_{\fS_s} \subset \fB_s^2
\end{align}
for $\cD$-almost all $s \in S$,
where
\begin{align*}
\fB_s = \alpha_s\cO_{\fS_s} + \beta_s\cO_{\fS_s}
\end{align*}
is an ideal in $\cO_{\fS_s}$.

By Proposition \ref{prop-main-prop1-existence-of-ultra-prime-whose-algebraic-part-coincides-with-a-given-prime}, there exists an ultra-prime ideal $\fQ = \prod_{s\in S}\fQ_s/\cD$ in $\cO_{\fS}$, where $\fQ_s$ is a prime ideal in $\cO_{\fS_s}$ for $\cD$-almost all $s \in S$ such that
\begin{align*}
\fQ_{\alg} = \fQ\cap \F^{\alg} = \fp.
\end{align*}

We know that $\alpha = \ulim_{s\in S}\alpha_s \in \fp \subset \fQ = \prod_{s\in S}\fQ_s/\cD$, and thus $\alpha_s \in \fQ_s$ for $\cD$-almost all $s \in S$. Using the same argument implies that $\beta_s \in \fQ_s$ for $\cD$-almost all $s\in S$. Thus
\begin{align*}
\fB_s = \alpha_s\cO_{\fS_s} + \beta_s\cO_{\fS_s} \subset \fQ_s
\end{align*}
for $\cD$-almost all $s \in S$, and it thus follows from (\ref{e2-thm-unramified-primes-in-algebraic-parts-finite-case}) that
\begin{align*}
P_s\cO_{\fS_s} \subset \fQ_s^2,
\end{align*}
for $\cD$-almost all $s \in S$, which proves that $P_s$ is ramified in $\fS_s$ for $\cD$-almost all $s \in S$, a contradiction. Thus $P$ is unramified in $\fL_{\alg}$.

\end{proof}

\begin{corollary}
\label{cor-unramified-primes-from-local-in-L=prod-L-s-to-L-alg-finite-case}

Let $P$ be an irreducible polynomial, i.e., a prime in $\A = \fK[t]$. By Lemma \ref{lem-elementary-lemma0}, write $P = \ulim_{s\in S}P_s$, where $P_s$ is an irreducible polynomial in $\fA_s = \bF_s[t]$ for $\cD$-almost all $s \in S$. If $P_s$ is unramified in $\fL_s$ \footnote{In the case where $\fL_s$ is an infinite Galois extension of $\fF_s$, the condition that $P_s$ is unramified in $\fL_s$, is equivalent to that $P_s$ is unramified in every finite Galois subextension $\fH_s/\fF_s$ of $\fL_s/\fF_s$.} for $\cD$-almost all $s \in S$, then $P$ is unramified in $\fL_{\alg}$.

\end{corollary}

\begin{proof}

Since $P_s$ is unramified in $\fL_s$ and $\fS_s/\fF_s$ is a finite Galois subextension of $\fL_s/\fF_s$ for $\cD$-almost all $s \in S$, we deduce that $P_s$ is also unramified in $\fS_s$ for $\cD$-almost all $s \in S$. Hence the corollary follows immediately from Theorem \ref{thm-unramified-primes-locally-imply-unramified-globally-in-algebraic-parts-finite-case}.

\end{proof}

The converse of Theorem \ref{thm-unramified-primes-locally-imply-unramified-globally-in-algebraic-parts-finite-case} also holds if the constant fields $\bF_s$ are \textbf{perfect procyclic field} for $\cD$-almost all $s \in S$.

\begin{definition}
\label{def-procyclic-field}
(perfect procyclic field)

A \textbf{perfect procyclic field} is a pair $(F, \sigma)$, where $F$ is a perfect field and $\sigma$ is a topological generator of the absolute Galois group $G_F := \Gal(F^{\alg}/F)$, that is, $G_F = \overline{\langle F\rangle}$.

\end{definition}

We recall that the following conditions on a perfect field $F$ are equivalent.
\begin{itemize}

\item [(PCF1)] $F$ has at most one extension of each degree;

\item [(PCF2)] every finite extension $K/F$ is cyclic; and

\item [(PCF3)] The absolute Galois group $G_F = \Gal(F^{\alg}/F)$ is procyclic, i.e., there exists a nonempty set $\cP$ of prime numbers such that as topological groups, 
    \begin{align*}
    G_F = \Gal(F^{\alg}/F) \cong \prod_{\ell \in \cP}\bZ_{\ell},
    \end{align*}
    where $\bZ_{\ell}$ denotes the ring of $\ell$-adic integers. 
    
\end{itemize}

\begin{remark}

If $\bF_s$ is a perfect procyclic field for $\cD$-almost all $s \in S$, then $\fK = \prod_{s\in S}\bF_s/\cD$ is also a perfect procylic field (see \cite{Chatzidakis} or \cite{FJ}).

\end{remark}

For the rest of this paper, for a perfect procyclic field $F$ and for each positive integer $d$, we always denote by $F(d)$ the unique algebraic extension of degree $d$ over $F$ if such an extension exists.

The following lists a certain class of procyclic fields that we use in most applications in this paper.

\begin{example}

\begin{itemize}

\item []

\item [(1)] Finite fields $\bF_q$ with $q$ being a power of a prime $p$.

\item [(2)] Algebraically closed fields of characteristic $p > 0$.

\item [(3)] Nonprincipal ultraproducts of quasi-finite fields (see Serre \cite{serre-local-fields} for a definition of quasi-fields).

\end{itemize}

\end{example}

\begin{lemma}
\label{lem-main-lem0-ramifications-of-primes-in-O_L-alg}

Let $P$ be an irreducible polynomial in $\A$. Write $P = \ulim_{s\in S}P_s$ for some irreducible polynomials $P_s \in \fA_s = \bF_s[t]$ (see Lemma \ref{lem-elementary-lemma0}). Let
\begin{align}
\label{e1-lem-main-lem0-ramifications-of-primes-in-O-L-alg}
P\cO_{\fL_{\alg}} = \left(\fp_1\cdots \fp_{g(P)}\right)^{e(P)},
\end{align}
where the $\fp_i$ are distinct prime ideals in $\cO_{\fL_{\alg}}$, $e(P)$ is the ramification index of the $\fp_i$ over $P$, and $g(P)$ denotes the number of distinct prime ideals in the factorization of $P\cO_{\fL_{\alg}}$. Let $f(P)$ denote the inertia degree of $\fp_i$ over $P$ for $1\le i \le g(P)$, i.e., $f(P) = [\cO_{\fL_{\alg}}/\fp_i : \A/P\A]$ for all $1 \le i \le g(P)$. \footnote{Note that the ramification indexes of $\fp_i$ over $P$ all equal $e(P)$, and the inertia degrees $[\cO_{\fL_{\alg}}/\fp_i : \A/P\A]$ are all equal to $f(P)$ since $\fL_{\alg}$ is a finite Galois extension of $\F$ (see Lang \cite{lang-ant} or Serre \cite{serre-local-fields})}. 

Let 
\begin{align*}
P_s\cO_{\fS_s} = \left(\fp_{1, s}\cdots \fp_{g(P_s), s}\right)^{e(P_s)},
\end{align*}
where the $\fp_{i, s}$ are distinct prime ideals in $\cO_{\fS_s}$, $e(P_s)$ is the ramification index of the $\fp_{i, s}$ over $P_s$, and $g(P_s)$ denotes the number of distinct prime ideals in the factorization of $P\cO_{\fS_s}$. Let $f(P_s)$ denote the inertia degree of $\fp_{i, s}$ over $P_s$ for $1\le i \le g(P_s)$, i.e., $f(P_s) = [\cO_{\fS_s}/\fp_{i, s} : \fA_s/P_s\fA_s]$ for all $1 \le i \le g(P_s)$. Then there exist positive integers $1 \le e, f, g \le m = [\fL_{\alg} : \F]$ that satisfy the following.

\begin{itemize}

\item [(i)] the set $\{s \in S \; | \; e(P_s) = e, f(P_s) = f, g(P_s) = g\}$ belongs to the ultrafilter $\cD$ and
\begin{align*}
e(P)f(P)g(P) = efg = m.
\end{align*}

\item [(ii)] $g \ge g(P)$.

\end{itemize}

\end{lemma}

\begin{proof}

We first prove part (i). Since $\fS_s$ is a finite Galois extension of degree $m$ over $\fF_s$ for $\cD$-almost all $s \in S$ (see Corollary \ref{cor-Galois-group-structure-of-shadows-of-algebraic-part-in-a-given-ultra-extension}), we deduce that $1 \le e(P_s) \le m$ for $\cD$-almost all $s \in S$. Thus
\begin{align*}
S = \bigcup_{i = 1}^mE_i,
\end{align*}
where
\begin{align*}
E_i = \{s\in S\; | \; e(P_s) = i \}
\end{align*}
for each $1 \le i \le m$. By Lemma \ref{lem-at-least-one-set-in-th-union-of-sets-is-in-D}, $E_e \in \cD$ for some integer $1 \le e \le m$.

Similarly, using the same argument, we can prove that there exist integers $1 \le f, g \le m$ such that
\begin{align*}
\{s \in S\; | \; f(P_s) = f\} \in \cD
\end{align*}
and
\begin{align*}
\{s \in S\; | \; g(P_s) = g\} \in \cD.
\end{align*}

Since $\cD$ is an ultrafilter, we deduce that
\begin{align*}
&\{s \in S\; | \; e(P_s) = e, f(P_s) = f, g(P_s) = g\} \\
&= \{s\in S\; | \; e(P_s) = e\} \cap \{s \in S\; | \; f(P_s) = f\} \cap \{s \in S\; | \; g(P_s) = g\} \in \cD.
\end{align*}

We now prove part (ii). By part (i), one can write
\begin{align}
\label{e2-lem-main-lem0-ramifications-of-primes-in-O-L-alg}
P_s\cO_{\fS_s} = \left(\fp_{1, s}\cdots \fp_{g, s}\right)^e
\end{align}
for $\cD$-almost all $s \in S$.

Take arbitrary distinct prime ideals $\fp_i, \fp_j$ with $ 1 \le i < j \le g(P)$ in the factorization (\ref{e1-lem-main-lem0-ramifications-of-primes-in-O-L-alg}) of $P\cO_{\fL_{\alg}}$. By Proposition \ref{prop-main-prop1-existence-of-ultra-prime-whose-algebraic-part-coincides-with-a-given-prime}, there exist ultra-prime ideals $\fB_i = \prod_{s \in S}\fq_{i, s}/\cD$ and $\fB_j = \prod_{s\in S}\fq_{j, s}/\cD$ in $\cO_{\fS}$ such that
\begin{align*}
\fB_i \cap \F^{\alg} = \fp_i
\end{align*}
and 
\begin{align*}
\fB_j \cap \F^{\alg} = \fp_j.
\end{align*}

Since $\fp_i \ne \fp_j$, we see that
\begin{align*}
\fB_i = \prod_{s \in S}\fq_{i, s}/\cD \ne \fB_j = \prod_{s\in S}\fq_{j, s}/\cD,
\end{align*}
and thus 
\begin{align}
\label{e3-lem-main-lem0-ramifications-of-primes-in-O-L-alg}
\{s \in S \; | \; \fq_{i, s} \ne \fq_{j, s}\} \in \cD.
\end{align}

Since $P  \in \fp_i = \fB_i \cap \F^{\alg}$, $P = \ulim_{s\in S}P_s \in \fB_i = \prod_{s\in S}\fq_{i, s}/\cD$, and thus $P_s \in \fq_{i, s}$ for $\cD$-almost all $s \in S$. Therefore $P_s\cO_{\fS_s} \subset \fq_{i, s}$, i.e., $\fq_{i, s}$ lies above $P_s$. Thus it follows from (\ref{e2-lem-main-lem0-ramifications-of-primes-in-O-L-alg}) that there exists an integer $1\le \iota(i) \le g$ such that
\begin{align*}
\fq_{i, s} = \fp_{\iota(i), s}.
\end{align*}

Using the same arguments, we deduce that there exists an integer $1 \le \iota(j) \le g$ such that 
\begin{align*}
\fq_{j, s} = \fp_{\iota(j), s}.
\end{align*}
By (\ref{e3-lem-main-lem0-ramifications-of-primes-in-O-L-alg}), we deduce that
\begin{align}
\label{e4-lem-main-lem0-ramifications-of-primes-in-O-L-alg}
S_{i, j} = \{s \in S\; | \; \fp_{\iota(i), s} \ne \fp_{\iota(j), s}\} \in \cD.
\end{align}

Thus we have showed that for each prime ideal $\fp_i$ with $1 \le i \le g(P)$ in the factorization (\ref{e1-lem-main-lem0-ramifications-of-primes-in-O-L-alg}) of $P\cO_{\fL_{\alg}}$, there is a corresponding prime ideal $\fp_{\iota(i), s}$ in the factorization (\ref{e2-lem-main-lem0-ramifications-of-primes-in-O-L-alg}) of $P_s\cO_{\fS_s}$ such that
\begin{align*}
\{s \in S\; |\; \text{$\fp_{\iota(i), s} \ne \fp_{\iota(j), s}$ for all $1 \le i < j \le g(P)$}\} = \bigcap_{1 \le i < j \le g(P)}S_{i, j} \in \cD,
\end{align*}
where $S_{i, j}$ is defined in (\ref{e4-lem-main-lem0-ramifications-of-primes-in-O-L-alg}). Thus there are at least $g(P)$ distinct prime ideals $\fp_{\iota(1), s}, \ldots, \fp_{\iota(g(P)), s}$ in the factorization (\ref{e2-lem-main-lem0-ramifications-of-primes-in-O-L-alg}) of $P_s\cO_{\fS_s}$ for $\cD$-almost all $s\in S$. Therefore $g \ge g(P)$ as required.

\end{proof}

\begin{corollary}
\label{cor-totally-split-primes-globally-imply-locally-finite-case}

We maintain the same notation as in Lemma \ref{lem-main-lem0-ramifications-of-primes-in-O_L-alg}. If $P$ is totally split in $\fL_{\alg}$, then $P_s$ is totally split in $\fS_s$ for $\cD$-almost all $s \in S$.

\end{corollary}

\begin{proof}

Since $P$ is totally split in $\fL_{\alg}$, $g(P) = m = [\fL_{\alg} : \F]$. Thus we deduce from part (ii) of Lemma \ref{lem-main-lem0-ramifications-of-primes-in-O_L-alg} that $m \ge g \ge g(P) = m$, which implies that $g = m$. Since $[\fS_s : \fF_s] = m = g$, $P_s$ is totally split in $\fS_s$ for $\cD$-almost all $s \in S$.

\end{proof}

\begin{lemma}
\label{lem-main-lem1-ramifications-of-primes-in-O_L-alg}

Assume that $\bF_s$ is a perfect procyclic field for $\cD$-almost all $s \in S$. Let $P$ be an irreducible polynomial of degree $d \ge 1$ in $\A$. Write $P = \ulim_{s\in S}P_s$ for some irreducible polynomials $P_s \in \fA_s = \bF_s[t]$ (see Lemma \ref{lem-elementary-lemma0}). Let $\fp_s$ be a prime ideal in $\cO_{\fS_s}$ lying above the prime ideal $P_s\fA_s$, i.e., $\fp_s \cap \fA_s = P_s\fA_s$. Let $\fP = \prod_{s\in S}\fp_s/\cD$ be the ultra-prime ideal in $\cO_{\fS}$. Let $\fp = \fP_{\alg} = \fP \cap \F^{\alg}$ be the prime ideal in $\cO_{\fL_{\alg}}$ \footnote{See Proposition \ref{prop-algebraic-part-of-ultra-prime-is-prime-in-O_L-alg}.} that is the algebraic part of $\fP$ over $\F$. Then
\begin{itemize}

\item [(i)] $\fp$ lies above the prime ideal $P\A$, i.e., $\fp \cap \A = P\A$.

\item [(ii)] $\cO_{\fS}/\fP$ is a field that is isomorphic to the ultra-field $\prod_{s\in S}\left(\cO_{\fS_s}/\fp_s\right)/\cD$ \footnote{Since $\cO_{\fS_s}$ is a Dedekind domain, $\fp_s$ is a maximal ideal in $\cO_{\fS_s}$ and thus $\cO_{\fS_s}/\fp_s$ is a field. Thus $\prod_{s\in S}\left(\cO_{\fS_s}/\fp_s\right)/\cD$ is an ultra-field, and hence, by \L{}o\'s' theorem, is a field (see Nguyen \cite[Proposition 2.1]{nguyen-APAL-2024}).}.

\item [(iii)] there is a field homomorphism from $\cO_{\fL_{\alg}}/\fp$ to $\cO_{\fS}/\fP$.

\item [(iv)] there exists a positive integer $1 \le f \le m$ such that the following hold.
    \begin{itemize}
    
    \item [(a)] $\{s \in S \; | \; f(\fp_s/P_s\fA_s) = f\} \in \cD$, where $f(\fp_s/P_s\fA_s) = [\cO_{\fS_s}/\fp_s : \fA_s/P_s\fA_s]$ is the inertia degree of $\fp_s$ over $P_s\fA_s$.
        
        \item [(b)] $[\cO_{\fS}/\fP : \A/P\A] = f$; in particular, $\cO_{\fS}/\fP$ is the unique extension $\fK(df)$ of degree $df$ over $\fK$.
    
    \item [(c)] $f(\fp/P)$ divides $f$, where $f(\fp/P) = [\cO_{\fL_{\alg}}/\fp : \A/P\A]$ denote the inertia degree of $\fp$ over $P\A$.

\end{itemize}

\end{itemize}

\end{lemma}

\begin{proof}

For part (i), we see that $P_s \in \fp_s$ for $\cD$-almost all $s \in S$, and thus $P = \ulim_{s\in S}P_s \in \prod_{s\in S}\fp_s/\cD = \fP$. Since $P \in \A$, we deduce that $P \in \fP \cap \A \subset \fP \cap \F^{\alg} = \fp$. Thus $P\cO_{\fL_{\alg}} \subset \fp$, and therefore $P\A = \fp \cap \A$, which proves part (i).

For part (ii), since $\fp_s$ is a prime ideal in $\cO_{\fS_s}$, there is a canonical homomorphism $\psi_s : \cO_{\fS_s} \to \cO_{\fS_s}/\fp_s$ that sends each element $\alpha_s \in \cO_{\fS_s}$ to $\psi_s(\alpha_s) = \alpha_s + \fp_s$, and thus there exits a homomorphism
\begin{align*}
\psi : \cO_{\fS} = \prod_{s\in S}\cO_{\fS_s}/\cD \to \prod_{s\in S}\left(\cO_{\fS_s}/\fp_s\right)/\cD
\end{align*}
that sends each element $\alpha = \ulim_{s\in S}\alpha_s \in \cO_{\fS}$ with $\alpha_s \in \cO_{\fS_s}$ to $\psi(\alpha) = \ulim_{s\in S}\psi_s(\alpha_s) \in \prod_{s\in S}\left(\cO_{\fS_s}/\fp_s\right)/\cD$.

It is clear that $\fP = \prod_{s\in S}\fp_s/\cD$ is a subset of the kernel $\ker(\psi)$ of $\psi$. Take an arbitrary element $\alpha = \ulim_{s\in S}\alpha_s \in \ker(\psi)$. Then $\alpha_s \in \ker(\psi_s)$ for $\cD$-almost all $s \in S$, and thus $\alpha_s \in \fp_s$ for $\cD$-almost all $s \in S$ since $\ker(\psi_s) = \fp_s$. Therefore $\alpha = \ulim_{s\in S}\alpha_s \in \prod_{s\in S}\fp_s/\cD = \fP$, and thus $\ker(\psi) = \fP$.

Since $\psi_s$ is surjective for $\cD$-almost all $s \in S$, it is immediate from the definition of $\psi$ that $\psi$ is also surjective. Hence by the isomorphism theorem (see Lang \cite{lang-algebra}), there is an isomorphism from $\cO_{\fS}/\fP$ to $\prod_{s\in S}\left(\cO_{\fS_s}/\fp_s\right)/\cD$. Since $\cO_{\fS_s}$ is a Dedekind domain and $\fp_s$ is a prime ideal in $\cO_{\fS_s}$, we deduce that $\fp_s$ is a maximal ideal in $\cO_{\fS_s}$, and thus $\cO_{\fS_s}/\fp_s$ is a field for $\cD$-almost all $s \in S$. Thus the ultraproduct $\prod_{s\in S}\left(\cO_{\fS_s}/\fp_s\right)\cD$ is an ultra-field, and hence is a field (see Nguyen \cite[Proposition 2.1]{nguyen-APAL-2024}). Thus $\cO_{\fS}/\fP$ is a field. 

We now prove part (iii). By Proposition \ref{prop-algebraic-part-of-ultra-integers}, $\cO_{\fL_{\alg}} = (\cO_{\fS})_{\alg} $, and thus there exist a canonical embedding $\phi : \cO_{\fL_{\alg}} = (\cO_{\fS})_{\alg} = \cO_{\fS} \cap \F^{\alg} \to \cO_{\fS}$ such that $\phi(\alpha) = \alpha$ for all $\alpha \in \cO_{\fL_{\alg}}$, and a canonical ring homomorphism $\sigma : \cO_{\fS} \to \cO_{\fS}/\fP$ that sends each element $\alpha \in \cO_{\fS}$ to $\sigma(\alpha) = \alpha + \fP$. Thus we obtain the homomorphism $\rho := \sigma\circ \phi: \cO_{\fL_{\alg}} \to \cO_{\fS}/\fP$ such that $\rho(\alpha) = \alpha + \fP$ for all elements $\alpha \in \cO_{\fL_{\alg}}$. It is clear that $\fp = \fP_{\alg} = \fP \cap \F^{\alg} \subset \cO_{\fL_{\alg}}$ is a subset of $\ker(\rho)$. 

For each element $\alpha \in \ker(\rho)$, the definition of $\rho$ immediately implies that $\rho(\alpha) = \alpha + \fP = \fP$, and thus $\alpha \in \fP$. Therefore $\alpha \in \fP \cap \cO_{\fL_{\alg}} \subset \fP \cap \F^{\alg} = \fp$. Therefore $\ker(\rho) = \fp$, and thus there is a field homomorphism from $\cO_{\fL_{\alg}}/\fp$ to $\cO_{\fS}/\fP$. 

For part (iv), since the degree of $\fS_s$ over $\fF_s$ is $m = [\fL_{\alg}:\F]$ for $\cD$-almost all $s \in S$ (see Corollary \ref{cor-Galois-group-structure-of-shadows-of-algebraic-part-in-a-given-ultra-extension}), we deduce that $1 \le f(\fp_s/P_s) \le m$ (see Lang \cite{lang-ant}), and thus
\begin{align*}
S = \bigcup_{i = 1}^m S_i,
\end{align*}
where
\begin{align*}
S_i = \{s \in S \; | \; f(\fp_s/P_s) = i\}.
\end{align*}

Since $S \in \cD$, we deduce from Lemma \ref{lem-at-least-one-set-in-th-union-of-sets-is-in-D} that $S_f \in \cD$ for some integer $1 \le f \le m$, and thus
\begin{align*}
\{s \in S \; | \; f(\fp_s/P_s) = f\} \in \cD,
\end{align*}
which proves assertion (a) of part (iv).

Since $[\cO_{\fS_s}/\fp_s : \fA_s/P_s\fA_s] = f(\fp_s/P_s) = f$ for $\cD$-almost all $s \in S$ and $\cO_{\fS}/\fP$ is isomorphic to $\prod_{s\in S}(\cO_{\fS_s}/\fp_s)/\cD$, we deduce from Proposition \ref{prop-explicit-description--for--algebraic-extension-of-degree-d-of-ultra-fields} that $\cO_{\fS}/\fP$ is a finite algebraic extension of $\prod_{s\in S}(\fA_s/P_s\fA_s)/\cD$ such that
\begin{align*}
[\cO_{\fS}/\fP : \prod_{s\in S}(\fA_s/P_s\fA_s)/\cD] = \ulim_{s\in S}f(\fp_s/P_s) = f.
\end{align*}

Since $\prod_{s\in S}\fA_s/\cD = \cU(\A)$ and $P = \ulim_{s\in S}P_s$, we deduce from Nguyen \cite[Lemma 4.12]{nguyen-APAL-2024} that
\begin{align*}
\prod_{s\in S}(\fA_s/P_s\fA_s)/\cD \cong \cU(\A)/P\cU(\A) \cong \A/P\A,
\end{align*} 
and thus
\begin{align}
\label{e1-lem-main-lem1-for-unramified-primes-in-O-L-alg}
[\cO_{\fS}/\fP : \A/P\A] = f.
\end{align}

Since $P$ is an irreducible polynomial of degree $d$ in $\A$ and $\fK = \prod_{s\in S}\bF_s/\cD$ is a perfect procyclic field, $\A/P\A$ is the unique algebraic extension of degree $d$ over $\fK$, and so $\A/P\A = \fK(d)$. By (\ref{e1-lem-main-lem1-for-unramified-primes-in-O-L-alg}), $\cO_{\fS}/\fP$ is an extension of degree $f$ over $\A/P\A$, which implies that $\cO_{\fS}/\fP$ is the unique extension of degree $df$ over $\fK$. Therefore $\cO_{\fS}/\fP = \fK(df)$ as required.

For assertion (c) of part (iv), we know from part (b) that there is a field homomorphism from $\cO_{\fL_{\alg}}/\fp$ to $\cO_{\fS}/\fP$. Since every field homomorphism is injective, one can embed $\cO_{\fL_{\alg}}/\fp$ into $\cO_{\fS}/\fP$, and thus $\cO_{\fS}/\fP$ is an extension of $\cO_{\fL_{\alg}}/\fp$. Since $[\cO_{\fL_{\alg}}/\fp : \A/P\A] = f(\fp/P)$, $\cO_{\fL_{\alg}}/\fp$ is an extension of degree $f(\fp/P)$ over $\A/P\A$, and thus $\cO_{\fL_{\alg}}/\fp$ is the unique extension of degree $df(\fp/P)$ over $\fK$, i.e., $\cO_{\fL_{\alg}}/\fp = \fK(df(\fp/P))$.

Since $\cO_{\fS}/\fP = \fK(df)$, we deduce that $\cO_{\fS}/\fP$ is a finite extension of $\cO_{\fL_{\alg}}/\fp$, and thus $df(\fp/P))$ divides $df$. Therefore $f(\fp/P))$ divides $f$ as required.

\end{proof}

We strengthen Theorem \ref{thm-unramified-primes-locally-imply-unramified-globally-in-algebraic-parts-finite-case} to prove a necessary and sufficient condition for a prime in $\A$ to be unramified or ramified in $\fL_{\alg}$.

\begin{theorem}
\label{thm-big-theorem-I-the-finite-case}

Assume that $\bF_s$ is a perfect procyclic field for $\cD$-almost all $s \in S$. Let $P$ be an irreducible polynomial, i.e., a prime in $\A = \fK[t]$. Write $P = \ulim_{s\in S}P_s$ (see Lemma \ref{lem-elementary-lemma0}), where $P_s$ is an irreducible polynomial in $\fA_s = \bF_s[t]$ for $\cD$-almost all $s \in S$. Then 
\begin{itemize}

\item [(i)] $P$ is unramified in $\fL_{\alg}$ if and only if $P_s$ is unramified in $\fS_s$ for $\cD$-almost all $s \in S$. 

\item [(ii)] $P$ is ramified in $\fL_{\alg}$ if and only if $P_s$ is ramified in $\fS_s$ for $\cD$-almost all $s \in S$.

\item [(iii)] $P$ is totally split in $\fL_{\alg}$ if and only if $P_s$ is totally split in $\fS_s$ for $\cD$-almost all $s \in S$.

    \item [(iv)] if $P_s$ is totally ramified in $\fL_s$ \footnote{In the case where $\fL_s$ is an infinite Galois extension of $\fF_s$, the condition that $P_s$ is totally ramified in $\fL_s$, is equivalent to that $P_s$ is totally ramified in every finite Galois subextension of $\fL_s/\fF_s$.} for $\cD$-almost all $s \in S$, then $P$ is totally ramified in $\fL_{\alg}$. 
        
        \item [(v)] if $P$ is inert in $\fL_{\alg}$, then $P_s$ is inert in $\fS_s$ for $\cD$-almost all $s \in S$. 
    
\end{itemize}

\end{theorem}

\begin{proof}

Let
\begin{align}
\label{e1-thm-unramified-primes-globally-imply-locally-in-O-L-alg-finite-case}
P\cO_{\fL_{\alg}} = \left(\fp_1\cdots \fp_{g(P)}\right)^{e(P)},
\end{align}
be the prime ideal factorization of $P\cO_{\fL_{\alg}}$ in $\cO_{\fL_{\alg}}$, where $e(P)$ denotes the ramification index $e(P)$ of the $\fp_i$ over $P\A$, and $g(P)$ denotes the number of distinct prime ideals in the factorization of $P\cO_{\fL_{\alg}}$. Let $f(P)$ denote the inertia degree of $\fp_i$ over $P\A$ for $1\le i \le g(P)$, i.e., $f(P) = [\cO_{\fL_{\alg}}/\fp_i : \A/P\A]$ for all $1 \le i \le g(P)$.

By Lemma \ref{lem-main-lem0-ramifications-of-primes-in-O_L-alg}, there exist positive integers $1 \le e, f, g \le m$ such that there exist prime ideals $\fp_{1, s}, \ldots, \fp_{g, s}$ in $\cO_{\fS_s}$ for which 
\begin{align}
\label{e2-thm-unramified-primes-globally-imply-locally-in-O-L-alg-finite-case}
P_s\cO_{\fS_s} = \left(\fp_{1, s}\cdots \fp_{g, s}\right)^e
\end{align}
for $\cD$-almost all $s \in S$, where $f = [\cO_{\fS_s}/\fp_{i, s}: \fA_s/P_s\fA_s]$ denotes the inertia degree of $\fp_{i, s}$ over $P_s$ for all $1 \le i \le g$. Furthermore 
\begin{align}
\label{e3-thm-unramified-primes-globally-imply-locally-in-O-L-alg-finite-case}
efg = e(P)f(P)g(P) = m.
\end{align}

We first prove part (i).

If $P_s$ is unramified in $\fS_s$ for $\cD$-almost all $s \in S$, then Theorem \ref{thm-unramified-primes-locally-imply-unramified-globally-in-algebraic-parts-finite-case} implies that $P$ is unramified in $\fL_{\alg}$ \footnote{This implication does not use the assumption that $\bF_s$ is a perfect procyclic field for $\cD$-almost all $s \in S$.}

Assume that $P$ is unramified in $\fL_{\alg}$. Then $e(P) = 1$. By part (iv)(c) of Lemma \ref{lem-main-lem1-ramifications-of-primes-in-O_L-alg}, $f(P)$ divides $f$, and thus $f(P) \le f$. Part (ii) of Lemma \ref{lem-main-lem0-ramifications-of-primes-in-O_L-alg} implies that $g(P) \le g$, and it thus follows from (\ref{e3-thm-unramified-primes-globally-imply-locally-in-O-L-alg-finite-case}) that
\begin{align*}
efg = e(P)f(P)g(P) = f(P)g(P) \le fg,
\end{align*}
and thus $e \le 1$. Therefore $e = 1$, and thus $P_s$ is unramified in $\cO_{\fS_s}$ for $\cD$-almost all $s \in S$. Thus part (i) follows immediately.

Part (ii) follows immediately from part (i) since $P$ is either unramified or ramified in $\cO_{\fL_{\alg}}$.

We now prove part (iii). If $P$ is totally split in $\fL_{\alg}$, then Corollary \ref{cor-totally-split-primes-globally-imply-locally-finite-case} implies that $P_s$ is totally split in $\fS_s$ for $\cD$-almost all $s \in S$ \footnote{This implication does not need the assumption that $\bF_s$ is a perfect procyclic field for $\cD$-almost all $s \in S$.}

If $P_s$ is totally split in $\fS_s$ for $\cD$-almost all $s \in S$, then $e = f = 1$, and $g = m = [\fS_s : \fF_s]$. By Lemma \ref{lem-main-lem1-ramifications-of-primes-in-O_L-alg}, $f(P)$ divides $f = 1$, and thus $f(P) = 1$.

Since $e = 1$, $P_s$ is unramified in $\fS_s$ for $\cD$-almost all $s \in S$, and thus we deduce from part (i) that $P$ is unramified in $\fL_{\alg}$. Hence $e(P) = 1$. Since $e(P)f(P)g(P) = m$, we deduce that $g(P) = m = [\fL_{\alg} : \F]$, which implies that $P$ is totally split in $\fL_{\alg}$. 

For part (iv), since $P_s$ is totally ramified in $\fL_s$ for $\cD$-almost all $s \in S$ and $\fS_s/\fF_s$ is a finite Galois subextension of $\fL_s/\fF_s$, $P_s$ is also totally ramified in $\fS_s$ for $\cD$-almost all $s \in S$, and thus $e = m = [\fS_s : \fF_s]$, $f = g = 1$. By Lemma \ref{lem-main-lem1-ramifications-of-primes-in-O_L-alg}, $f(P)$ divides $f = 1$, and thus $f(P) = 1$. Part (ii) of Lemma \ref{lem-main-lem0-ramifications-of-primes-in-O_L-alg} implies that $g(P) \le g = 1$, and hence $g(P) = 1$. Therefore 
\begin{align*}
e(P) = \dfrac{m}{f(P)g(P)} = m = [\fL_{\alg} : \F],
\end{align*}
which proves that $P$ is totally ramified in $\fL_{\alg}$.

We now prove part (v). Since $P$ is inert in $\fL_{\alg}$, $f(P) = m = [\fL_{\alg}: \F]$. By Lemma \ref{lem-main-lem1-ramifications-of-primes-in-O_L-alg}, $f(P) = m$ divides $f$, and thus since $f \le m$, $f$ must equal $m$. Therefore $P_s$ is inert for $\cD$-almost all $s \in S$.

\end{proof}

\begin{remark}
\label{rem-splitting-of-arbitrary-primes-P_s-when-F_s-is-algebraically-closed}

\begin{itemize}

\item []

\item [(i)] Following the proof of part (iv) of the above theorem, we see that if $P_s$ is totally ramified in the $s$-th shadow $\fS_s$ of $\fL_{\alg}$ for $\cD$-almost all $s \in S$, then $P$ is totally ramified in $\fL_{\alg}$. 

\item [(ii)] In the above theorem, we must start with arbitrary primes $P$ from the side of $\F = \fK(t)$, and study ramification of $P$ in algebraic extensions of $\F$, based on ramifications of primes $P_s$ from the side of the sequence of function fields $\{\fF_s = \bF_s(t)\; | \; s\in S\}$, where the $P_s$ are the $s$-th components of $P$ in the ultra-limit $P = \ulim_{s\in S}P_s$. 
    
    In general, we cannot begin with arbitrary primes $(P_s)_{s\in S}$ from the side of function fields $\{\fF_s = \bF_s(t)\; | \; s\in S\}$, and analyze ramifications of primes $(P_s)_{s\in S}$, using ramification of the ultra-prime $P = \ulim_{s\in S}$ since $P$, in general, fails to be a prime on the side of $\F$. In fact, such ultra-primes $P$ may not even belong to $\F$; for example, whenever the degrees of primes $P_s \in \fA_s = \bF_s[t]$ are unbounded for $\cD$-almost all $s \in S$, i.e., for every constant $N > 0$, the set $\{s \in S\; | \; \deg(P_s) > N\}$ belongs to the ultrafilter $\cD$, then the ultra-prime $P = \ulim_{s\in S}P_s$ belongs to the ultra-hull $\cU(\A)$, but is not an element in $\A$.

   For studying ramifications of arbitrary primes $(P_s)_{s\in S}$ simultaneously from an asymptotical point of view, i.e., ramifications of primes $P_s$ for $\cD$-almost all $s \in S$, we can strengthen the assumptions of Theorem \ref{thm-big-theorem-I-the-finite-case} at the constant field level by further assuming that the constant fields $\bF_s$ of $\fF_s = \bF_s(t)$ are \textbf{algebraically closed for $\cD$-almost all $s \in S$}, which implies that the constant field $\fK = \prod_{s\in S}\bF_s/\cD$ of $\F$ is also algebraically closed. In this case, for a given sequence of primes $(P_s)_{s\in S}$ in the sequence of function fields $(\fF_s)_{s\in S}$, the corresponding ultra-prime $P = \ulim_{s\in S}P_s$ is also a prime in $\F$ since all primes in $\fF_s$ are of degree $1$ for $\cD$-almost all $s \in S$. Thus we can obtain ramifications of arbitrary primes $(P_s)_{s\in S}$ from the side of the function fields $(\fF_s)_{s\in S}$ from the ramification information of its ultra-limit $P = \ulim_{s\in S}P_s$ from the side of $\F$. The proof of this observation is unchanged and follows the same arguments as in Theorem \ref{thm-big-theorem-I-the-finite-case}. We record this remark in the following.

\end{itemize}

\end{remark}

\begin{theorem}
\label{thm-big-theorem-I-the-finite-case-constant-fields-are-ACF}

Assume that $\bF_s$ is an algebraically closed field for $\cD$-almost all $s \in S$ so that by \L{}o\'s' theorem, $\fK = \prod_{s\in S}\bF_s/\cD$ is also algebraically closed. Let  $P_s$ be an irreducible polynomial in $\fA_s = \bF_s[t]$ for $\cD$-almost all $s \in S$ so that $P = \ulim_{s\in S}P_s$ is also a prime in $\A$. Then 
\begin{itemize}

\item [(i)] $P$ is unramified in $\fL_{\alg}$ if and only if $P_s$ is unramified in $\fS_s$ for $\cD$-almost all $s \in S$. 

\item [(ii)] $P$ is ramified in $\fL_{\alg}$ if and only if $P_s$ is ramified in $\fS_s$ for $\cD$-almost all $s \in S$.

\item [(iii)] $P$ is totally split in $\fL_{\alg}$ if and only if $P_s$ is totally split in $\fS_s$ for $\cD$-almost all $s \in S$.

    \item [(iv)] if $P_s$ is totally ramified in $\fL_s$ for $\cD$-almost all $s \in S$, then $P$ is totally ramified in $\fL_{\alg}$. 
        
        \item [(v)] if $P$ is inert in $\fL_{\alg}$, then $P_s$ is inert in $\fS_s$ for $\cD$-almost all $s \in S$. 
    
\end{itemize}

\end{theorem}

\subsection{Prime ramifications in $\fL_{\alg}$ when $\fL_{\alg}$ is an infinite Galois extension of $\F$}
\label{subsec-prime-ramification-in-L-alg-infinite-case}

Throughout this subsection, we assume $\fL_{\alg}$ is an infinite Galois extension of $\F$. Note that $\fL_s$ may be an infinite Galois extension of $\fF_s$ for $\cD$-almost all $s \in S$. 

One can write
\begin{align*}
\fL_{\alg} = \bigcup_{i \in I}\H_i,
\end{align*}
where the extensions $\H_i/\F$ range over all the finite Galois subextensions of $\fL_{\alg}/\F$. 

We recall that a prime $P \in \A$ is \textbf{unramified in $\fL_{\alg}$} if and only if $P$ is unramified in $\H_i$ for all $i \in I$. Thus $P$ is \textbf{ramified in $\fL_{\alg}$} if and only if $P$ is ramified in some finite Galois subextension $\H_i/\F$ of $\fL_{\alg}/\F$.

Take an arbitrary finite Galois subextension $\H_i/\F$ of $\fL_{\alg}/\F$. Let $\fH_{i, s}$ denote the $s$-th shadow of $\H_i$ for $\cD$-almost all $s \in S$, and let $\fH_i = \prod_{s\in S}\fH_{i, s}/\cD$ denote the ultra-shadow of $\H_i$. By Proposition \ref{prop-shadows-of-the-algebraic-part-are-inside}, $\fH_{i, s} \subset \fL_s$ for $\cD$-almost all $s \in S$, and $\fH_i \subset \fL$. By definition of shadows and ultra-shadows (see Definition \ref{def-shadows-of-algebraic-extensions}), we know that the algebraic part of $\fH_i$ over $\F$ is $\H_i$, i.e., ${\fH_i}_{\alg} = \fH_i \cap \F^{\alg} = \H_i$ for every $i \in I$. 

We obtain the following result that is weaker than Theorem \ref{thm-big-theorem-I-the-finite-case} in the case where $\fL_{\alg}$ is an infinite extension of $\F$.

\begin{theorem}
\label{thm-big-theorem-II-infinite-case}

Assume that $\bF_s$ is a perfect procyclic field for $\cD$-almost all $s \in S$. Let $P$ be an irreducible polynomial, i.e., a prime in $\A = \fK[t]$. Write $P = \ulim_{s\in S}P_s$ (see Lemma \ref{lem-elementary-lemma0}), where $P_s$ is an irreducible polynomial in $\fA_s = \bF_s[t]$ for $\cD$-almost all $s \in S$. Then 
\begin{itemize}

\item [(1)] 
\begin{itemize}

\item [(A1)] If $P$ is unramified in $\fL_{\alg}$ then for every $i \in I$, $P_s$ is unramified in $\fH_{i, s}$ for $\cD$-almost all $s \in S$, i.e.,
    \begin{align*}
    \{s \in S\; | \; \text{$P_s$ is unramified in $\fH_{i, s}$}\} \in \cD
    \end{align*}
for every $i \in I$.

\item [(B1)] if $P_s$ is unramified in $\fL_s$ for $\cD$-almost all $s \in S$, then $P$ is unramified in $\fL_{\alg}$.

\end{itemize}

\item [(2)] $P$ is ramified in $\fL_{\alg}$ if and only if there exists an element $i \in I$ such that $P_s$ is ramified in $\fH_{i, s}$ for $\cD$-almost all $s \in S$.

\item [(3)] 
\begin{itemize}

\item [(A3)] if $P$ is totally split in $\fL_{\alg}$, then for every $i \in I$, $P_s$ is totally split in $\fH_{i, s}$ for $\cD$-almost all $s \in S$, that is, 
    \begin{align*}
    \{s \in S\; | \; \text{$P_s$ is totally split in $\fH_{i, s}$}\} \in \cD 
    \end{align*}
    for every $i \in I$.

\item [(B3)] If $P_s$ is totally split in $\fL_s$ for $\cD$-almost all $s \in S$, then $P$ is totally split in $\fL_{\alg}$. 

\end{itemize}

    \item [(4)] if $P_s$ is totally ramified in $\fL_s$ for $\cD$-almost all $s \in S$, then $P$ is totally ramified in $\fL_{\alg}$. 
        
        \item [(5)] if $P$ is inert in $\fL_{\alg}$, then for every $i \in I$, $P_s$ is inert in $\fH_{i, s}$ for $\cD$-almost all $s \in S$. 
    
\end{itemize}

\end{theorem}

\begin{proof}

We first prove (A1) of part (1). If $P$ is unramified in $\fL_{\alg}$, then $P$ is unramified in every finite Galois subextension $\H_i/\F$ of $\fL_{\alg}/\F$. Since $\H_i$ is a finite Galois extension of $\F$, we know from Theorem \ref{thm-big-theorem-I-the-finite-case} that $P_s$ is unramified in $\fH_{i, s}$ for $\cD$-almost all $s \in S$, which proves (A1).

For (B1), suppose that $P_s$ is unramified in $\fL_s$ for $\cD$-almost all $s \in S$. Take an arbitrary element $i \in I$. Since $\fH_{i, s}/\fF_s$ is a finite Galois subextension of $\fL_s/\fF_s$, we see that $P_s$ is unramified in $\fH_{i, s}$ for $\cD$-almost all $s \in S$. Since $\H_i$ is a finite Galois extension of $\F$, part (i) of Theorem \ref{thm-big-theorem-I-the-finite-case} implies that $P$ is unramified in $\H_i$. Since $i$ is arbitrary in $I$ and $\fL_{\alg} = \cup_{i \in I}\H_i$, $P$ is unramified in $\fL_{\alg}$. 

We now prove part (2). Suppose that $P$ is ramified in $\fL_{\alg}$. Then there exists an element $i \in I$ such that $P$ is ramified in $\H_i$. By part (ii) of Theorem \ref{thm-big-theorem-I-the-finite-case}, $P_s$ is ramified in $\fH_{i, s}$ for $\cD$-almost all $s \in S$.

For (A3) of part (3), suppose that $P$ is totally split in $\fL_{\alg}$. Then $P$ is totally split in $\H_i$ for every $i \in I$. By part (iii) of Theorem \ref{thm-big-theorem-I-the-finite-case}, we deduce that for all $i \in I$, $P_s$ is totally split in $\fH_{i, s}$ for $\cD$-almost all $s \in S$. 

(B3) of part (3) follows immediately using the same arguments as above and part (iii) of Theorem \ref{thm-big-theorem-I-the-finite-case}.

Since ${\fH_i}_{\alg} = \H_i$ for every $i \in I$, parts (4) and (5) follow immediately using similar arguments as above and parts (iv) and (v) of Theorem \ref{thm-big-theorem-I-the-finite-case}, respectively.

\end{proof}

In the same spirit of part (ii) in Remark \ref{rem-splitting-of-arbitrary-primes-P_s-when-F_s-is-algebraically-closed}, we obtain the following result from the above theorem that allows to study ramifications of arbitrary primes $(P_s)_{s\in S}$ from the side of function fields $(\fF_s)_{s\in S}$, based on ramification of its ultra-limit $P = \ulim_{s\in S}P_s$ from the side of $\F$, under the assumption that $\bF_s$ is algebraically closed for $\cD$-almost all $s \in S$. 

\begin{theorem}
\label{thm-big-theorem-II-infinite-case-F_s-algebraically-closed-field}

Assume that $\bF_s$ is an algebraically closed field for $\cD$-almost all $s \in S$ so that by \L{}o\'s' theorem, $\fK = \prod_{s\in S}\bF_s/\cD$ is also algebraically closed. Let  $P_s$ be an irreducible polynomial in $\fA_s = \bF_s[t]$ for $\cD$-almost all $s \in S$ so that $P = \ulim_{s\in S}P_s$ is also a prime in $\A$. Then

\begin{itemize}

\item [(1)] 
\begin{itemize}

\item [(A1)] If $P$ is unramified in $\fL_{\alg}$ then for every $i \in I$, $P_s$ is unramified in $\fH_{i, s}$ for $\cD$-almost all $s \in S$, i.e.,
    \begin{align*}
    \{s \in S\; | \; \text{$P_s$ is unramified in $\fH_{i, s}$}\} \in \cD
    \end{align*}
for every $i \in I$.

\item [(B1)] if $P_s$ is unramified in $\fL_s$ for $\cD$-almost all $s \in S$, then $P$ is unramified in $\fL_{\alg}$.

\end{itemize}

\item [(2)] $P$ is ramified in $\fL_{\alg}$ if and only if there exists an element $i \in I$ such that $P_s$ is ramified in $\fH_{i, s}$ for $\cD$-almost all $s \in S$.

\item [(3)] 
\begin{itemize}

\item [(A3)] if $P$ is totally split in $\fL_{\alg}$, then for every $i \in I$, $P_s$ is totally split in $\fH_{i, s}$ for $\cD$-almost all $s \in S$, that is, 
    \begin{align*}
    \{s \in S\; | \; \text{$P_s$ is totally split in $\fH_{i, s}$}\} \in \cD 
    \end{align*}
    for every $i \in I$.

\item [(B3)] If $P_s$ is totally split in $\fL_s$ for $\cD$-almost all $s \in S$, then $P$ is totally split in $\fL_{\alg}$. 

\end{itemize}

    \item [(4)] if $P_s$ is totally ramified in $\fL_s$ for $\cD$-almost all $s \in S$, then $P$ is totally ramified in $\fL_{\alg}$. 
        
        \item [(5)] if $P$ is inert in $\fL_{\alg}$, then for every $i \in I$, $P_s$ is inert in $\fH_{i, s}$ for $\cD$-almost all $s \in S$. 
    
\end{itemize}

\end{theorem}

\subsection{Ramification in the inverse Galois problem}
\label{subsec-ramificaion-in-the-inverse-Galois-problem}

In this subsection, we focus on a refinement of the inverse Galois problem which explores which finite groups appear as the Galois group of an extension of the rational function field over an algebraically closed field of positive characteristic in which only a given set of primes may ramify. The following result is well-known, but we provide a new look from a model-theoretic viewpoint at the result. 

\begin{proposition}
\label{prop-ramification-in-the-inverse-galois-problem}

Let $G$ be a finite group with minimum number of generators $n$. Then there is a finite set $\cP$ of prime numbers such that for all primes $p \not\in \cP$, if $k_p$ is an algebraically closed field of characteristic $p$ and $K_p = k_p(t)$, then for any integer $m \ge n$, $G$ is realizable as a Galois group of a Galois extension $L_p/K_p$, unramified outside a set $\cA_p$ containing $m + 1$ primes.

\end{proposition}

\begin{proof}

Assume that there exists a finite group $G$ with minium number of generators $n$ such that the assertion fails for infinitely many primes $p$, i.e., there exists an infinite set $S$ of primes $p$ such that for each prime $p \in S$, there exist an algebraically closed field $k_p$ of characteristic $p$, an integer $m \ge n$, a set $\cA_p$ of $m + 1$ primes in $K_p = k_p(t)$  such that $G$ is not realizable as a Galois group of any Galois extension of $K_p$, unramified outside $\cA_p$.

Let $\cD$ be a nonprincipal ultrafilter on $S$. Setting $\fK = \prod_{p \in S}k_p/\cD$, we know from \L{}o\'s' theorem that $\fK$ is an algebraically closed field of characteristic $0$. Let $\F = \fK(t)$ be the rational function field over $\fK$.

Since $K_p$ is a rational function field over an algebraically closed field, all primes in $\cA_p$ of degree $1$, and thus the set of ultra-primes $\cA = \prod_{p \in S}\cA_p/\cD$ consists of primes of degree $1$ in $\F$. Since each $\cA_p$ is of cardinality $m + 1$, it follows from Bell--Solomon \cite[Lemma 3.7]{bell-slomson} that $\cA = \prod_{p \in S}\cA_p/\cD$ has exactly $m + 1$ primes in $\F$.

Since $\fK$ is an algebraically closed field of characteristic $0$, we deduce from a result of Jarden (see \cite[Section 1.8]{jarden-1995}) for the rational function field $\F = \fK(t)$ that there exists a Galois extension $\H$ of $\F$, unramified outside $\cA$ such that $G$ is realizable as a Galois group of $\H$ over $\F$ \footnote{If $k_p$ is the algebraic closure of the finite field $\bF_p$ for $\cD$-almost all $p \in S$, $\fK = \prod_{s\in S}k_p/\cD$ is isomorphic to the complex field $\bC$ (see Schoutens \cite{schoutens}). Then we can use Riemann Existence Theorem over $\bC(t)$ (see V\"olklein \cite{volklein-book}) to confirm the existence of such Galois extension $\H/\F$.}. Applying the theory of shadows in Subsection \ref{subsec-shadows-and-ultra-shadows} with $k_p, K_p = k_p(t)$ in place of $\bF_s, \fF_s = \bF_s(t)$, respectively, we can construct a Galois extension $\fH_p$ of $K_p$ whose Galois group is isomorphic to $G$ for $\cD$-almost all $p \in S$, where $\fH_p$ is the $p$-th shadow of $\H$ over $K_p$ (see Definition \ref{def-shadows-of-algebraic-extensions} and Theorem \ref{thm-main-thm3-Galois-group-structures-of-shadows}). Let $\fH = \prod_{p\in S}\fH_p/\cD$ be the ultra-shadow of $\H$.

Since $G$ is the Galois group of $\fH_p$ over $K_p$ for $\cD$-almost all $p \in S$ and $G$ is not realizable as a Galois group of any Galois extension of $K_p$, unramified outside $\cA_p$ for all $p \in S$, $\fH_p$ must be not unramified outside $\cA_p$ for $\cD$-almost all $p \in S$. Thus
\begin{align}
\label{e-prop-ramification-in-the-inverse-Galois-problem}
\{p \in S\; | \; \text{there exists a prime $Q_p$ in $K_p$ that is not in $\cA_p$ such that $Q_p$ is ramified in $\fH_p$}\} \in \cD.
\end{align}

We see that the ultra-prime $Q = \ulim_{p\in S}Q_p$ is a prime of degree $1$ in $\F$ that is not in $\cA = \prod_{p \in S}\cA_p/\cD$. By the construction of $\H$, $Q$ is unramified in $\H$. Since $\fH$ is the ultra-shadow of $\H$, the algebraic part of $\fH$ over $\F$ coincides with $\H$, i.e., $\fH_{\alg} = \fH \cap \F^{\alg} = \H$. Thus by part (i) of Theorem \ref{thm-big-theorem-I-the-finite-case-constant-fields-are-ACF} with $K_p, \fH_p, \H$ in place of $\fF_s, \fS_s, \fL_{\alg}$, respectively, we deduce that $Q_p$ is unramified in $\fH_p$ for $\cD$-almost all $p \in S$, i.e.,
\begin{align*}
\{p \in S\; | \; \text{$Q_p$ is unramified in $\fH_p$}\} \in \cD,
\end{align*}
which is a contradiction to (\ref{e-prop-ramification-in-the-inverse-Galois-problem}). Thus Proposition \ref{prop-ramification-in-the-inverse-galois-problem} follows immediately.

\end{proof}

\begin{remark}
\label{rem-Harbater}

In a private correspondence, Harbater pointed out to me that Proposition \ref{prop-ramification-in-the-inverse-galois-problem} follows immediately from \cite[Corollary 2.12, XIII]{SGA1} that was proven using the corresponding classical statement over the complex numbers whose proof is based on analysis and topology. He also pointed out to me that Corollary 2.12 in Chapter XIII in \cite{SGA1} provided background for Abhyankar's conjecture that was proven in full generality by Raynaud \cite{raynaud-cole-prize} (for the affine line) and Harbater \cite{harbater-cole-prize} (for all affine curves).

\end{remark}

\section{Hilbert's 12th Problem for certain rational function fields}
\label{sec-Hilbert-12th-problem}

In this section, using the theory of algebraic parts of ultra-field extensions and shadows that we develop in Subsections \ref{subsec-algebraic-parts} and \ref{subsec-shadows-and-ultra-shadows}, we prove the main result in this paper that can be viewed as an analogue of the Kronecker--Weber theorem for rational function fields over $n$-th level ultra-finite fields, which provides a solution to Hilbert's 12th Problem for such rational function fields.

\subsection{Cyclotomic function fields in positive characteristics and the Kronecker--Weber theorem for $\bF_q[t]$}
\label{subsec-KW-theorem-for-Fq[t]}

In this subsection, we recall a notion of cyclotomic function fields and an analogue of the Kronecker--Weber theorem for rational function fields over finite fields. Throughout this section, let $A = \bF_q[t]$ denote the polynomial ring over a finite field $\bF_q$, where $q$ is a power of a prime $p$. Let $L = \bF_q(t)$ be the rational function field over $\bF_q$. Let $\tau$ denote the Frobenius map that sends $x$ to $x^q$. Let $L\langle \tau \rangle$ denote the ring of polynomials in $\tau$ with twisted multiplication, that is, $\tau a = a^q \tau$ for every $a \in L$.

 Let $C : A \to L\langle \tau \rangle, \; a \mapsto C_a$ be the $\bF_q$-algebra homomorphism such that 
 \begin{align*}
 	C_t = t + \tau,
 \end{align*}
 or equivalently $C_t(\beta) = t\beta + \beta^q$ for all $\beta \in L^{\alg}$, where $L^{\alg}$ is an algebraic closure of $L$.
 
 The $\bF_q$-algebra homomorphism $C$ is called the \textbf{Carlitz module for $A$ defined over $L$} (see Goss \cite{goss} or Rosen \cite{rosen} for a detailed exposition of Carlitz module.)

For every commutative $L$-algebra $B$, one can introduce an $A$-algebra structure using the Carlitz module $C$ as follows. For every $a \in A$ and $\beta \in B$, define an action of $A$ on $B$ by
\begin{align}
\label{e-Carlitz-module-action}
a \cdot \beta := C_a(\beta).
	\end{align}

This action ``$\cdot$" equips $B$ with a new algebra structure that is different from the action of $A$ on $B$ induced from the $L$-algebra structure of $B$. 

For each nonzero element $a \in A$, we define
\begin{align}
	\label{def-torsions-of-the-Carlitz-module}
	\Lambda_C[a] := \{\beta \in L^{\alg} \; | \; C_a(\beta) = 0\}.
\end{align}

$\Lambda_C[a]$ is equipped with an $A$-module structure, using the action (\ref{e-Carlitz-module-action}) of $A$. It is known (see \cite[Proposition 12.4, p. 201]{rosen}) that 
\begin{align}
\label{e-structure-of-torison-module-Lambda_C}
	\Lambda_C[a] \cong A/aA
\end{align}

\begin{definition}
	\label{def-cyclotomic-function-fields-p>0}
(cyclotomic function fields for $L = \bF_q(t)$)	

	For each nonzero polynomial $a \in A$, the field $Q_a = L(\Lambda_C[a])$ is called the \textbf{$a$-th cyclotomic function field}.

\end{definition}

\begin{theorem}
	\label{thm-P^h-cyclotomic-function-fields-structure}
	$(\text{see Drinfeld \cite{drinfeld1, drinfeld2} and Hayes \cite[Proposition 2.2]{hayes}})$
	
	Let $P$ be a monic irreducible polynomial in $A$, and let $h$ be a positive integer. Let $Q_{P^h} = L(\Lambda_C[P^h])$ be the $P^h$-th cyclotomic function field. Then
	\begin{itemize}
		\item [(i)] $Q_{P^h}$ is a finite Galois extension over $L$ such that
		\begin{align*}
			\Gal(Q_{P^h}/L) \cong (A/P^hA)^{\times}.
		\end{align*}
		
		\item [(ii)] $Q_{P^h}$ is unramified at every prime ideal $\fp A$ with $\fp A \ne PA$.
		
		\item [(iii)] $PA$ is totally ramified in $Q_{P^h}$ and the only prime ideal above $PA$ is $\lambda \cO_{Q_{P^h}}$, where $\cO_{Q_{P^h}}$ is the ring of integers of $Q_{P^h}$ and $\lambda$ is any generator of $\Lambda_C[P^h]$.

	\end{itemize}

\end{theorem}

\begin{theorem}
	\label{thm-a-cyclotomic-function-fields-structure}
	(see Drinfeld \cite{drinfeld1, drinfeld2} and Hayes \cite[Theorem 2.3]{hayes})
	
	Let $a = \alpha P_1^{h_1}\cdots P_r^{h_r}$ be a polynomial of positive degree in $A$, where $\alpha \in \bF_q$ and the $P_i$ are monic irreducible polynomials in $A$. Let $Q_a$ be the compositum of the fields $Q_{P_i^{h_i}}$. Then
	\begin{itemize}
	
		\item [(i)] $Q_a$ is a finite abelian extensionover $L$ such that
	\begin{align*}
		\Gal(Q_a/L) \cong (A/aA)^{\times}.
	\end{align*}
	
	\item [(ii)] The only ideals in $A$ that ramify in $\cO_{Q_a}$ are $P_1A, P_2A, \ldots, P_rA$. 
	
	\end{itemize}

\end{theorem}

\subsection{An analogue of cyclotomic function fields for rational function fields over $n$-th level ultra-finite fields}
\label{subsec-cyclotomic-function-fields-over-nth-level-ultra-finite-fields}

In this subsection, we develop an analogue of cyclotomic function fields for rational function fields over $n$-th level ultra-finite fields. We will construct such an analogue inductively on $n$.

We begin by fixing notation that we will use throughout this subsection. We always denote by $\fK^{(n)} = \prod_{s\in S}\bF_s^{(n - 1)}/\cD$ \footnote{In Definition \ref{def-higher-dimensional-ultra-finite-fields}, we use a sequence of infinite sets $S_1, S_2, \ldots$ and a sequence of nonprincipal ultrafilters $\cD_1, \cD_2, \ldots$ on $S_1, S_2, \ldots$, respectively to define $n$-th level ultra-finite fields at each level $n$. Throughout this section, since we only need to deal with exactly one $n$-th level ultra-finite field each time, for the sake of brevity, we \textit{only} use one infinite set $S$ and a nonprincipal ultrafilter $\cD$ on $S$ to describe such an $n$-th level ultra-finite field without causing any confusion.} an $n$-th level ultra-finite field, where the superscript $n$ indicates the level of the ultra-finite field $\fK$, and $\bF_s^{(n - 1)}$ is an $(n - 1)$-th level ultra-finite field for $\cD$-almost all $s \in S$. The superscript $n - 1$ in the notation $\bF_s^{(n - 1)}$ also indicates the level of the ultra-finite field $\bF_s^{(n - 1)}$. When $n = 1$, $\fK^{(1)} = \prod_{s\in S}\bF_s^{(0)}/\cD$ is a $1$st level ultra-finite field and $\bF_s^{(0)}$ is a finite field of $q_s$ elements for $\cD$-almost all $s \in S$, where $q_s$ is a power of a prime number $p_s > 0$. Such an ultra-finite field is known as a nonprincipal ultraproduct of finite fields that was first studied systematically by Ax \cite{ax-1968}. In this case, we remove all superscripts, and simply write $\fK = \prod_{s\in S}\bF_s/\cD$. For brevity, we simply call $\fK$ an ultra-finite field instead of a $1$st level ultra-finite field.

For an $n$-th level ultra-finite field $\fK^{(n)} = \prod_{s\in S}\bF_s^{(n - 1)}/\cD$, let $\F^{(n)} = \fK^{(n)}(t)$ be the rational function field over $\fK^{(n)}$, and $\A^{(n)} = \fK^{(n)}[t]$ the polynomial ring over $\fK^{(n)}$. The quotient field of $\A^{(n)}$ is $\F^{(n)}$. Let $\fF_s^{(n - 1)} = \bF_s^{(n - 1)}(t)$ be the rational function field over $\bF_s^{(n-1)}$, and let $\fA_s^{(n-1)} = \bF_s^{(n-1)}[t]$ denote the polynomial ring over $\bF_s^{(n-1)}$ for $\cD$-almost all $s \in S$. The quotient field of $\fA_s^{(n-1)}$ is $\fF_s^{(n-1)}$. 

Let $\cU(\A^{(n)})$ be the ultra-hull of $\A^{(n)}$, i.e., $\cU(\A^{(n)})$ is the ultraproduct of $\fA_s^{(n - 1)} = \bF_s^{(n - 1)}[t]$ with respect to $\cD$ of the form $\cU(\A^{(n)}) = \prod_{s \in S}\fA_s^{(n - 1)}/\cD$. Let $\cU(\F^{(n)})$ be the ultra-hull of $\F^{(n)}$, i.e., $\cU(\F^{(n)})$ is the ultraproduct of $\fF_s^{(n - 1)} = \bF_s^{(n - 1)}(t)$ with respect to $\cD$ of the form $\cU(\F^{(n)}) = \prod_{s \in S}\fF_s^{(n - 1)}/\cD$.

For $\cD$-almost all $s \in S$, choose an algebraic closure $\fF_s^{(n- 1), \alg}$ and a separable closure $\fF_s^{(n-1), \sep}$ inside the algebraic closure $\fF_s^{(n-1), \alg}$. Let $\cU(\F^{(n)})^{\alg}_{\ultra}$ denote the ultra-algebraic closure of $\cU(\F^{(n)})$, i.e., 
  \begin{align*}
  	\cU(\F^{(n)})^{\alg}_{\ultra} = \prod_{s\in S}\fF_s^{(n-1), \alg}/\cD,
  \end{align*}
  and let $\cU(\F^{(n)})^{\sep}_{\ultra} = \prod_{s\in S}\fF_s^{(n), \sep}/\cD$ denote the ultra-separable closure of $\cU(\F^{(n)})$.
  
  Inside $\cU(\F^{(n)})^{\alg}_{\ultra}$, we choose an algebraic closure $\F^{(n), \alg}$ of $\F^{(n)}$ that is the set of all elements in $\cU(\F^{(n)})^{\alg}_{\ultra}$ that are algebraic over $\F^{(n)}$ (see Lang \cite{lang-algebra} or Zariski--Samuel \cite{Zariski-Samuel}). We choose a separable closure $\F^{(n), \sep}$ of $\F^{(n)}$ inside $\F^{(n), \alg}$.

  At the constant field level, since $\fK^{(n)}$, $\bF_s^{(n - 1)}$ are quasi-finite (see Corollary \ref{cor-ultra-finite-fields-are-quasi-finite}), the algebraic closures $\fK^{(n), \alg}$, $\bF_s^{(n-1), \alg}$ of $\fK^{(n)}$ and $\bF_s^{(n - 1)}$ are unique. By Proposition \ref{prop-explicit-description--for--algebraic-extension-of-degree-d-of-ultra-fields}, for every integer $d \ge 1$,
  \begin{align*}
  \fK^{(n)}(d) = \prod_{s\in S}\bF_s^{(n - 1)}(d)/\cD,
  \end{align*}
  where $\fK^{(n)}(d)$ is the unique extension of degree $d$ over $\fK^{(n)}$ and $\bF_s^{(n - 1)}(d)$ is the unique extension of degree $d$ over $\bF_s^{(n - 1)}$. Since
  \begin{align*}
  \fK^{(n), \alg} = \bigcup_{d \ge 1}\fK^{(n)}(d),
  \end{align*}
  we deduce the following.
  
\begin{lemma}
\label{lem-algebraic-closure-of-n-th-level-constant-field-is-subfield-of-ultra-algebraic-closure}
  
$\fK^{(n), \alg}$ is a subfield of the ultra-algebraic closure $\prod_{s\in S}\bF_s^{(n-1), \alg}/\cD$.

\end{lemma}

Following exactly the same arguments as in the proof of Corollary \ref{cor-F^sep-is-contained-in-U(F)^sep_ultra}, we obtain the following result by induction on $n$.

\begin{proposition}
\label{prop-F^n-sep-is-contained-in-U(F^n)^sep_ultra}

For every integer $n \ge 1$, $\F^{(n), \sep}$ is contained in $\cU(\F^{(n)})^{\sep}_{\ultra}$.

\end{proposition}

When $n = 1$, we drop the superscript $(n)$ in the notation, and simply write $\fK = \prod_{s\in S}\bF_s/\cD$, $\A = \fK[t]$, $\F = \fK(t)$, $\fA_s = \bF_s[t]$, $\fF_s = \bF_s(t)$, $\cU(\A)$, $\cU(\F)$, $\cU(\F)^{\alg}_{\ultra}$, $\cU(\F)^{\sep}_{\ultra}$, $\F^{\alg}$, $\F^{\sep}$, respectively.

\subsubsection{Cyclotomic function fields for rational function fields over $1$st level ultra-finite fields $\fK = \prod_{s\in S}\bF_s/\cD$.}
\label{subsubsec-cyclotomic-function-fields-for-1st-level-ultra-finite-fields}

We first define an analogue of cyclotomic function fields for $\F = \fK(t)$. 

Let $C^{(s)} : \fA_s = \bF_s[t] \to \fF_s\langle \tau_s\rangle$ be the Carlitz module for $\fA_s$ that sends each polynomial $a_s \in \fA_s$ to an element $C_{a_s}^{(s)}$ in $\fF_s\langle \tau_s \rangle$ as in Subsection \ref{subsec-KW-theorem-for-Fq[t]}, where $\fF_s\langle \tau_s \rangle$ denotes the ring of polynomials in $\tau_s$ with twisted multiplication, that is, $\tau a_s = a_s^{q_s}\tau_s$, and $\tau_s$ denotes the Frobenius map that sends $x$ to $x^{q_s}$.

Take an arbitrary nonzero polynomial $a \in \A = \fK[t]$. We can write $a = \ulim_{s\in S}a_s$ for some nonzero elements $a_s \in \fA_s$ (see Lemma \ref{lem-elementary-lemma}). 

Let
\begin{align}
\label{e-Lambda_s-torsion-points-of-the-s-th-Carlitz-module}
\Lambda_{C^{(s)}}[a_s] = \{\beta_s \in \fF_s^{\alg} \; | \; C_{a_s}^{(s)}(\beta_s) = 0\}
\end{align}
for $\cD$-almost all $s \in S$.

As in Definition \ref{def-cyclotomic-function-fields-p>0}, the extension $\fL_s^{(a_s)} = \fF_s\left(\Lambda_{C^{(s)}}[a_s]\right)$ of $\fF_s$, obtained from $\fF_s$ by adjoining $\Lambda_{C^{(s)}}[a_s]$, is the $a_s$-th cyclotomic function field for $\fF_s$. Following Section \ref{sec-algebraic-part-of-an-ultra-finite-extension}, we form the ultra-field extension $\fL^{(a)}$ of $\cU(\F)$ by letting
\begin{align}
\label{e-ultra-field-extension-L_a-for-a-th-CFF-in-ultra-finite-fields}
\fL^{(a)} = \prod_{s\in S}\fL_s^{(a_s)}/\cD.
\end{align}

We are now ready to define an analogue of cyclotomic function fields for $\F = \fK(t)$. 

\begin{definition}
\label{def-a-th-cyclotomic-function-field-for-ultra-finite-fields}
($a$-th cyclotomic function field for $\F = \fK(t)$)

For each nonzero polynomial $a \in \A = \fK[t]$, the algebraic part of $\fL^{(a)}$ over $\F$, i.e., $\fL^{(a)}_{\alg} = \fL^{(a)} \cap \F^{\alg}$, is called the \textbf{$a$-th cyclotomic function field for $\F$}.

\end{definition}

\begin{remark}

In Subsection \ref{subsubsec-L_alg-is-Z-hat-profinite-Galois-over-F}, we provide an example of an $a$-th cyclotomic function field for $\F$.

\end{remark}

By Theorem \ref{thm-a-cyclotomic-function-fields-structure}, the $a_s$-th cyclotomic function field $\fL_s^{(a_s)}$ is a  Galois extension of $\fF_s$ for $\cD$-almost all $s \in S$, and thus the following result follows immediately from Theorem \ref{thm-main-thm2-Galois-property-of-algebraic-parts}.

\begin{proposition}
\label{prop-cyclotomic-function-fields-are-Galois}

For each nonzero polynomial $a \in \A = \fK[t]$, the $a$-th cyclotomic function field $\fL^{(a)}_{\alg} = \fL^{(a)} \cap \F^{\alg}$ is a (possibly infinite) Galois extension of $\F$.

\end{proposition}

We will prove that $\fL^{(a)}_{\alg}$ carries many properties that are very analogous to cyclotomic fields over the rationals $\bQ$, or to cyclotomic function fields over rational function fields over finite fields that are defined in Hayes \cite{hayes}. The following is an analogue of Theorem \ref{thm-P^h-cyclotomic-function-fields-structure} for $\F = \fK(t)$.

\begin{theorem}
\label{thm-important-thm1-the-case-P^h-for-ultra-finite-fields}

Let $a = P^h$ for some positive integer $h$, where $P$ is an irreducible polynomial in $\A$. Let $\fL^{(P^h)}_{\alg}$ be the $P^h$-th cyclotomic function field for $\F$ as in Definition \ref{def-a-th-cyclotomic-function-field-for-ultra-finite-fields}. Then
\begin{itemize}

\item [(i)] $\fL^{(P^h)}_{\alg}$ is an abelian extension of $\F$. When $h = 1$, the Galois group $\Gal(\fL^{(P^h)}_{\alg}/\F)$ is a procyclic group.

\item [(ii)] $\fL^{(P^h)}_{\alg}$ is unramified at every prime ideal $\fp\A$ with $\fp\A \ne P\A$.

\item [(iii)] $P\A$ is totally ramified in $\fL^{(P^h)}_{\alg}$.

\end{itemize}

\end{theorem}

\begin{proof}

Since $\bF_s$ is a finite field for $\cD$-almost all $s \in S$, $\bF_s$ is quasi-finite, and it thus follows from Proposition \ref{prop-quasi-finite-for-ultra-fields} that $\fK = \prod_{s\in S}\bF_s/\cD$ is also quasi-finite.

By Lemma \ref{lem-elementary-lemma0}, one can write $P^h = \ulim_{s\in S}P_s^h$, where $P_s$ is an irreducible polynomial in $\fA_s = \bF_s[t]$ for $\cD$-almost all $s \in S$.

By Theorem \ref{thm-P^h-cyclotomic-function-fields-structure}, $\fL_s^{(P_s^h)}$ is an abelian extension of $\fF_s$ for $\cD$-almost all $s\in S$. Thus we deduce from Corollary \ref{cor-main-corollary-1-abelian-Galois-group-of-L-alg-and-shadows} that $\fL^{(P^h)}_{\alg}$ is an abelian extension of $\F$.

When $h= 1$,  Theorem \ref{thm-P^h-cyclotomic-function-fields-structure} implies that $\fL_s^{(P_s^h)}$ is a cyclic extension of $\fF_s$ for $D$-almost all $s\in S$. Thus we deduce from Corollary \ref{cor-main-corollary-1-procyclic-Galois-group-of-L-alg-and-shadows} that $\Gal(\fL_{\alg}^{(P)}/\F)$ is a procyclic group.

For part (ii), take an arbitrary prime ideal $\fp \A \ne P\A$ in $\A$ for some irreducible polynomial $\fp \in \A$. By Lemma \ref{lem-elementary-lemma0}, one can write $\fp = \ulim_{s\in S}\fp_s$, where $\fp_s$ is an irreducible polynomial in $\fA_s$ for $\cD$-almost all $s \in S$. Since $\fp A\ne P\A$, we see that $\fp, P$ are relatively prime in $\A$, and thus $\alpha \fp + \beta P = 1$ for some polynomials $\alpha, \beta \in \A$. By Lemma \ref{lem-elementary-lemma}, $\alpha = \ulim_{s\in S}\alpha_s$ and $\beta = \ulim_{s\in S}\beta_s$ for some polynomials $\alpha_s, \beta_s \in \fA_s$. Thus
\begin{align*}
\alpha\fp + \beta P = \ulim_{s\in S}(\alpha_s\fp_s + \beta_sP_s) = 1,
\end{align*}
and therefore 
\begin{align*}
\alpha_s\fp_s + \beta_sP_s = 1
\end{align*}
for $\cD$-almost all $s \in S$. Therefore $\fp_s, P_s$ are relatively prime in $\fA_s$, and thus $\fp_s\fA_s \ne P_s\fA_s$ for $\cD$-almost all $s \in S$. By Theorem \ref{thm-P^h-cyclotomic-function-fields-structure}, $\fp_s\fA_s$ is unramified in $\fL_s^{(P_s^h)}$ for $\cD$-almost all $s \in S$.

If $\fL^{(P^h)}_{\alg}$ is a finite extension of $\F$, using Corollary \ref{cor-unramified-primes-from-local-in-L=prod-L-s-to-L-alg-finite-case}, we deduce that $\fp = \ulim_{s\in S}\fp_s$ is unramified in $\fL_{\alg}^{(P^h)}$.

If $\fL^{(P^h)}_{\alg}$ is an infinite Galois extension of $\F$, it follows from (B1) of part (1) in Theorem \ref{thm-big-theorem-II-infinite-case} that $\fp = \ulim_{s\in S}\fp_s$ is unramified in $\fL_{\alg}^{(P^h)}$.

For part (iii), we know from part (iii) in Theorem \ref{thm-P^h-cyclotomic-function-fields-structure} that $P_s\fA_s$ is totally ramified in $\fL_s^{(P_s^h)}$ for $\cD$-almost all $s \in S$. Thus we deduce from part (iv) in Theorem \ref{thm-big-theorem-I-the-finite-case} and part (4) in Theorem \ref{thm-big-theorem-II-infinite-case} that $P\A$ is totally ramified in $\fL_{\alg}^{(P^h)}$.

\end{proof}

We prove an analogue of Theorem \ref{thm-a-cyclotomic-function-fields-structure} for $\F = \fK(t)$. 

\begin{theorem}
\label{thm-important-thm1-the-case-a-for-ultra-finite-fields}

Let $a = \alpha P_1^{h_1}\cdots P_r^{h_r}$ be a polynomial of positive degree in $\A$, where $\alpha \in \fK = \prod_{s\in S}\bF_s/\cD$, the $P_i$ are distinct monic irreducible polynomials in $\A$, and the $h_i$ are positive integers. Let $\fL_{\alg}^{(a)}$ be the $a$-th cyclotomic function field as in Definition \ref{def-a-th-cyclotomic-function-field-for-ultra-finite-fields}. Then
	\begin{itemize}
	
		\item [(i)] $\fL_{\alg}^{(a)}$ is an abelian extension of $\F$.
	
	\item [(ii)] The only ideals in $\A$ that ramify in $\fL_{\alg}^{(a)}$ are $P_1\A, P_2\A, \ldots, P_r\A$. 
	
	\end{itemize}

\end{theorem}

\begin{proof}

We begin by proving part (i). For each $1 \le i \le r$, using Lemma \ref{lem-elementary-lemma0}, one can write $P_i = \ulim_{s\in S}P_{i, s}$, where $P_{i, s}$ is an irreducible polynomial in $\fA_s = \bF_s[t]$ for $\cD$-almost all $s \in S$. Since $\alpha \in \fK = \prod_{s\in S}\bF_s/\cD$, one can write $\alpha = \ulim_{s\in S}\alpha_s$, where $\alpha_s \in \bF_s$ for $\cD$-almost all $s \in S$.

 Let 
\begin{align*}
a_s = \alpha_s P_{1, s}^{h_1}\cdots P_{r, s}^{h_r} \in \fA_s.
\end{align*}
It is trivial that $a = \ulim_{s\in S}a_s$. 

By Theorem \ref{thm-a-cyclotomic-function-fields-structure}, the $a_s$-th cyclotomic function field $\fL_s^{(a_s)}$ for $\fF_s = \bF_s(t)$ is an abelian extension for $\cD$-almost all $s \in S$. By Corollary \ref{cor-main-corollary-1-abelian-Galois-group-of-L-alg-and-shadows}, the $a$-th cyclotomic function field $\fL_{\alg}^{(a)} = \fL^{(a)} \cap \F^{\alg}$ is an abelian extension of $\F$, where 
\begin{align*}
\fL^{(a)} = \prod_{s\in S}\fL_s^{(a_s)}/\cD.
\end{align*}

For part (ii), take an arbitrary integer $1 \le i \le r$. By Theorem \ref{thm-a-cyclotomic-function-fields-structure}, $\fL_s^{(a_s)}$ is the compositum of the fields $\fL_s^{(P_{1, s}^{h_1})}, \ldots, \fL_s^{(P_{r, s}^{h_r})}$; in particular, $\fL_s^{(a_s)}$ contains $\fL_s^{(P_{i, s}^{h_i})}$ as a subfield for $\cD$-almost all $s \in S$. Since $P_i^{h_i} = \ulim_{s\in S}P_{i, s}^{h_i}$, we deduce that 
\begin{align*}
\fL^{(P_i^{h_i})} = \prod_{s\in S}\fL_s^{(P_{i, s}^{h_i})}/\cD \subset \prod_{s\in S}\fL^{(a_s)}/\cD = \fL^{(a)},
\end{align*}
and thus $\fL_{\alg}^{(P_i^{h_i})}$ is a subfield of $\fL_{\alg}^{(a)}$. By Theorem \ref{thm-important-thm1-the-case-P^h-for-ultra-finite-fields}, $P_i\A$ is totally ramified in $\fL_{\alg}^{(P_i^{h_i})}$, and thus is ramified in $\fL_{\alg}^{(a)}$. Since $i$ is an arbitrary integer such that $1 \le i \le r$, all the prime ideals $P_1\A, \ldots, P_r\A$ are ramified in $\fL_{\alg}^{(a)}$.

Now take an arbitrary ideal $\fp\A$ in $\A$ for some irreducible polynomial $\fp \in \A$ such that $\fp\A \ne P_i\A$ for all $1 \le i \le r$. By Lemma \ref{lem-elementary-lemma0}, we can write $\fp = \ulim_{s\in S}\fp_s$, where $\fp_s$ is an irreducible polynomial in $\fA_s = \bF_s[t]$ for $\cD$-almost all $s \in S$. Following the proof of part (ii) in Theorem \ref{thm-important-thm1-the-case-P^h-for-ultra-finite-fields}, we know that for each $1 \le i \le r$, 
\begin{align*}
\cA_i = \{s \in S\; | \; \fp_s\fA_s \ne P_{i, s}\fA_s\} \in \cD.
\end{align*}
Since $\cD$ is an ultrafilter, the intersection $\cA_1 \cap \cA_2 \cap \cdots \cap \cA_r$ belongs to $\cD$, and thus $\fp_s\fA_s \ne P_{i, s}\fA_s$ for all $1 \le i \le r$ for $\cD$-almost all $s \in S$. 

By part (ii) in Theorem \ref{thm-a-cyclotomic-function-fields-structure}, $\fp_s\fA_s$ is unramified in $\fL_s^{(a_s)}$ for $\cD$-almost all $s \in S$, and it thus follows from Corollary \ref{cor-unramified-primes-from-local-in-L=prod-L-s-to-L-alg-finite-case} (for the case where $\fL^{(a)}_{\alg}$ is a finite Galois extension of $\F$) and (B1) of part (1) in Theorem \ref{thm-big-theorem-II-infinite-case} (for the case where $\fL^{(a)}_{\alg}$ is an infinite Galois extension of $\F$)  that $\fp = \ulim_{s\in S}\fp_s$ is unramified in $\fL_{\alg}^{(a)}$.

\end{proof}

In complete analogy with cyclotomic function fields for rational function fields $\bF_q(t)$ over finite fields $\bF_q$ as in the works of Carlitz \cite{Carlitz-1935, Carlitz-1938}and Hayes \cite{hayes}, we prove that cyclotomic function fields for $\F = \fK(t)$ are geometric extensions of $\F$.

\begin{proposition}
\label{prop-cyclotomic-function-fields-are-geometric}

For every nonzero polynomial $a \in \A = \fK[t]$, the $a$-th cyclotomic function field $\fL_{\alg}^{(a)}$ is a geometric extension of $\F$, i.e., $\fL_{\alg}^{(a)} \cap \fK^{\alg} = \fK$, where $\fK^{\alg}$ is the algebraic closure of $\fK$.

\end{proposition}

\begin{remark}

Note that since $\fK = \prod_{s\in S}\bF_s/\cD$, where $\bF_s$ is a finite field of $q_s$ elements, $\fK$ is quasi-finite (see Proposition \ref{prop-quasi-finite-for-ultra-fields}), and thus
\begin{align*}
\fK^{\alg} = \bigcup_{d \ge 1}\fK(d),
\end{align*}
where for each integer $d \ge 1$, $\fK(d)$ denotes a unique extension of degree $d$ over $\fK$.

\end{remark}

\begin{proof}

Let $a = \ulim_{s\in S}a_s$, where $a_s \in \fA_s = \bF_s[t]$ for $\cD$-almost all $s \in S$. Since the $a_s$-th cyclotomic function field $\fL_s^{(a_s)}$ for $\fF_s = \bF_s(t)$ is a geometric extension of $\fF_s$ (see Rosen \cite{rosen}), $\fL_s^{(a_s)} \cap \bF_s^{\alg} = \bF_s$ for $\cD$-almost all $s \in S$. Thus
\begin{align*}
\fL^{(a)} \cap \prod_{s\in S}\bF_s^{\alg}/\cD &= \left(\prod_{s\in S}\fL_s^{(a_s)}/\cD \right) \cap \left(\prod_{s\in S}\bF_s^{\alg}/\cD\right) \\
&= \prod_{s\in S}\left(\fL_s^{(a_s)}\cap \bF_s^{\alg}\right)/\cD \\
&= \prod_{s\in S}\bF_s/\cD \\
&= \fK,
\end{align*}

By Lemma \ref{lem-algebraic-closure-of-n-th-level-constant-field-is-subfield-of-ultra-algebraic-closure}, we know that $\fK^{\alg}$ is a subfield of $\prod_{s\in S}\bF_s^{\alg}/\cD $, and thus it follows from Lemma \ref{lem-algebraic-part-of-ultra-algebraic-closure-of-constant-fields} that
\begin{align*}
\left(\prod_{s\in S}\bF_s^{\alg}/\cD\right) \cap \F^{\alg} = \fK^{\alg}.
\end{align*}
Since $\fK \subset \F^{\alg}$, we deduce that
\begin{align*}
\fK = \fK \cap \F^{\alg} &= \left(\fL^{(a)}\cap \prod_{s\in S}\bF_s^{\alg}/\cD\right) \cap \F^{\alg} \\
&= \left(\fL^{(a)}\cap \F^{\alg}\right) \cap \left(\prod_{s\in S}\bF_s^{\alg}/\cD \cap \F^{\alg} \right) \\
&= \fL^{(a)}_{\alg} \cap \fK^{\alg}.
\end{align*}

\end{proof}

\subsubsection{Cyclotomic function fields for rational function fields over $n$-th level ultra-finite fields $\fK^{(n)} = \prod_{s\in S}\bF_s^{(n - 1)}/\cD$.}
\label{subsubsec-Cyclotomic-function-fields-for-function-fields-over-n-th-level-ultra-finite-fields}

In this section, we define an analogue of cyclotomic function fields for rational function fields over $n$-th level ultra-finite fields $\fK^{(n)} = \prod_{s\in S}\bF_s^{(n - 1)}/\cD$ for a general positive integer $n$.

Let $n$ be a positive integer such that $n > 1$. Assume that an analogue of cyclotomic function fields for rational function fields  over $(n - 1)$-th level ultra-finite fields $\fK^{(n - 1)}$ has already been defined. More precisely, for each nonzero polynomial $a \in \A^{(n - 1)} = \fK^{(n - 1)}[t]$, assume that the $a$-th cyclotomic function field, denoted by $\fL_{\alg}^{(a), (n - 1)}$, as an algebraic extension of $\F^{(n - 1)} = \fK^{(n - 1)}(t)$, has already been defined such that the following are true:
\begin{itemize}

\item [(CFF1)] $\fL_{\alg}^{(a), (n - 1)}$ is an abelian extension of $\F^{(n - 1)}$.

\item [(CFF2)] if $a = P^h$ for some monic irreducible polynomial $P$ in $\A^{(n - 1)} = \fK^{(n - 1)}[t]$ and some positive integer $h$, then $\fL_{\alg}^{(P^h), (n - 1)}$ is unramified at every prime ideal $\fp\A^{(n - 1)} \ne P\A^{(n - 1)}$, and $P\A^{(n - 1)}$ is totally ramified in $\fL_{\alg}^{(P^h), (n - 1)}$.

\item [(CFF3)] if $a = \alpha P_1^{n_1}\cdots P_r^{n_r}$ is the prime factorization of $a$ into distinct prime elements $P_1, \ldots, P_r$ in $\A^{(n - 1)}$ for some element $\alpha \in \fK^{(n - 1)}$, then the only prime ideals in $\A^{(n - 1)} = \fK^{(n - 1)}[t]$ that ramify in $\fL_{\alg}^{(a), (n - 1)}$ are prime ideals $P_1\A^{(n - 1)}, \ldots, P_r\A^{(n - 1)}$.

\end{itemize}

Now take an arbitrary nonzero polynomial $a \in \A^{(n)} = \fK^{(n)}[t]$. Since $\fK^{(n)} = \prod_{s\in S}\bF_s^{(n - 1)}/\cD$, using Lemma \ref{lem-elementary-lemma0}, one can write $a = \ulim_{s\in S}a_s$, where $a_s \in \fA_s^{(n - 1)} = \bF_s^{(n - 1)}[t]$ for $\cD$-almost all $s \in S$. Set
\begin{align*}
\fL^{(a), (n)} = \prod_{s\in S}\fL_s^{(a_s), (n - 1)}/\cD,
\end{align*}
where $\fL_s^{(a_s), (n - 1)}$ is the $a_s$-th cyclotomic function field for the rational function field $\fF_s^{(n - 1)} = \bF_s^{(n - 1)}(t)$ over the $(n - 1)$-th level ultra-finite field $\bF_s^{(n - 1)}$ for $\cD$-almost all $s\in S$ that is defined in (CFF1), (CFF2), (CFF3).

\begin{definition}
\label{def-a-th-cyclotomic-function-field-for-n-th-level-ultra-finite-fields}
($a$-th cyclotomic function fields for $\F^{(n)} = \fK^{(n)}(t)$)

For each nonzero polynomial $a \in \A^{(n)} = \fK^{(n)}[t]$, the algebraic part of $\fL^{(a), (n)}$ over $\F^{(n)}$, i.e., $\fL^{(a), (n)}_{\alg} = \fL^{(a), (n)} \cap \F^{(n), \alg}$, is called the $a$-th cyclotomic function field for $\F^{(n)}$.

\end{definition}

Using induction over the level $n$, and using Theorems \ref{thm-P^h-cyclotomic-function-fields-structure}, \ref{thm-a-cyclotomic-function-fields-structure}, \ref{thm-important-thm1-the-case-P^h-for-ultra-finite-fields}, and \ref{thm-important-thm1-the-case-a-for-ultra-finite-fields}, we obtain the following.

\begin{theorem}
\label{thm-greatest-thm-cyclotomic-function-fields-for-n-th-level-ultra-finite-fields}

Let $n$ be a positive integer. For any nonzero polynomial $a \in \A^{(n)} = \fK^{(n)}[t]$, the $a$-th cyclotomic function field $\fL_{\alg}^{(a), (n)}$ defined in Definition \ref{def-a-th-cyclotomic-function-field-for-n-th-level-ultra-finite-fields} satisfies the following.
\begin{itemize}

\item [(i)] $\fL_{\alg}^{(a), (n)}$ is an abelian extension of $\F^{(n)}$.

\item [(ii)] if $a = P^h$ for some monic irreducible polynomial $P$ in $\A^{(n)} = \fK^{(n)}[t]$ and some positive integer $h$, then $\fL_{\alg}^{(P^h), (n)}$ is unramified at every prime ideal $\fp\A^{(n)} \ne P\A^{(n)}$, and $P\A^{(n)}$ is totally ramified in $\fL_{\alg}^{(P^h), (n)}$.

\item [(iii)] if $a = \alpha P_1^{n_1}\cdots P_r^{n_r}$ is the prime factorization of $a$ into distinct prime elements $P_1, \ldots, P_r$ in $\A^{(n)}$ for some element $\alpha \in \fK^{(n)}$, then the only prime ideals in $\A^{(n)} = \fK^{(n)}[t]$ that ramify in $\fL_{\alg}^{(a), (n)}$ are $P_1\A^{(n)}, \ldots, P_r\A^{(n)}$.

\end{itemize}

\end{theorem}

\begin{proof}

When $n = 1$, the above theorem follows immediately from Theorems \ref{thm-important-thm1-the-case-P^h-for-ultra-finite-fields} and \ref{thm-important-thm1-the-case-a-for-ultra-finite-fields}.

Assume the theorem holds for $n - 1$ for some integer $n > 1$, i.e., (CFF1), (CFF2), (CFF3) hold for $n - 1$. Using the same arguments as in the proofs of Theorems \ref{thm-important-thm1-the-case-P^h-for-ultra-finite-fields} and \ref{thm-important-thm1-the-case-a-for-ultra-finite-fields}, the theorem also holds for $n$, and thus the theorem follows immediately from the induction.

\end{proof}

Using the same arguments as in the proof of Proposition \ref{prop-cyclotomic-function-fields-are-geometric}, we also obtain the following.

\begin{proposition}

For all integers $n \ge 1$ and all nonzero polynomials $a \in \A^{(n)} = \fK^{(n)}[t]$, the $a$-th cyclotomic function field $\fL_{\alg}^{(a), (n)}$ for $\F^{(n)} = \fK^{(n)}(t)$ are geometric extensions of $\F^{(n)}$, i.e., 
\begin{align*}
\fL_{\alg}^{(a), (n)} \cap \fK^{(n), \alg} = \fK^{(n)}.
\end{align*}

\end{proposition}

\subsection{An analogue of the Kronecker--Weber theorem for rational function fields over $n$-level ultra-finite fields}
\label{subsec-KW-theorem}

In this subsection, we prove an analogue of the Kronecker--Weber theorem for rational function fields over $n$-level ultra-finite fields. We begin by proving such an analogue for $\F = \fK(t)$, where $\fK = \prod_{s\in S}\bF_s/\cD$ is a $1$st level ultra-finite field.

\subsubsection{The maximal abelian extension $\cA$ over the rational function field $L = \bF_q(t)$ over a finite field $\bF_q$.}
\label{subsubssec-MAE-for-Fq(t)}

We maintain the same notation as in Subsection \ref{subsec-KW-theorem-for-Fq[t]}. We recall the construction of the maximal abelian extension of $L = \bF_q(t)$ from Hayes \cite{hayes}. We construct three pairwise linearly disjoint extensions $\cC/L$, $\cQ_t/L$, and $\cR_{\infty}/L$ as follows.
\begin{itemize}

\item [(MAE1)] $\cC/L$ is the union of all the constant field extensions of $L$; more precisely, 
\begin{align*}
\cC = \bigcup_{n \ge 1}\bF_{q^n}L = \bigcup_{n \ge 1}\bF_{q^n}(t).
\end{align*}

\item [(MAE2)] $\cQ_t/L$ is the union of all the $a$-th cyclotomic function fields $L(\Lambda_C[a])$ for all polynomials $a \in A = \bF_q[t]$.

\item [(MAE3)] Replacing the generator $t$ by $1/t$, and defining the Carlitz module $C_{1/t}$ for $A_{1/t} = \bF_q[1/t]$ instead of $A = \bF_q[t]$ as in Subsection \ref{subsec-KW-theorem-for-Fq[t]}, i.e., $C^{1/t} : A_{1/t} \to L\langle \tau\rangle$, $a \mapsto C^{1/t}_a$ is the $\bF_q$-algebra homomorphism such that $C^{1/t}_{1/t} = 1/t + \tau$, we construct the $t^{-n - 1}$-th cyclotomic function field $\cF_n = L(\Lambda_{C^{1/t}}[t^{-n - 1}])$ for $n = 1, 2,\ldots$, where
    \begin{align*}
    \Lambda_{C^{1/t}}[t^{-n - 1}] = \{\beta \in L^{\alg} \; | \; C^{1/t}_{t^{-n - 1}}(\beta) = 0\}.
    \end{align*}
    
    Let $\lambda_{-n - 1}$ be a generator of the $A_{1/t}$-module $\Lambda_{C^{1/t}}[t^{-n - 1}]$. One can identify $\bF_q^{\times}$ with the group of automorphisms $\sigma_{\beta}$ for $\beta \in \bF_q^{\times}$ that map $\lambda_{-n - 1}$ to $\beta\lambda_{-n - 1}$. Let $\cR_n$ be the fixed field of $\bF_q^{\times}$ in $\cF_n$. It is known (see Hayes \cite[Section 5, p.86]{hayes}) that the extension $\cR_n/L$ is Galois of degree $q^n$ and that $\cR_n \subset \cR_{n + 1}$. Set
    \begin{align*}
    \cR_{\infty} = \bigcup_{n = 1}^{\infty}\cR_n.
    \end{align*}

\end{itemize}

\begin{definition}
\label{def-maximal-abelian-extension-of-Fq(t)}

Let $\cA = \cC\cdot \cQ_t\cdot \cR_{\infty}$ be the compositum of $\cC$, $\cQ_t$, and $\cR_{\infty}$.

\end{definition}

\begin{theorem}
\label{thm-Drinfeld-Hayes-thm-maximal-abelian-extension-version1}
$(\text{See Drinfleld \cite{drinfeld1, drinfeld2} and Hayes \cite[Theorem 7.1]{hayes}})$

The extension $\cA/L$ constructed in Definition \ref{def-maximal-abelian-extension-of-Fq(t)} is the maximal abelian extension of $L = \bF_q(t)$.

\end{theorem}

\begin{remark}

It is known from Hayes \cite[Proposition 5.2]{hayes} that $\cC, \cQ_t, \cR_{\infty}$ are pairwise linearly disjoint over $L$.

\end{remark}

There is another characterization of the maximal abelian extension $\cA$ of $L$. Let $\cQ_{1/t}$ be the union of all the $a$-th cyclotomic function fields $L(\Lambda_{C^{1/t}}[a])$ for all polynomials $a \in A_{1/t} = \bF_q[1/t]$. 

\begin{theorem}
\label{thm-Drinfeld-Hayes-thm-for-maximal-abelian-extension-version2}
$(\text{See Hayes \cite[Theorem 7.2]{hayes}})$

The maximal abelian extension $\cA$ of $L$ is the compositum $\cQ_t\cdot \cQ_{1/t}$.

\end{theorem}

\subsubsection{The maximal abelian extension of the rational function field over a ($1$st level) ultra-finite field.}
\label{subsec-MAE-for-1st-ultra-finite-fields}

We maintain the same notation as in the beginning of Section \ref{subsec-cyclotomic-function-fields-over-nth-level-ultra-finite-fields} and  Subsection \ref{subsubsec-cyclotomic-function-fields-for-1st-level-ultra-finite-fields}. We recall that $\fK = \prod_{s\in S}\bF_s/\cD$ is a ($1$st level) ultra-finite field, where $\bF_s$ is a finite field of $q_s$ elements and $q_s$ is a power of a prime $p_s > 0$. 

For $\cD$-almost all $s \in S$, let $\cA_s$ be the maximal abelian extension of $\fF_s = \bF_s(t)$ that is defined as in Definition \ref{def-maximal-abelian-extension-of-Fq(t)} and Theorem \ref{thm-Drinfeld-Hayes-thm-maximal-abelian-extension-version1} with the rational function field $\fF_s = \bF_s(t)$ in place of the function field $L = \bF_q(t)$.
  
  Set
\begin{align}
\label{e-def-of-the-ultra-maximal-abelian-extension-of-F-1st-level-ultra-finite-field}
\cA^{(1)} = \prod_{s\in S}\cA_s/\cD.
\end{align} 

Note that the superscript $(1)$ of $\cA^{(1)}$ in the above equation indicates the level of the ultra-finite field $\fK$.

Using Corollary \ref{cor-relation-between-maximal-abelian-extensions-of-F_s-and-F}, we immediately obtain a characterization of the maximal abelian extension of $\F = \fK(t)$.

\begin{theorem}
\label{thm-greatest-thm-an-analogue-of-KW-for-F=fK(t)}
(Analogue of the Kronecker--Weber theorem for $\F = \fK(t)$)

The maximal abelian extension of $\F = \fK(t)$ is the algebraic part of $\cA^{(1)}$ over $\F$, i.e., $\cA^{(1)}_{\alg} = \cA^{(1)}\cap \F^{\alg}$.

\end{theorem}

\begin{remark}

For $\cD$-almost all $s \in S$, let $\cQ^{(s)}_t$ be the union of all the $a_s$-th cyclotomic function fields $\fF_s(\Lambda_{C^{(s)}}[a_s])$ for all polynomials $a_s \in \fA_s = \bF_s[t]$, where $C^{(s)}$ is the Carlitz module for $\fA_s = \bF_s[t]$ (see Subsection \ref{subsubsec-cyclotomic-function-fields-for-1st-level-ultra-finite-fields}). Let $\cQ^{(s)}_{1/t}$ be the union of all the $a_s$-th cyclotomic function fields $\fF_s(\Lambda_{C^{(s)}_{1/t}}[a_s])$ for all polynomials $a_s \in \fA_{s, 1/t} = \bF_s[1/t]$, where $C^{(s)}_{1/t}$ is the Carlitz module for $\fA_{s, 1/t} = \bF_s[1/t]$ that is defined in the same way as in Subsection \ref{subsubssec-MAE-for-Fq(t)} with $\fF_s = \bF_s(t)$ in place of $L = \bF_q(t)$.

By Theorem \ref{thm-Drinfeld-Hayes-thm-for-maximal-abelian-extension-version2}, the maximal abelian extension $\cA_s$ of $\fF_s$ is the compositum $\cQ^{(s)}_t\cdot \cQ^{(s)}_{1/t}$. Thus one obtains another characterization of the maximal abelian extension of $\F = \fK(t)$ as follows.

\begin{corollary}
\label{cor-2nd-characterization-of-MAE-of-F-1st-level-ultra-finite-field}

The algebraic part of the ultraproduct $\prod_{s\in S}\cQ^{(s)}_t\cdot \cQ^{(s)}_{1/t}/\cD$ over $\F$ is the maximal abelian extension of $\F$.

\end{corollary}

\end{remark}

\subsubsection{The maximal abelian extension of the rational function field over an $n$-th level ultra-finite field $\fK^{(n)} = \prod_{s\in S}\bF_s^{(n - 1)}/\cD$.}
\label{subsec-MAE-for-n-th-ultra-finite-fields}

In this section, we characterize the maximal abelian extension of the rational function field $\F^{(n)} = \fK^{(n)}(t)$ over an $n$-th level ultra-finite field $\fK^{(n)} = \prod_{s\in S}\bF_s^{(n - 1)}/\cD$ inductively on $n \ge 1$. We maintain the same notation as in Subsection \ref{subsubsec-Cyclotomic-function-fields-for-function-fields-over-n-th-level-ultra-finite-fields}.

Let $n$ be a positive integer such that $n \ge 1$. When $n = 1$, we know from Theorem \ref{thm-greatest-thm-an-analogue-of-KW-for-F=fK(t)} that $\cA^{(1)}_{\alg}$ is the maximal abelian extension of $\F = \fK(t)$. 

Assume that the maximal abelian extension, denoted by $\cA^{(n - 1)}_{\alg}$, of the rational function field $\F^{(n - 1)} = \fK^{(n - 1)}(t)$ over an arbitrary $(n - 1)$-th level ultra-finite field $\fK^{(n - 1)}$ has already been known for some integer $n > 1$. 

We consider the rational function field $\F^{(n)} = \fK^{(n)}(t)$, where $\fK^{(n)} = \prod_{s\in S}\bF_s^{(n -1)}/\cD$ is an $n$-th level ultra-finite field, and $\bF_s^{(n - 1)}$ is an $(n - 1)$-th level ultra-finite field for $\cD$-almost all $s \in S$. 

By induction hypothesis, let $\cA^{(n - 1)}_s$ denote the maximal abelian extension of $\fF_s^{(n - 1)} = \bF_s^{(n - 1)}(t)$ for $\cD$-almost all $s \in S$, and set
\begin{align}
\label{e-the-ultra-maximal-abelian-extension-of-function-field-over-nth-level-ultra-finite-field}
\cA^{(n)} = \prod_{s\in S}\cA^{(n - 1)}_s/\cD.
\end{align}

Using Corollary \ref{cor-relation-between-maximal-abelian-extensions-of-F_s-and-F}, we obtain an analogue of the Kronecker--Weber theorem for the rational function field $\F^{(n)} = \fK^{(n)}(t)$.

\begin{theorem}
\label{thm-the-greatest-thm-KW-thm-for-function-fields-over-nth-level-ultra-finite-fields}
(An analogue of the Kronecker--Weber theorem for the rational function field over an arbitrary $n$-th level ultra-finite field)

The algebraic part of $\cA^{(n)}$ over $\F^{(n)}$, i.e., $\cA^{(n)}_{\alg} = \cA^{(n)}\cap \F^{(n),\alg}$, is the maximal abelian extension of $\F^{(n)} = \fK^{(n)}(t)$, where $\cA^{(n)}$ is defined by (\ref{e-the-ultra-maximal-abelian-extension-of-function-field-over-nth-level-ultra-finite-field}).

\end{theorem}

By Theorems \ref{thm-greatest-thm-an-analogue-of-KW-for-F=fK(t)} and \ref{thm-the-greatest-thm-KW-thm-for-function-fields-over-nth-level-ultra-finite-fields}, we have explicitly described the maximal abelian extension of the rational function field over an arbitrary $n$-th level ultra-finite field for all $n \ge 1$, which can be viewed as an analogue of the Kronecker--Weber theorem for the rational function field over an arbitrary $n$-th level ultra-finite field for all $n \ge 1$.

\section*{Acknowledgements}

I am very grateful to David Harbater for his generosity with his time in answering many questions of mine in Galois theory, and for explaining his work on Abhyankar's conjecture and on the Inverse Galois Problem to me on numerous occasions. I would also like to thank him for pointing out to me many relevant references during the preparation of this paper.

\end{document}